\documentclass[UTF8,fontset=windows,final, leqno]{siamltex}
\usepackage{ctex}
\usepackage{amsmath}
\usepackage{epsfig}
\usepackage{graphicx}
\usepackage{amssymb}
\usepackage{CJK}
\usepackage{color}
\usepackage{subfigure}
\usepackage{caption}
\usepackage{float}
\usepackage{latexsym, bm}
\usepackage{hyperref}
\usepackage[noadjust]{cite}
\usepackage{comment}

\newtheorem{remark}{Remark}[section]

\allowdisplaybreaks[4]

\normalsize

\def\nab{{\nabla}}

\def\reff#1{\eqref{#1}}

\def\no{{\nonumber}}

\def\div{{\mbox{\rm div\,}}}

\def\nab{\nabla}

\renewcommand\figurename{Figure}
\renewcommand\tablename{Table}

\renewcommand\abstractname{Abstract}

\renewcommand\figurename{Figure}
\renewcommand\tablename{Table}
\renewcommand\refname{References}

\CTEXoptions[today=old]

\def\be{\mathbf{e}}

\def\bu{\mathbf{u}}

\def\bv{\mathbf{v}}

\def\bV{\mathbf{V}}

\def\bV{\mathbf{V}}

\begin{document}
	
	
	\title{Analysis of Multiphysics Finite Element Method for  quasi-static  Thermo-Poroelasticity with a nonlinear convective transport term\footnote{Last update: \today}}

	\author{
				Zhihao Ge\thanks{School of Mathematics and Statistics, Henan University, Kaifeng, 475004, P.R. China ({\tt Email: zhihaoge@henu.edu.cn}).
					The work was supported by the National Natural Science Foundation of China under grant No. 12371393 and  11971150. }
		\and
				Dandan Xu\thanks{School of Mathematics and Statistics, Henan University, Kaifeng, 475004, P.R. China.}
		%
	}
	
	\maketitle
	
	
	\setcounter{page}{1}
	
	
	\begin{abstract}
		In this paper, we propose a multiphysics finite element method for a quasi-static  thermo-poroelasticity model with a nonlinear convective transport term. To design some stable numerical methods and reveal the multi-physical processes of deformation, diffusion and heat, we introduce three new variables to reformulate the original model into a fluid coupled problem. Then, we introduce an Newton's iterative algorithm by replacing the convective transport term with $\nabla T^{i}\cdot(\bm{K}\nabla p^{i-1})$, $\nabla T^{i-1}\cdot(\bm{K}\nabla p^{i})$ and $\nabla T^{i-1}\cdot(\bm{K}\nabla p^{i-1})$, and apply the Banach fixed point theorem to prove the convergence of the proposed method. Then, we propose a multiphysics finite element method with Newton's iterative algorithm, which is equivalent to a stabilized method, can effectively overcome the numerical oscillation caused by the nonlinear thermal convection term. Also, we prove that the fully discrete multiphysics finite element method has an optimal convergence order. Finally, we draw conclusions to summarize the main results of this paper.
		
	\end{abstract}

	\section{Introduction}\label{sec-1}
		Thermo-poroelasticity model is a fluid-solid-heat interaction system at pore scale, which can be regarded as an extension of the porous elasticity model  in non isothermal states \cite{coussy2004}, and it has important applications in many fields such as modeling and optimizing control of carbon dioxide storage, reservoir engineering,  biomechanics and so on.  In the process of carbon dioxide storage, carbon dioxide is affected by many factors such as permeability, deformation displacement, temperature, and pressure, one can refer to \cite{Bachu2007,Teng1979,Ju2021,Sasaki2008}. In reservoir engineering, oil recovery is enhanced by capturing carbon dioxide from the atmosphere for oil displacement, while reducing the atmospheric carbon dioxide content \cite{coussy2004,Mortezaei2017}.  Thermo-poroelasticity model is used to simulate geothermal extraction and utilization, frozen soil dynamics \cite{Selvadurai2016,Duijn2020}, etc. In biomechanics, it can simulate the mechanism of tumor growth, the distribution of brain pressure after external force damage, and provide assistance for auxiliary diagnosis and treatment\cite{Selvadurai2016,Andreozzi2019,Ge2022}. The poroelastic parameters derived under isothermal conditions initially derived by Biot\cite{Biot1941}, followed by Rice and Cleary \cite{Rice1976}, and Zimmerman et al.\cite{Zimmerman1986}, which have been extended to account for temperature effects on the pore fluid and the matrix\cite{McTigue1986}. Theoretical derivations and experiments have shown that undrained thermal loadings in low-permeability materials, such as shales or cement pastes, not only result in strain variation, but also lead to pressure variation\cite{Ghabezloo2008}. 
		The authors of \cite{BrunBerre2018} derive a nonlinear  thermo-poroelasticity model by the conservation of energy equation, which is coupled to momentum and mass equations, the governing equations are then given by 
	\begin{align}
		\partial_{t}(a_{0}T-b_{0}p+\beta\nabla\cdot\textbf{u})-\nabla T\cdot(\bm{K}\nabla p)-\nabla\cdot(\bm{\Theta}\nabla T)=\phi~~~~&{\rm in}~\Omega_{\tau}=\Omega\times(0, \tau),\label{eq-2-1}\\
		-(\lambda+\mu)\nabla(\nabla\cdot \textbf{u})-\mu\nabla^2\textbf{u}+\alpha\nabla p+\beta\nabla T=\textbf{f}~~~~&{\rm in}~\Omega_{\tau},\label{eq-2-2}\\
		\partial_{t}(c_{0}p-b_{0}T+\alpha\nabla\cdot\textbf{u})-\nabla\cdot(\bm{K}\nabla p)=g~~~~&{\rm in}~\Omega_{\tau},\label{eq-2-3}
	\end{align}	
	where $\Omega\subset \mathbb{R}^{d}(d=2,3)$ is a bounded polygonal domain with the boundary $\partial\Omega$, $\partial_{t}$ is the derivative with respect to time, $a_{0}$ is the effective thermal capacity, $b_{0}$ is the thermal dilation coefficient, $\beta$ is thermal stress coefficient, $\bm{K}=(K_{ij})_{i,j=1}^{d}$ is the permeability tensor, $\bm{\Theta}=(\Theta_{ij})_{i,j=1}^{d}$ is the effective thermal conductivity, $\mu$ and $\lambda$ are the Lam\'{e} parameters, $\alpha$ is the Biot-Willis constant and $c_{0}$ is the specific storage coefficient. The primary variables are the temperature distribution $T$, displacement $\textbf{u}$ and fluid pressure $p$. The source terms $f$, $\phi$, $g$ are given functions.
	
	Note that the problem \reff{eq-2-1}-\reff{eq-2-3} includes a nonlinear convective transport term of $\nabla T\cdot(\bm{K}\nabla p)$. The presence of this nonlinear coupling term strongly complicates the problem compared to the linear case\cite{Chen2022}. As for the PDE analysis and numerical methods for the problem \reff{eq-2-1}-\reff{eq-2-3}, the authors of \cite{Brun2019} investigate, in the context of mixed formulations, the existence and uniqueness of a weak solution to this model problem. A monolithic and splitting-based iterative procedures for the coupled nonlinear thermo-poroelasticity model problem can be find in \cite{Brun2020}. The simulation of \reff{eq-2-1}-\reff{eq-2-3} has the following two difficulties: how to deal with the nonlinear term in PDE analysis and numerical analysis; numerical oscillation phenomenon of pressure and temperature, locking phenomenon for displacement. In this paper, we borrow the idea of \cite{Feng2018} to reformulate the problem \reff{eq-2-1}-\reff{eq-2-3} into a fluid coupled system to reveal the underlying multiphysics processes in the original model and propose a stable finite element method base on the multiphysics model. To prove the well-posedness of the reformulated nonlinear thermo-poroelasticity model, we firstly analyze a linearized version of the reformulated model, i.e., replace the convective transport term with  $\bm{L_1}\cdot\nabla T $, $\bm{L_2}\cdot\nabla p $ and $\bm{L_3}$ for some given $\bm{L_1}, \bm{L_2}, \bm{L_3} \in  L^{\infty}$. Then, we introduce an Newton's iterative algorithm by replacing the convective transport term with $\nabla T^{i}\cdot(\bm{K}\nabla p^{i-1})$, $\nabla T^{i-1}\cdot(\bm{K}\nabla p^{i})$ and $\nabla T^{i-1}\cdot(\bm{K}\nabla p^{i-1})$, where $i\geq1$ is the iteration index. Finally, we use the Banach fixed point theorem to prove the convergence of Newton's iterative algorithm. As for the numerical methods, we propose a time-stepping algorithm--multiphysics finite element method with Newton's iterative method. Also, we prove that the proposed method has an optimal convergence order. In a word, this paper has three main innovations: in PDE analysis, we introduce an Newton's iterative algorithm to replace the convective transport term with $\nabla T^{i}\cdot(\bm{K}\nabla p^{i-1})$, $\nabla  T^{i-1}\cdot(\bm{K}\nabla p^{i})$ and $\nabla T^{i-1}\cdot(\bm{K}\nabla p^{i-1})$, and apply the Banach fixed point theorem to prove the convergence of the proposed method; we propose a multiphysics finite element method with Newton's iterative algorithm, which is equivalent to a stabilized method, can effectively overcome the numerical oscillation caused by the nonlinear thermal convection term; 
	 we introduce three new variables to not only overcome the pressure and temperature oscillations and the "locking" of the displacement $\textbf{u}$ when $\lambda\rightarrow\infty$, but also clearly reveal the underlying multi-physical processes of temperature, deformation and pressure in the original model.
	
	The remaining parts of this paper is organized as follows. In Section \ref{sec-2}, we introduce the thermo-poroelasticity model and give the PDE analysis. In Section \ref{sec-3}, we propose a multiphysics finite element method with Newton's iterative method and prove the optimal order error estimates. In Section \ref{sec-4}, we show some several numerical experiments to verify the theoretical results. Finally, we draw conclusions to summarize the main results and of this paper.
	
	\section{PDE analysis}\label{sec-2}
	\subsection{Multiphysics reformulation}		
	To close the problem (\ref{eq-2-1})-(\ref{eq-2-3}), we prescribe the following boundary conditions:	
	\begin{eqnarray}
		&&\bm{\Theta}\nabla T\cdot\textbf{n}=\phi_{1}~~~~{\rm on}~\partial\Omega_{\tau}:=\partial\Omega\times(0,\tau),\label{eq-2-4}\\
		&&\sigma(\textbf{u})\textbf{n}-(\alpha p+\beta T)\textbf{n}=\bm{{\rm f}}_{1}~~~~{\rm on}~\partial\Omega_{\tau},\label{eq-2-5}\\
		&&\bm{K}\nabla p\cdot\textbf{n}=g_{1}~~~~{\rm on}~\partial\Omega_{\tau},\label{eq-2-6}
	\end{eqnarray}
	where $
\sigma(\textbf{u}):=\mu \varepsilon(\textbf{u})+\lambda \operatorname{div} \textbf{u} I, \quad \varepsilon(\textbf{u}):=\frac{1}{2}(\nabla \bm{u}+\nabla \textbf{u}^{t})$.

	Next, we introduce a new variable $q=\nabla\cdot \textbf{u}$, and denote  
	\begin{eqnarray*}
		\gamma:=a_{0}T-b_{0}p+\beta q,~ \xi:=\alpha p+\beta T-\lambda q, ~\eta:=c_{0}p-b_{0}T+\alpha q.
	\end{eqnarray*}
	It is easy to check that
	\begin{eqnarray}\label{eq-2-7}
		T=k_{1}\xi+k_{2}\eta+k_{3}\gamma,~~~~
		p=k_{4}\xi+k_{5}\eta+k_{2}\gamma,~~~~
		q=-k_{6}\xi+k_{4}\eta+k_{1}\gamma,
	\end{eqnarray}
where
\begin{eqnarray*}
	k_{1}=\frac{\alpha\beta c_{0}+\alpha^{2}b_{0}}{\mathcal{M}},~~
	k_{2}=\frac{\alpha b_{0} \lambda-\alpha^{2}\beta}{\mathcal{M}},~~
	k_{3}=\frac{\alpha^{3}+\alpha c_{0}\lambda}{\mathcal{M}},\\
	k_{4}=\frac{a_{0}\alpha^{2}+\alpha\beta b_{0}}{\mathcal{M}},~~
	k_{5}=\frac{a_{0}\alpha\lambda+\alpha\beta^{2}}{\mathcal{M}},~~
	k_{6}=\frac{\alpha c_{0}a_{0}-\alpha b^{2}_{0}}{\mathcal{M}},\\
	\mathcal{M}=\alpha c_{0}\beta^{2}+2\alpha^{2}\beta b_{0}+a_{0}\alpha^{3}+(c_{0}a_{0}\alpha-b^{2}_{0}\alpha)\lambda.
\end{eqnarray*}
Thus, using the above notations, we can reformulate the problem \reff{eq-2-1}-\reff{eq-2-3} into a fluid-fluid-fluid coupled problem: find $(\textbf{u}, \xi, \eta, \gamma)$ satisfying
	\begin{eqnarray}
		&&-\mu \operatorname{div}(\varepsilon(\textbf{u}))+\nabla\xi=\textbf{f}~~~~{\rm  in}~\Omega_{\tau},\label{eq-2-8}\\
		&&\nabla\cdot \textbf{u}+k_{6}\xi=k_{4}\eta+k_{1}\gamma~~~~{\rm  in}~\Omega_{\tau},\label{eq-2-9}\\
		&&\eta_{t}-\nabla\cdot(\bm{K}\nabla(k_{4}\xi+k_{5}\eta+k_{2}\gamma))=g~~~~{\rm  in}~\Omega_{\tau},\label{eq-2-10}\\
		&&\gamma_{t}-\nabla T\cdot(\bm{K}\nabla (k_{4}\xi+k_{5}\eta+k_{2}\gamma) )-\nabla\cdot(\bm{\Theta}\nabla(k_{1}\xi+k_{2}\eta+k_{3}\gamma))=\phi~~~~{\rm  in}~\Omega_{\tau},\label{eq-2-11}\\
		&&\sigma(\textbf{u})\textbf{n}-(\alpha (k_{4}\xi+k_{5}\eta+k_{2}\gamma)+\beta (k_{1}\xi+k_{2}\eta+k_{3}\gamma))\textbf{n}=\textbf{f}_{1}~~~~{\rm  on}~\partial\Omega_{\tau},\label{eq-2-12}\\
		&&\bm{K}\nabla (k_{4}\xi+k_{5}\eta+k_{2}\gamma)\cdot\textbf{n}=g_{1}~~~~{\rm  on}~\partial\Omega_{\tau},\label{eq-2-13}\\
		&&\bm{\Theta}\nabla (k_{1}\xi+k_{2}\eta+k_{3}\gamma)\cdot\textbf{n}=\phi_{1}~~~~{\rm  on}~\partial\Omega_{\tau}:=\partial\Omega\times(0,\tau),\label{eq-2-14}\\
		&&\textbf{u}(\cdot,0)=\textbf{u}_{0},~~\gamma(\cdot,0)=\gamma_{0},~~\eta=\eta_{0}~~~~{\rm  in} ~\Omega\times \{t=0\},\label{eq-2-15}
	\end{eqnarray}
	where $p, T$ and $q$ are related to $\xi, \eta$ and $\gamma$ through the algebraic equations in ({\ref{eq-2-7}}), for the sake of notation brevity later, we use $\mu$ instead of $2\mu$ in (\ref{eq-2-8}).
	
	Throughout the paper, we assume that the following conditions hold:
	
	A1: $\bm{K}: \mathbb{R}^{d} \rightarrow \mathbb{R}^{d \times d}$ is symmetric and uniformly positive definite in the sense that there exist positive constants $k_{m}>0$ and $k_{M}>0$ such that $k_{m}|\zeta|^{2} \leq \zeta^{\top} \bm{K}(x) \zeta \leq k_{M}|\zeta|^2, \forall \zeta \in \mathbb{R}^{d} \backslash\{0\}$.
	
	A2: $\boldsymbol{\bm{\Theta}}: \mathbb{R}^{d} \rightarrow \mathbb{R}^{d \times d}$ is symmetric and uniformly positive definite in the sense that there exist positive constants $\theta_{m}>0$ and $\theta_{M}>0$ such that $\theta_{m}|\zeta|^{2} \leq \zeta^{\top} \boldsymbol{\bm{\Theta}}(x) \zeta \leq \theta_{M}|\zeta|^2, \forall \zeta \in \mathbb{R}^{d} \backslash\{0\}$.
	
	A3: The coefficients $a_{0}, b_{0}, c_{0},\alpha$ and $\beta$ are nonnegative constants, and $c_{0}-b_{0}>0$, $a_{0}-b_{0}>0$.
	
	A4: The source terms $g,\phi\in L^2(0,\tau;L^2(\Omega))$, and $\textbf{f}\in H^1(0,\tau;L^2(\Omega))$.
	For $1\leq p<\infty$, let $L^{p}(\Omega)=\{u:\Omega\rightarrow \mathbb{R}:\int_{\Omega}|u|^{p}dx<\infty\}$, with the associate norm$\|\cdot\|_{L^{p}(\Omega)}$. In particular, $L^{2}(\Omega)$ is the Hilbert space of square integrable functions defined on $\Omega$, endowed with the inner product$(\cdot,\cdot)$. For $p=\infty, L^{\infty}(\Omega)$ is the space of uniformly bounded measurable function defined on $\Omega$, i.e. $L^{\infty}(\Omega)=\{u:\Omega\rightarrow\mathbb{R}: {\rm ess\sup}_{x\in\Omega}|u|\leq\infty\},$ endowed with the norm $\|u\|_{\infty}=\inf\{C:|u|\leq C ~a.e.~{\rm in }~\Omega\}$. For any Banach space $\textbf{B}$, we let $\textbf{B}=[B]^{d}$ and use $\textbf{B}'$ to denote its dual space. In particular, 
	and $\|\cdot\|_{L^{p}(B)}$ is a shorthand notation for $\|\cdot\|_{L^{p}((0,\tau);B)}$. Also, we introduce the following notations:
\begin{align}
	&\textbf{RM}:=\{\bm{r}=\bm{a}+\bm{b}\times x;~\bm{a},\bm{b}, x \in \mathbb{R}^{3}; \bm{r}=\bm{a}+b\left(x_{2},-x_{1}\right)^{t}, \bm{a} \in \mathbb{R}^{2}, b \in \mathbb{R}\},\\
	&L^{2}_{0}(\Omega):=\{q\in L^{2}(\Omega);~(q,1)=0\}.
\end{align}
	\textbf{RM}	denotes the space of infinitesimal rigid motions. Let $\textbf{L}_{\bot}^{2}(\partial\Omega)$ and $\textbf{H}_{\bot}^{1}(\Omega)$ denote, respectively, the subspace of $\textbf{L}^{2}(\partial\Omega)$ and $\textbf{H}^{1}(\Omega)$, which are orthogonal to \textbf{RM}, that is
	\begin{eqnarray*}
		\textbf{H}_{\bot}^{1}(\Omega):=\{\textbf{v}\in \textbf{H}^{1}(\Omega); (\textbf{v},\textbf{r})=0~\forall \textbf{r}\in \textbf{RM}\},\\
		\textbf{L}_{\bot}^{2}(\partial\Omega):=\{\textbf{g}\in\textbf{L}^{2}(\partial\Omega); <\textbf{g},\textbf{r}>=0~\forall \textbf{r}\in \textbf{RM} \}.
	\end{eqnarray*}
	It is well known that $\textbf{RM}$ is the kernel of the strain operator $\varepsilon$, that is, $\textbf{r} \in \textbf{RM}$ if and only if $\varepsilon(\textbf{r})=0$. Hence, we have
	\begin{equation}
		\varepsilon(\textbf{r})=0,~~~~\operatorname{div}\textbf{r}=0 ~~~~~~ \forall \textbf{r}\in \textbf{RM}.
	\end{equation}
	From \cite{Lions1990}, we know that there exists a constant $c_{1}>0$ such that
	\begin{equation*}
		\inf_{\textbf{r}\in \textbf{RM}}\|\textbf{v}+\textbf{r}\|_{L^{2}(\Omega)}\leqslant c_{1}\|\varepsilon(\textbf{v})\|_{L^{2}(\Omega)}~~~~\forall \textbf{v}\in \textbf{H}^{1}(\Omega).
	\end{equation*}
	Hence, for each $\textbf{v} \in \textbf{H}_{\perp}^{1}(\Omega)$ there holds
	\begin{eqnarray}
		\|\textbf{v}\|_{L^{2}(\Omega)}=\inf _{\textbf{r} \in \textbf{RM}} \sqrt{\|\textbf{v}+\textbf{r}\|_{L^{2}(\Omega)}^{2}-\|\textbf{r}\|_{L^{2}(\Omega)}^{2}} \leq c_{1}\|\varepsilon(\textbf{v})\|_{L^{2}(\Omega)}.\label{eq210601-1}
	\end{eqnarray}
	Using (\ref{eq210601-1}) and the Korn's inequality (cf. \cite{Lions1990}),
	 we know that there exists  $c_{2}>0$ such that
	\begin{align}
		\|\textbf{v}\|_{H^{1}(\Omega)} & \leq c_{2}\left[\|\textbf{v}\|_{L^{2}(\Omega)}+\|\varepsilon(\textbf{v})\|_{L^{2}(\Omega)}\right] \nonumber\\
		& \leq c_{2}\left(1+c_{1}\right)\|\varepsilon(\textbf{v})\|_{L^{2}(\Omega)} \quad \forall \textbf{v} \in \textbf{H}_{\perp}^{1}(\Omega).
	\end{align}
	\begin{definition}\label{de-2-1}
	Suppose $\bm{{\rm u}}_{0}\in \bm{{\rm H}}_{\perp}^{1}(\Omega),\bm{{\rm f}}\in H^1(0,\tau;L^2(\Omega)),\bm{{\rm f}}_1\in H^1(0,\tau;L^2(\partial\Omega)),g,\phi\in L^2(0,\tau;L^2(\Omega))$ 
	and $g_1,\phi_1\in L^2(0,\tau;L^2(\partial\Omega))$. Assume that $(\bm{{\rm f}},\bm{{\rm v}})+\langle\bm{{\rm f}}_1,\bm{{\rm v}}\rangle=0$ for any $\bm{{\rm v}}\in\bm{{\rm RM}}$. Given $\tau>0$, a tuple $(\bm{{\rm u}},p,T)$ with
	\begin{eqnarray*}
		&&\bm{{\rm u}}\in L^{\infty}(0,\tau;\bm{{\rm H}}_{\perp}^{1}(\Omega)),~~~~~~~~~~~~~~p,T\in L^{\infty}(0,\tau;{\rm L}^{2}(\Omega))\cap {\rm L}^{2}(0,\tau;H^{1}(\Omega)),\\
		&&p_{t},(\operatorname{div} \bm{{\rm u}})_{t}, T_{t}\in L^{2}(0,\tau; H^{1}(\Omega)^{'})
	\end{eqnarray*}
	is called a weak solution to the problem (\ref{eq-2-1})-(\ref{eq-2-3}) with (\ref{eq-2-4})-(\ref{eq-2-5}) if there hold for almost every $t\in[0,\tau]$
\end{definition}
\begin{align}
	\mu(\varepsilon(\textbf{u}),\varepsilon(\textbf{v}))-\alpha(p,\nabla\cdot \textbf{v})-\beta(T,\nabla\cdot \textbf{v})+\lambda(\nabla\cdot \textbf{u},\nabla\cdot \textbf{v})\nonumber\\
	=(\textbf{f},\textbf{v})+\langle\textbf{f}_{1},\textbf{v}\rangle~~~&\forall \textbf{v}\in \textbf{H}^{1}(\Omega),\label{eq-2-16}\\
	(\partial_{t}(c_{0}p-b_{0}T+\alpha\nabla\cdot \textbf{u}),\varphi)+(K\nabla p,\nabla \varphi)
	=(g,\varphi)+\langle g_{1},\varphi\rangle~~~&\forall \varphi\in H^{1}(\Omega),\label{eq-2-17}\\
	(\partial_{t}(a_{0}T-b_{0}p+\beta\nabla\cdot \textbf{u}),\psi)-(\nabla T\cdot(\bm{K}\nabla p),\psi)+(\bm{\Theta}\nabla T,\nabla\psi)\nonumber\\
	=(\phi,\psi)+\langle \phi_{1},\psi\rangle~~~&\forall \psi\in H^{1}(\Omega),\label{eq-2-18}
	\\
	\textbf{u}(0)=\textbf{u}_{0},~~~~p(0)=p_{0},~~~~T(0)=T_{0}.\label{eq-2-19}
\end{align}

\begin{definition}\label{de-2-2}
	Suppose $\bm{{\rm u}}_{0}\in \bm{{\rm H}}_{\perp}^{1}(\Omega),\bm{{\rm f}}\in H^1(0,\tau;L^2(\Omega)),\bm{{\rm f}}_1\in H^1(0,\tau;L^2(\partial\Omega)),g,\phi\in L^2(0,\tau;L^2(\Omega))$ 
	and $g_1,\phi_1\in L^2(0,\tau;L^2(\partial\Omega))$. Assume that $(\bm{{\rm f}},\bm{{\rm v}})+\langle\bm{{\rm f}}_1,\bm{{\rm v}}\rangle=0$ for any $\bm{{\rm v}}\in\bm{{\rm RM}}$. Given $\tau>0$, a 6-tuple $(\bm{{\rm u}},\xi,\eta,\gamma,p,T)$ with
	\begin{eqnarray*}
		&&\bm{{\rm u}}\in L^{\infty}(0,\tau;\bm{{\rm H}}_{\perp}^{1}(\Omega)), \quad\xi\in L^{\infty}(0,\tau;L^{2}(\Omega))\cap L^{2}(0,\tau;H^{1}(\Omega)),\\
		&&\eta\in L^{\infty}(0,\tau;L^{2}(\Omega))\cap  H^{1}(0,\tau;H^{1}(\Omega)^{'}),\quad\gamma\in L^{\infty}(0,\tau;L^{2}(\Omega))\cap  H^{1}(0,\tau;H^{1}(\Omega)^{'}),\\
		&&p,T\in L^{\infty}(0,\tau;{\rm L}^{2}(\Omega))\cap {\rm L}^{2}(0,\tau;H^{1}(\Omega)),\quad q\in L^{\infty}(0,\tau;L^{2}(\Omega))
	\end{eqnarray*}
	is called a weak solution to the problem (\ref{eq-2-8})-(\ref{eq-2-15}), such that
\end{definition}	
\begin{align}
	\mu(\varepsilon(\bm{{\rm u}}),\varepsilon(\bm{{\rm v}}))-(\xi, \nabla\cdot \textbf{v})=(\bm{{\rm f}},\bm{{\rm v}})+\langle \bm{{\rm f}}_{1},\bm{{\rm v}}\rangle \quad&\forall \bm{{\rm v}}\in \bm{{\rm H}}^{1}(\Omega),\label{eq-2-20}\\
	k_{6}(\xi,\varphi)+(\nabla\cdot \bm{{\rm u}},\varphi) =k_{4}(\eta,\varphi)+k_{1}(\gamma,\varphi) \quad&\forall \varphi\in L^{2}(\Omega),\label{eq-2-21}\\
	(\eta_{t},y)+(\bm{K}\nabla p,\nabla y)=(g,y)+\langle g_{1},y\rangle \quad&\forall y\in H^{1}(\Omega),\label{eq-2-22}\\
	(\gamma_{t},z)+(\bm{\Theta}\nabla T,\nabla z)-(\nabla T\cdot(\bm{K}\nabla p),z)
	=(\phi,z)+\langle \phi_{1},z\rangle \quad&\forall z\in H^{1}(\Omega).\label{eq-2-23-1}
\end{align}	

\subsection{PDE analysis of thermo-poroelasticity model}
Firstly, we introduce the resulting linear problem which reads: find $(\mathbf{u}(t)$, $\xi(t)$, $\eta(t)$, $\gamma(t)$, $p(t)$, $T(t))$ $\in$ $\bm{{\rm H}}^{1}(\Omega)$ $\times$ $L^{2}(\Omega)$ $\times$ $L^{2}(\Omega)$ $\times$ $L^{2}(\Omega)$ $\times$ $H^{1}(\Omega)$ $\times$ $H^{1}(\Omega)$, such that for a.e. $t \in [0,\tau]$ there holds
\begin{align}
	\mu(\varepsilon(\bm{{\rm u}}),\varepsilon(\bm{{\rm v}}))-(\xi,\nabla\cdot\bm{{\rm v}})=(\bm{{\rm f}},\bm{{\rm v}})+\langle \bm{{\rm f}}_{1},\bm{{\rm v}}\rangle \quad&\forall \bm{{\rm v}}\in \bm{{\rm H}}^{1}(\Omega),\label{eq-2-23}\\
	k_{6}(\xi,\varphi)+(\nabla\cdot \bm{{\rm u}},\varphi) =k_{4}(\eta,\varphi)+k_{1}(\gamma,\varphi) \quad&\forall \varphi\in L^{2}(\Omega),\label{eq-2-24}\\
	(\eta_{t},y)+(\bm{K}\nabla p,\nabla y)=(g,y)+\langle g_{1},y\rangle \quad&\forall y\in H^{1}(\Omega),\label{eq-2-25}\\
	(\gamma_{t},z)+(\bm{\Theta}\nabla T,\nabla z)-(\boldsymbol{L_1}\cdot\nabla T,z)-(\boldsymbol{L_2}\cdot(\bm{K}\nabla p),z)\no\\+(\boldsymbol{L_3},z)
	=(\phi,z)+\langle \phi_{1},z\rangle \quad&\forall z\in H^{1}(\Omega),\label{eq-2-26}
\end{align}	
The remaining part of this section is devoted to prove the well-posedness of the problem \reff{eq-2-23}-\reff{eq-2-26}.  We denote by $l_{1}=\|\boldsymbol{L_1}\|_{\infty},l_{2}=\|\boldsymbol{L_2}\|_{\infty}$ and $l_{3}=\|\boldsymbol{L_3}\|_{\infty}$.
\begin{theorem}\label{th-2-3}
	Assume that the conditions of A1, A2, A3 and A4 hold, $\bm{{\rm f}}\in H^{1}(0,\tau;L^2(\Omega))\cap L^{\infty}(0,\tau;L^2(\Omega))$, $\bm{{\rm f}}_{1}\in H^{1}(0,\tau;L^2(\partial\Omega))\cap L^{\infty}(0,\tau;L^2(\partial\Omega))$, $g$,  $\phi\in L^{2}(0,\tau;L^2(\Omega))$, $g_{1}$, $\phi_{1}\in L^{2}(L^2(0,\tau;\partial\Omega))$, then there exists a positive constant $C_{1}=C_{1}(\|\varepsilon(\textbf{u}(0))\|_{L^2(\Omega)}$, $\|\xi_{0}\|_{L^2(\Omega)}$, $\|\eta_{0}|_{L^2(\Omega)}$,$\|\gamma_{0}\|_{L^2(\Omega)}$, $\|\bm{{\rm f}}\|_{H^{1}(0,\tau;L^2(\Omega))}$, $\|\bm{{\rm f}}_1\|_{H^{1}(L^2(0,\tau;\partial\Omega))}$, $\|g\|_{L^{2}(0,\tau;L^2(\Omega))}$, $\|g_1\|_{L^{2}(L^2(0,\tau;\partial\Omega))}$, $\|\phi\|_{L^{2}(0,\tau;L^2(\Omega))}$, $\|\phi_1\|_{L^{2}(L^2(0,\tau;\partial\Omega))})$, and $C_2=C_2(C_1, \|p(0)\|^{2}_{H^{1}(\Omega)}$, $\|T(0)\|^{2}_{H^{1}(\Omega)})$, such that	
	\begin{align}\label{eq-2-27}
		&\frac{\sqrt{\mu}}{2}\|\varepsilon(\bm{{\rm u}})\|_{L^{\infty}(0,\tau;L^2(\Omega))}+ \sqrt{\frac{k_6}{2}}\|\xi\|_{L^{\infty}(L^2(0,\tau;\Omega))}+\sqrt{\frac{{k_{5}-k_{2}}}{2}}\|\eta\|_{L^{\infty}(0,\tau;L^2(\Omega))}\nonumber\\
		&+\sqrt{\frac{{k_{3}-k_{2}}}{2}}\|\gamma\|_{L^{\infty}(0,\tau;L^2(\Omega))}
		+\sqrt{\frac{k_m}{2}}\|\nabla p\|_{L^{2}(0,\tau;L^2(\Omega))}\no\\
		&+\sqrt{\frac{\theta_m}{2}}\|\nabla T\|_{L^{2}(0,\tau;L^2(\Omega))}\leq
		C_1,
	\end{align}	
	\begin{align}\label{eq-2-28}
		&\sqrt{\frac{\mu}{2}}\|\varepsilon(\bm{{\rm u}}_{t})\|_{L^{2}(0,\tau;L^2(\Omega))}+\sqrt{\frac{k_6}{2}}\|\xi_{t}\|_{L^{2}(0,\tau;L^2(\Omega))}
		+\sqrt{\frac{k_{5}-k_{2}}{2}}\|\eta_{t}\|_{L^{\infty}(0,\tau;L^2(\Omega))}\nonumber\\
		&+\sqrt{\frac{k_{3}-k_{2}}{2}}\|\gamma_{t}\|_{L^{\infty}(0,\tau;L^2(\Omega))}+\sqrt{\frac{k_m}{2}}\|\nabla p(t)\|_{L^{\infty}(0,\tau;L^2(\Omega))}\no\\
		&+\sqrt{\frac{\theta_m}{2}}\|\nabla T(t)\|_{L^{\infty}(0,\tau;L^2(\Omega))}\leq C_2,
	\end{align}
	\begin{align}
		&\|\bm{{\rm u}}\|_{L^{\infty}(0,\tau,L^2(\Omega))}\leq\frac{2c_{1}}{\sqrt{\mu}}C_{1},\label{eq-2-51-2}\\
		&\|p\|_{L^{\infty}(0,\tau,L^2(\Omega))}\leq C_{1}\big(k_{4}(\frac{k_{6}}{2})^{-\frac{1}{2}}+k_{5}(\frac{k_{5}-k_{2}}{2})^{-\frac{1}{2}}+k_{2}(\frac{k_{3}-k_{2}}{2})^{-\frac{1}{2}}\big),\label{eq-2-52-1}\\			&\|p\|_{L^{2}(0,\tau;L^{2}(\Omega))}\leq\frac{C_{1}\sqrt{2k_{m}}}{k_m},~~~~\|\xi\|_{L^{2}(0,\tau;L^{2}(\Omega))}\leq k_4^{-1}\frac{C_{1}\sqrt{2k_{m}}}{k_m}.\label{eq-2-53-1}
	\end{align}	
\end{theorem}
\begin{proof}
	Differentiating (\ref{eq-2-24}) with respect to $t$, taking $\textbf{v}=\textbf{u}_{t}, \varphi=\xi, y=p=k_{4}\xi+k_{5}\eta+k_{2}\gamma$ and $z=T=k_{1}\xi+k_{2}\eta+k_{3}\gamma$ in (\ref{eq-2-23})-(\ref{eq-2-26}) respectively, we have
	\begin{align}
		&\frac{1}{2}(\mu\frac{d}{dt}\|\varepsilon(\textbf{u})\|^{2}_{L^2(\Omega)}+k_{6}\frac{d}{dt}\|\xi\|^{2}_{L^2(\Omega)}+k_{5}\frac{d}{dt}\|\eta\|^{2}_{L^2(\Omega)}+k_{3}\frac{d}{dt}\|\gamma\|^{2}_{L^2(\Omega)})\nonumber\\
		&+k_{2}(\eta_{t},\gamma)+k_{2}(\gamma_{t},\eta)+(\bm{K}\nabla p,\nabla p)+(\bm{\Theta} \nabla T,\nabla T)-(\boldsymbol{L_1}\cdot\nabla T, T)\nonumber\\
		&-(\boldsymbol{L_2}\cdot(\bm{K}\nabla p), T)+(\boldsymbol{L_3}, T)
		=(g,p)+\langle g_{1},p\rangle+(\phi,T)+\langle\phi_{1},T\rangle+(\textbf{f},\textbf{u}_{t})+\langle\textbf{f}_{1},\textbf{u}_{t}\rangle.
	\end{align}
	It is easy to check
	\begin{align}
		(\textbf{f},\textbf{u}_{t})=\frac{d}{dt}(\textbf{f},\textbf{u})-(\textbf{f}_{t},\textbf{u}),~~~~
		\langle\textbf{f}_{1},\textbf{u}_{t}\rangle=\frac{d}{dt}\langle\textbf{f}_{1},\textbf{u}\rangle-\langle\textbf{f}_{1t},\textbf{u}\rangle.
	\end{align}
	Integrating from 0 to $t$, using  Cauchy-Schwarz inequality, Young inequality and Korn's inequality, \reff{eq-2-7}, we have
	\begin{align}
		&\frac{1}{2}(\mu\|\varepsilon(\textbf{u}(t))\|^{2}_{L^2(\Omega)}+k_{6}\|\xi(t)\|^{2}_{L^2(\Omega)}+(k_{5}-k_{2})\|\eta(t)\|^{2}_{L^2(\Omega)}+(k_{3}-k_{2})\|\gamma(t)\|^{2}_{L^2(\Omega)})\no\\
		& +\int_{0}^{t}\big[k_{m}\|\nabla p(s)\|^2_{L^2(\Omega)}+\theta_{m}\|\nabla T(s)\|^2_{L^2(\Omega)}-\frac{3l_1^2}{2\epsilon_1}\|\nabla T\|^{2}_{L^2(\Omega)}
		-\frac{3l_2^2k_M^2}{2\epsilon_2}\|\nabla p\|^{2}_{L^2(\Omega)}\big]~ds\nonumber\\
		&\leq\int_{0}^{t}\big[\frac{(\epsilon_1+\epsilon_2)k_{1}^{2}}{2}\|\xi\|^{2}_{L^2(\Omega)}
		+\frac{\epsilon_1+\epsilon_2}{2}k_{2}^{2}\|\eta\|^{2}_{L^2(\Omega)}+\frac{\epsilon_1+\epsilon_2}{2}k_{3}^{2}\|\gamma\|^{2}_{L^2(\Omega)}+\frac{3}{2}\|\boldsymbol{L_3}\|^{2}_{L^2(\Omega)}\nonumber\\
		&+\frac{k_{1}^{2}}{2}\|\xi\|^{2}_{L^2(\Omega)}+\frac{k_{2}^{2}}{2}\|\eta\|^{2}_{L^2(\Omega)}+\frac{k_{3}^{2}}{2}\|\gamma\|^{2}_{L^2(\Omega)}+\frac{1}{2}\big(3\|g\|^{2}_{L^2(\Omega)}+3\|g_1\|^{2}_{L^2(\partial\Omega)}
		+2k_{4}^2\|\xi\|^{2}_{L^2(\Omega)}\no\\
		&+2k_{5}^{2}\|\eta\|^{2}_{L^2(\Omega)}+2k_{2}^{2}\|\gamma\|^{2}_{L^2(\Omega)}+3\|\phi\|^{2}_{L^2(\Omega)}+3\|\phi_{1}\|^{2}_{L^2(\partial\Omega)}+2k_{1}^{2}\|\xi\|^{2}_{L^2(\Omega)}+2k_{2}^{2}\|\eta\|^{2}_{L^2(\Omega)}\nonumber\\
		&+2k_{3}^{2}\|\gamma\|^{2}_{L^2(\Omega)}+\|\textbf{f}_{t}\|^{2}_{L^2(\Omega)}+\|\textbf{f}_{1t}\|^{2}_{L^2(\partial\Omega)}
		+2c_{1}\|\varepsilon(\textbf{u}(t))\|^{2}_{L^2(\Omega)}\big)\big]~ds+\frac{\epsilon_3}{2}(\|\textbf{f}(t)\|^{2}_{L^2(\Omega)}\nonumber\\
		&+\|\textbf{f}_{1}(t)\|^{2}_{L^2(\partial\Omega)})+\frac{\epsilon_3}{2}(\|\textbf{f}(0)\|^{2}_{L^2(\Omega)}+\|\textbf{f}_{1}(0)\|^{2}_{L^2(\partial\Omega)})+\frac{c_{1}}{\epsilon_3}(\|\varepsilon(\textbf{u}(t))\|^{2}_{L^2(\Omega)}+\|\varepsilon(\textbf{u}(0))\|^{2}_{L^2(\Omega)})\nonumber\\
		&+\frac{1}{2}\big(\mu\|\varepsilon(\textbf{u}(0))\|^{2}_{L^2(\Omega)}+k_{6}\|\xi(0)\|^{2}_{L^2(\Omega)}+(k_{5}+k_{2})\|\eta(0)\|^{2}_{L^2(\Omega)}+(k_{3}+k_{2})\|\gamma(0)\|^{2}_{L^2(\Omega)}\big).\nonumber
	\end{align}
	Choosing $ \epsilon_1=\frac{3l_1^2}{\theta_m}$ , $ \epsilon_2=\frac{3l_2^2k_M^2}{k_m}$,$ \epsilon_3=\frac{4c_1}{\mu}$ and using  Gronwall's inequality, we get
	\begin{align}
		&\frac{1}{4}\mu\|\varepsilon(\textbf{u}(t))\|^{2}_{L^2(\Omega)}+\frac{1}{2}\big( k_{6}\|\xi(t)\|^{2}_{L^2(\Omega)}+(k_{5}-k_{2})\|\eta(t)\|^{2}_{L^2(\Omega)}+(k_{3}-k_{2})\|\gamma(t)\|^{2}_{L^2(\Omega)}\big)\nonumber\\
		&+\int_{0}^{t}\frac{k_m}{2}\|\nabla p(s)\|^{2}_{L^2(\Omega)}+\frac{\theta_{m}}{2}\|\nabla T(s)\|^{2}_{L^2(\Omega)}~ds \leq C\big[\int_{0}^{t}\frac{1}{2}\big(3\|\boldsymbol{L_3}\|^{2}_{L^2(\Omega)}+3\|g\|^{2}_{L^2(\Omega)}+3\|g_1\|^{2}_{L^2(\partial\Omega)}\nonumber\\
		&+3\|\phi\|^{2}_{L^2(\Omega)}+3\|\phi_{1}\|^{2}_{L^2(\partial\Omega)}+\|\textbf{f}_{t}\|^{2}_{L^2(\Omega)}+\|\textbf{f}_{1t}\|^{2}_{L^2(\partial\Omega)}\big)ds
		+\frac{3\mu}{4}\|\varepsilon(\textbf{u}(0))\|^{2}_{L^2(\Omega)}+\frac{1}{2}\big(k_{6}\|\xi(0)\|^{2}_{L^2(\Omega)}\nonumber\\
		&+(k_{5}+k_{2})\|\eta(0)\|^{2}_{L^2(\Omega)}+(k_{3}+k_{2})\|\gamma(0)\|^{2}_{L^2(\Omega)}\big)+\frac{2c_1}{\mu}(\|\textbf{f}(t)\|^{2}_{L^2(\Omega)}+\|\textbf{f}_{1}(t)\|^{2}_{L^2(\partial\Omega)}\no\\
		&+\|\textbf{f}(0)\|^{2}_{L^2(\Omega)}+\|\textbf{f}_{1}(0)\|^{2}_{L^2(\partial\Omega)})\big],
	\end{align}
which implies that \reff{eq-2-27} holds.
	
	Differentiating (\ref{eq-2-23}) and (\ref{eq-2-24}) with respect to $t$, taking $\textbf{v}=\textbf{u}_{t}, \varphi=\xi_{t}, y=p_{t}=k_{4}\xi_{t}+k_{5}\eta_{t}+k_{2}\gamma_{t}$ and $z=T_{t}=k_{1}\xi_{t}+k_{2}\eta_{t}+k_{3}\gamma_{t}$ in (\ref{eq-2-23}), (\ref{eq-2-24}), (\ref{eq-2-25}) and (\ref{eq-2-26}), we yield
	\begin{align}\label{eq-2-33}
		&\mu\|\varepsilon(\textbf{u}_{t})\|^{2}_{L^2(\Omega)}+k_{6}\|\xi_{t}\|^{2}_{L^2(\Omega)}+k_{5}\|\eta_{t}\|^{2}_{L^2(\Omega)}+k_{3}\|\gamma_{t}\|^{2}_{L^2(\Omega)}+2k_{2}(\gamma_{t},\eta_{t})\nonumber\\
		&+(\bm{K}\nabla p,\nabla p_{t})+(\bm{\Theta}\nabla T,\nabla T_t)-(\boldsymbol{L_1}\cdot\nabla T,T_t)-(\boldsymbol{L_2}\cdot(\bm{K}\nabla p),T_t)+(\boldsymbol{L_3},T_t)\nonumber\\
		&=(\textbf{f}_{t}, \textbf{u}_{t})+\langle \textbf{f}_{1t},  \textbf{u}_{t}\rangle+(g,p_{t})+\langle g_{1},p_{t}\rangle+(\phi,T_{t})+\langle \phi_{1},T_{t}\rangle.
	\end{align}
	Using the Cauchy-Schwarz inequality and Young inequality, we obtain
	\begin{eqnarray}\label{eq-2-34}
		2k_{2}(\gamma_{t},\eta_{t})\leq k_{2}\|\gamma_{t}\|^{2}_{L^2(\Omega)}+k_{2}\|\eta_{t}\|^{2}_{L^2(\Omega)}.
	\end{eqnarray}
	Substituting \reff{eq-2-34} into \reff{eq-2-33} and integrating from $0$ to $t$, we have
	\begin{align}\label{eq-2-35}
		&\int_{0}^{t}\mu\|\varepsilon(\textbf{u}_{t})\|^{2}_{L^2(\Omega)}+k_{6}\|\xi_{t}\|^{2}_{L^2(\Omega)}+(k_{5}-k_{2})\|\eta_{t}\|^{2}_{L^2(\Omega)}+(k_{3}-k_{2})\|\gamma_{t}\|^{2}_{L^2(\Omega)}~ds\nonumber\\
		&+\frac{k_{m}}{2}\|\nabla p(t)\|^{2}_{L^2(\Omega)}
		+\frac{\theta_{m}}{2}\|\nabla T(t)\|^{2}_{L^2(\Omega)}\leq \int_{0}^{t}(\boldsymbol{L_1}\cdot\nabla T,T_t)+(\boldsymbol{L_2}\cdot(\bm{K}\nabla p),T_t) \no\\
		&+(\boldsymbol{L_3},T_t)+(f_{t}, \textbf{u}_{t})+\langle f_{1t},  \textbf{u}_{t}\rangle+(g,p_{t})
		+\langle g_{1},p_{t}\rangle+(\phi,T_{t})+\langle \phi_{1},T_{t}\rangle~ds\nonumber\\
		&+\frac{k_{M}}{2}\|\nabla p(0)\|^{2}_{L^2(\Omega)}+\frac{\theta_{M}}{2}\|\nabla T(0)\|^{2}_{L^2(\Omega)}.
	\end{align}
	Using \reff{eq-2-7}, we have
	\begin{align}\label{eq-2-3500}
		&\int_{0}^{t}\mu\|\varepsilon(\textbf{u}_{t})\|^{2}_{L^2(\Omega)}+k_{6}\|\xi_{t}\|^{2}_{L^2(\Omega)}+(k_{5}-k_{2})\|\eta_{t}\|^{2}_{L^2(\Omega)}+(k_{3}-k_{2})\|\gamma_{t}\|^{2}_{L^2(\Omega)}~ds\no\\
		&+\frac{k_{m}}{2}\|\nabla p(t)\|^{2}_{L^2(\Omega)}
		+\frac{\theta_{m}}{2}\|\nabla T(t)\|^{2}_{L^2(\Omega)}\leq \int_{0}^{t}(\boldsymbol{L_1}\cdot\nabla T,k_{1}\xi_{t}+k_{2}\eta_{t}+k_{3}\gamma_{t})\nonumber\\
		&+(\boldsymbol{L_2}\cdot(\bm{K}\nabla p),k_{1}\xi_{t}+k_{2}\eta_{t}+k_{3}\gamma_{t})+(\boldsymbol{L_3},,k_{1}\xi_{t}+k_{2}\eta_{t}+k_{3}\gamma_{t})+(\textbf{f}_{t}, \textbf{u}_{t})\no\\
		&+\langle\textbf{f}_{1t},  \textbf{u}_{t}\rangle+(g,k_{4}\xi_{t}+k_{5}\eta_{t}+k_{2}\gamma_{t})+\langle g_{1},k_{4}\xi_{t}+k_{5}\eta_{t}+k_{2}\gamma_{t}\rangle\no\\
		&+(\phi,k_{1}\xi_{t}+k_{2}\eta_{t}+k_{3}\gamma_{t})+\langle \phi_{1},k_{1}\xi_{t}+k_{2}\eta_{t}+k_{3}\gamma_{t}\rangle~ds\nonumber\\
		&+\frac{k_{M}}{2}\|\nabla p(0)\|^{2}_{L^2(\Omega)}+\frac{\theta_{M}}{2}\|\nabla T(0)\|^{2}_{L^2(\Omega)}.
	\end{align}
	Applying Cauchy-Schwarz inequality and Young inequality, we have
	\begin{align}
		&\int_{0}^{t}\mu\|\varepsilon(\textbf{u}_{t})\|^{2}_{L^2(\Omega)}+k_{6}\|\xi_{t}\|^{2}_{L^2(\Omega)}+(k_{5}-k_{2})\|\eta_{t}\|^{2}_{L^2(\Omega)}+(k_{3}-k_{2})\|\gamma_{t}\|^{2}_{L^2(\Omega)}~ds\nonumber\\
		&+\frac{k_{m}}{2}\|\nabla p(t)\|^{2}_{L^2(\Omega)}
		+\frac{\theta_{m}}{2}\|\nabla T(t)\|^{2}_{L^2(\Omega)}\leq \int_{0}^{t}(\frac{l_{1}^2}{2\epsilon_4}+\frac{l_{1}^2}{2\epsilon_5}+\frac{l_{1}^2}{2\epsilon_6})\|\nabla T\|^{2}_{L^2(\Omega)}\nonumber\\
		&	(\frac{l_{2}^2k_M^2}{2\epsilon_7}+\frac{l_{2}^2k_M^2}{2\epsilon_8}+\frac{l_{2}^2k_M^2}{2\epsilon_9})\|\nabla P\|^{2}_{L^2(\Omega)}+(\frac{1}{2\epsilon_{10}}+\frac{1}{2\epsilon_{11}}+\frac{1}{2\epsilon_{12}})\|\boldsymbol{L_3}\|^{2}_{L^2(\Omega)}\nonumber\\
		&+\frac{\epsilon_{13}}{2}(\|\textbf{f}_{t}\|^{2}_{L^2(\Omega)}+\|\textbf{f}_{1t}\|^{2}_{L^2(\Omega)})+\frac{c_{1}}{\epsilon_{13}}\|\varepsilon(\textbf{u}_{t})\|^{2}_{L^2(\Omega)}
		+(\frac{1}{2\epsilon_{14}}+\frac{1}{2\epsilon_{15}}+\frac{1}{2\epsilon_{16}})\|g\|^{2}_{L^2(\Omega)}\nonumber\\
		&+(\frac{1}{2\epsilon_{17}}+\frac{1}{2\epsilon_{18}}+\frac{1}{2\epsilon_{19}})\|g_1\|^{2}_{L^2(\partial\Omega)}
		+(\frac{1}{2\epsilon_{20}}+\frac{1}{2\epsilon_{21}}+\frac{1}{2\epsilon_{22}})\|\phi\|^{2}_{L^2(\partial\Omega)}\nonumber\\
		&+(\frac{1}{2\epsilon_{23}}+\frac{1}{2\epsilon_{24}}+\frac{1}{2\epsilon_{25}})\|\phi_1\|^{2}_{L^2(\partial\Omega)}+\frac{k^2_{1}(\epsilon_{4}+\epsilon_{7}+\epsilon_{10}+\epsilon_{20}+\epsilon_{23})+k^{2}_{4}(\epsilon_{14}+\epsilon_{17})}{2}\|\xi_{t}\|^{2}_{L^2(\Omega)}\nonumber\\
		&+\frac{k^2_{2}(\epsilon_{5}+\epsilon_{8}+\epsilon_{11}+\epsilon_{21}+\epsilon_{24})+k^{2}_{5}(\epsilon_{15}+\epsilon_{18})}{2}\|\eta_{t}\|^{2}_{L^2(\Omega)}\nonumber\\
		&+\frac{k^2_{3}(\epsilon_{6}+\epsilon_{9}+\epsilon_{12}+\epsilon_{22}+\epsilon_{25})+k^{2}_{2}(\epsilon_{16}+\epsilon_{19})}{2}\|\gamma_{t}\|^{2}_{L^2(\Omega)}~ds\no\\
		&+\frac{k_{M}}{2}\|\nabla p(0)\|^{2}_{L^2(\Omega)}+\frac{\theta_{M}}{2}\|\nabla T(0)\|^{2}_{L^2(\Omega)}.
	\end{align}
	Choosing $\epsilon_{13}=\frac{2c_{1}}{\mu}, \epsilon_{4}=\epsilon_{7}=\epsilon_{10}=\epsilon_{20}=\epsilon_{23}=\frac{k_6}{10k^{2}_{1}},\epsilon_{14}=\epsilon_{17}=\frac{k_6}{4k^{2}_{4}},\epsilon_{5}=\epsilon_{8}=\epsilon_{11}=\epsilon_{21}=\epsilon_{24}=\frac{k_5-k_2}{10k^{2}_{2}},\epsilon_{15}=\epsilon_{18}=\frac{k_5-k_2}{4k^{2}_{5}},\epsilon_{6}=\epsilon_{9}=\epsilon_{12}=\epsilon_{22}=\epsilon_{25}=\frac{k_3-k_2}{10k^{2}_{3}},\epsilon_{16}=\epsilon_{19}=\frac{k_3-k_2}{4k^{2}_{2}} $ and using Gronwall's inequality, we deduce
	\begin{align}
		&\int_{0}^{t}\frac{\mu}{2}\|\varepsilon(\textbf{u}_{t})\|^{2}_{L^2(\Omega)}+\frac{k_{6}}{2}\|\xi_{t}\|^{2}_{L^2(\Omega)}+\frac{(k_{5}-k_{2})}{2}\|\eta_{t}\|^{2}_{L^2(\Omega)}+\frac{(k_{3}-k_{2})}{2}\|\gamma_{t}\|^{2}_{L^2(\Omega)}~ds\nonumber\\
		&+\frac{k_{m}}{2}\|\nabla p(t)\|^{2}_{L^2(\Omega)}
		+\frac{\theta_{m}}{2}\|\nabla T(t)\|^{2}_{L^2(\Omega)}\leq C\big[\int_{0}^{t}\big(\|\textbf{f}_{t}\|^{2}_{L^2(\Omega)}+\|\textbf{f}_{1t}\|^{2}_{L^2(\Omega)}+\|\boldsymbol{L_3}\|^{2}_{L^2(\Omega)}\nonumber\\
		&+\|g\|^{2}_{L^2(\Omega)}+\|g_1\|^{2}_{L^2(\partial\Omega)}+\|\phi\|^{2}_{L^2(\Omega)}+\|\phi_1\|^{2}_{L^2(\Omega)}\big)~ds\nonumber\\
		&+\frac{k_{M}}{2}\|\nabla p(0)\|^{2}_{L^2(\Omega)}+\frac{\theta_{M}}{2}\|\nabla T(0)\|^{2}_{L^2(\Omega)}\big],
	\end{align}
	which implies that \reff{eq-2-28} holds.\\
	Using (\ref{eq-2-27}), we obtain 
	\begin{align}\label{eq-2-3501}
		\|\varepsilon(\bm{{\rm u}}(t))\|_{L^{\infty}(0,\tau;L^2(\Omega))}\leq\frac{2}{\sqrt{\mu}}C_{1}.
	\end{align}
	Using (\ref{eq210601-1}) and (\ref{eq-2-3501}), we get
	\begin{align}
		\|\bm{{\rm u}}(t)\|_{L^{\infty}(0,\tau;L^2(\Omega))}\leq	c_{1}\|\varepsilon(\bm{{\rm u}}(t))\|_{L^{\infty}(0,\tau;L^2(\Omega))}\leq\frac{2c_{1}}{\sqrt{\mu}}C_{1}.
	\end{align}
	Using \reff{eq-2-28} and \reff{eq-2-7}, we have
	\begin{align}
		&\|p\|_{L^{\infty}(0,\tau;L^2(\Omega))}\leq k_{4}\|\xi\|_{L^{\infty}(0,\tau;L^2(\Omega))}+k_{5}\|\eta\|_{L^{\infty}(0,\tau;L^2(\Omega))}++k_{2}\|\gamma\|_{L^{\infty}(0,\tau;L^2(\Omega))}\no\\
		&\leq\big(k_{4}(\frac{k_{6}}{2})^{-\frac{1}{2}}+k_{5}(\frac{k_{5}-k_{2}}{2})^{-\frac{1}{2}}+k_{2}(\frac{k_{3}-k_{2}}{2})^{-\frac{1}{2}})\big(\sqrt{\frac{k_{6}}{2}}\|\xi\|_{L^{\infty}(0,\tau;L^2(\Omega))}\no\\
		&+\sqrt{\frac{k_{5}-k_{2}}{2}}\|\eta\|_{L^{\infty}(0,\tau;L^2(\Omega))}+\sqrt{\frac{k_{3}-k_{2}}{2}}\|\gamma\|_{L^{\infty}(0,\tau;L^2(\Omega))}\big)\no\\
		&\leq C_{1}\big(k_{4}(\frac{k_{6}}{2})^{-\frac{1}{2}}+k_{5}(\frac{k_{5}-k_{2}}{2})^{-\frac{1}{2}}+k_{2}(\frac{k_{3}-k_{2}}{2})^{-\frac{1}{2}}\big).
	\end{align}
	Using Poincar$ \acute{\rm e}$ inequality, we get
	\begin{align}
		\|p\|_{L^{2}(0,\tau;L^2(\Omega))}\leq\|\nabla p\|_{L^{2}(0,\tau;L^2(\Omega))}\leq\frac{C_{1}\sqrt{2k_{m}}}{k_m}.
	\end{align}
	It easy to check that 
	\begin{align}
		\|\xi\|_{L^{2}(0,\tau;L^2(\Omega))}\leq k^{-1}_{4}	\|k_{4}\xi+k_{5}\eta+k_{2}\gamma\|_{L^{2}(0,\tau;L^2(\Omega))}\leq k_4^{-1}\frac{C_{1}\sqrt{2k_{m}}}{k_m}.
	\end{align}
	This proof is complete.
\end{proof}
\begin{theorem}\label{th-2-2-3-1}
	Assume that the conditions of A1, A2, A3, A4 hold and $\bm{{\rm f}}\in H^{2}(0,\tau;L^2(\Omega))\cap W^{1,\infty}(0,\tau;L^2(\Omega))$, $\bm{{\rm f}}_{1}\in H^{2}(0,\tau;\textbf{L}^2(\partial\Omega))\cap W^{1,\infty}(0,\tau;L^2(\partial\Omega))$, $g$, $\phi\in H^{1}(0,\tau;L^2(\Omega))$, $g_{1}$, $\phi_{1}\in H^{1}(0,\tau;L^2(\partial\Omega))$, then there exists a positive constant $C_3=C_3(\|\varepsilon(\textbf{u}_t(0))\|_{L^2(\Omega)}$, $\|\xi_{t}(0)\|_{L^2(\Omega)}$, $\|\eta_{t}(0)\|_{L^2(\Omega)}$, $\|\gamma_{t}(0)\|_{L^2(\Omega)}$,
	$\|\bm{{\rm f}}\|_{H^{2}(0,\tau;L^2(\Omega))}$, $\|\bm{{\rm f}}_1\|_{H^{2}(0,\tau;L^2(\partial\Omega))}$, $\|g\|_{H^{1}(0,\tau;L^2(\Omega))}$, $\|g_1\|_{H^{1}(0,\tau;L^2(\partial\Omega))}$, $\|\phi\|_{H^{1}(0,\tau;L^2(\Omega))}$, $\|\phi_1\|_{H^{1}(0,\tau;L^2(\partial\Omega))})$, such that 
	\begin{eqnarray}\label{eq-2-29}	
		&&\frac{\sqrt{\mu}}{2}\|\varepsilon(\bm{{\rm u}}_{t})\|_{L^{\infty}(0,\tau;L^2(\Omega))}+
		\sqrt{\frac{k_6}{2}}\|\xi_{t}\|_{L^{\infty}(0,\tau;L^2(\Omega))}
		+\sqrt{\frac{k_5-k_2}{2}}\|\eta_{t}\|_{L^{\infty}(0,\tau;L^2(\Omega))}\nonumber\\	
		&&+\sqrt{\frac{k_3-k_2}{2}}\|\gamma_{t}\|_{L^{\infty}(0,\tau;L^2(\Omega))}+\sqrt{\frac{{k_m}}{2}}\|\nabla p\|_{H^{1}(0,\tau;L^2(\Omega))}\nonumber\\
		&&+\sqrt{\frac{\theta_{m}}{2}}\|\nabla T\|_{H^{1}(0,\tau;L^2(\Omega))}\leq C_3.
	\end{eqnarray}	
\end{theorem}
\begin{proof}
	Differentiating (\ref{eq-2-23}), (\ref{eq-2-25}) and (\ref{eq-2-26}) one time with respect to $t$ and setting $\textbf{v}=\textbf{u}_{tt},\varphi =\xi_{t},y=p_{t}=k_{4}\xi_{t}+k_{5}\eta_{t}+k_{2}\gamma_{t},z=T_{t}=k_{1}\xi_{t}+k_{2}\eta_{t}+k_{3}\gamma_{t}$, differentiating (\ref{eq-2-24}) twice with respect to $t$ and setting $\varphi=\xi_{t}$, we get
	\begin{eqnarray}         
		&&\frac{1}{2}\big(\mu\frac{d}{dt}\|\varepsilon(\textbf{u}_{t})\|^{2}_{L^2(\Omega)}+k_{6}\frac{d}{dt}\|\xi_{t}\|^{2}_{L^2(\Omega)}+k_{5}\frac{d}{dt}\|\eta_{t}\|^{2}_{L^2(\Omega)}\nonumber\\
		&&+k_{3}\frac{d}{dt}\|\gamma_{t}\|^{2}_{L^2(\Omega)} \big)+k_{2}(\eta_{tt},\gamma_{t})+k_{2}(\gamma_{tt},\eta_{t})\nonumber\\
		&&+(\bm{K}\nabla p_{t},\nabla p_{t})+(\bm{\Theta}\nabla T_{t},\nabla T_{t})-(\boldsymbol{L_1}\cdot\nabla T_t, T_t)
		-(\boldsymbol{L_2}\cdot(\bm{K}\nabla P_t), T_t)\nonumber\\
		&&=(f_{t}, \textbf{u}_{tt})+\langle f_{1t},  \textbf{u}_{tt}\rangle+(g_{t},p_{t})+\langle g_{1t},p_{t}\rangle+(\phi_{t},T_{t})+\langle \phi_{1t},T_{t}\rangle.
	\end{eqnarray}
	Applying  Cauchy-Schwarz inequality, Young inequality and integrating from $0$ to $t$, we get
	\begin{align}	    
		&\frac{1}{2}(\mu\|\varepsilon(\textbf{u}_{t}(t))\|^{2}_{L^2(\Omega)}+k_{6}\|\xi_{t}(t)\|^{2}_{L^2(\Omega)}
		+(k_{5}-k_{2})\|\eta_{t}(t)\|^{2}_{L^2(\Omega)}+(k_{3}-k_{2})\|\gamma_{t}(t)\|^{2}_{L^2(\Omega)})\nonumber\\	
		&+\int_{0}^{t}k_{m}\|\nabla p_{t}(s)\|^{2}_{L^2(\Omega)}+\theta_{m}\|\nabla T_{t}(s)\|^{2}_{L^2(\Omega)}-\frac{3l_1^2}{2\epsilon_{1}}\|\nabla T_t\|^{2}_{L^2(\Omega)}-\frac{3l_2^2k_M^2}{2\epsilon_{2}}\|\nabla p_t\|^{2}_{L^2(\Omega)}~ds\leq
		\int_{0}^{t}\frac{\epsilon_{1}}{6}\|T_t\|^{2}_{L^2(\Omega)}\nonumber\\
		&+\frac{\epsilon_{2}}{6}\|T_t\|^{2}_{L^2(\Omega)}+\frac{1}{2}\big(\|g_{t}\|^{2}_{L^2(\Omega)}+\|g_{1t}\|^{2}_{L^2(\partial\Omega)}+2\|p_{t}\|^{2}_{L^2(\Omega)}+\|\phi_{t}\|^{2}_{L^2(\Omega)}+\|\phi_{1t}\|^{2}_{L^2(\partial\Omega)}+2\|T_{t}\|^{2}_{L^2(\Omega)}\nonumber\\	&+\|\textbf{f}_{tt}\|^{2}_{L^2(\Omega)}+\|\textbf{f}_{1tt}\|^{2}_{L^2(\partial\Omega)}+2c_{1}\|\varepsilon(\textbf{u}_{t}(s))\|^{2}_{L^2(\Omega)}\big)~ds+\frac{\epsilon_{3}}{2}(\|\textbf{f}_{t}(t)\|^{2}_{L^2(\Omega)}+\|\textbf{f}_{1t}(t)\|^{2}_{L^2(\partial\Omega)})\nonumber\\
		&+\frac{\epsilon_{3}}{2}(\|\textbf{f}_{t}(0)\|^{2}_{L^2(\Omega)}+\|\textbf{f}_{1t}(0)\|^{2}_{L^2(\partial\Omega)})+\frac{c_{1}}{\epsilon_{3}}(\|\varepsilon(\textbf{u}_{t}(t))\|^{2}_{L^2(\Omega)}+\|\varepsilon(\textbf{u}_{t}(0))\|^{2}_{L^2(\Omega)})\nonumber\\
		&+ \frac{1}{2}\big(\mu\|\varepsilon(\textbf{u}_{t}(0))\|^{2}_{L^2(\Omega)}+k_{6}\|\xi_{t}(0)\|^{2}_{L^2(\Omega)}
		+(k_{5}+k_2)\|\eta_{t}(0)\|^{2}_{L^2(\Omega)}+(k_{3}+k_2)\|\gamma_{t}(0)\|^{2}_{L^2(\Omega)}\big).
	\end{align}
	Choosing $\epsilon_{1}=\frac{3l_1^2}{\theta_m} $, $\epsilon_{2}=\frac{3l_2^2k_M^2}{k_m} $ and $\epsilon_{3}=\frac{4c_1}{\mu } $, using \reff{eq-2-7} and Gronwall's inequality, we obtain estimate \reff{eq-2-29}.
\end{proof}

%

\begin{theorem}
	Let $\bm{{\rm f}}\in H^{1}(0,\tau;L^2(\Omega))\cap L^{\infty}(0,\tau;L^2(\Omega))$,$\bm{{\rm f}}_{1}\in H^{1}(0,\tau;L^2(\partial\Omega))\cap L^{\infty}(0,\tau;L^2(\partial\Omega))$, $g$,  $\phi\in L^{2}(0,\tau;L^2(\Omega))$, $g_{1}$, $\phi_{1}\in L^{2}(L^2(0,\tau;\partial\Omega))$. Then there exists a unique weak solution to the problem (\ref{eq-2-23})-(\ref{eq-2-26}) .
\end{theorem}
\begin{proof}
	Since the problem \reff{eq-2-23}-\reff{eq-2-26} is linear, so the existence can be easily proved by using the Galerkin method and the compactness argument \cite{Temam1977}. Theorem \ref{th-2-3} provides the necessary uniform estimates for Galerkin approximate solutions, since the derivation is standard, here we omit the details.
	
	We now prove the uniqueness of the weak solution of the problem (\ref{eq-2-23})$-$(\ref{eq-2-26}). Suppose $(\textbf{u}_{1},p_{1},T_{1},\xi_{1},\eta_{1},\gamma_{1})$ and $(\textbf{u}_{2},p_{2},T_{2},\xi_{2},\eta_{2},\gamma_{2})$ are two solutions to the problem (\ref{eq-2-23})-(\ref{eq-2-26}). Let $e_{\textbf{u}}:=\textbf{u}_{1}-\textbf{u}_{2},e_{p}:=p_{1}-p_{2},e_{T}:=T_{1}-T_{2},e_{\xi}:=\xi_{1}-\xi_{2},e_{\eta}:=\eta_{1}-\eta_{2} $ and $ e_{\gamma}:=\gamma_{1}-\gamma_{2} $. Then $(e_{\textbf{u}},e_{p},e_{T},e_{\xi},e_{\eta},e_{\gamma})$ satisfies
	\begin{eqnarray}
		\mu(\varepsilon(e_{\textbf{u}}),\varepsilon(\textbf{v}))-(e_{\xi},\nabla\cdot \textbf{v})=0~~~~&&\forall ~~\textbf{v}\in \textbf{H}^{1}(\Omega),\label{eq-2-41}\\
		k_{6}(e_{\xi},\varphi)+(\nabla\cdot e_{\textbf{u}},\varphi) =k_{4}(e_{\eta},\varphi)+k_{1}(e_{\gamma},\varphi)~~~~&&\forall~~ \varphi\in L^{2}(\Omega),\label{eq-2-42}\\
		(e_{\eta_{t}},y)+(\bm{K}\nabla(e_{p}),\nabla y)=0~~~~&&\forall ~~y\in H^{1}(\Omega),\label{eq-2-43}\\
		(e_{\gamma_{t}},z)+(\bm{\Theta}\nabla e_{T},\nabla z)-(\boldsymbol{L_1}\cdot(\nabla e_T),z)-(\boldsymbol{L_2}\cdot(\bm{K}\nabla e_p),z)=0~~~~&&\forall ~~z\in H^{1}(\Omega),\label{eq-2-44}\\
		e_{\textbf{u}}(0)=e_{p}(0)=e_{T}(0)=e_{\xi}(0)=e_{\eta}(0)=e_{\gamma}(0)=0.
	\end{eqnarray}
	Differentiating \reff{eq-2-41}-\reff{eq-2-42} with respect to $t$, taking $\textbf{v}=e_{\textbf{u}_{t}}=\textbf{u}_{1t}-\textbf{u}_{2t}, \varphi=e_{\xi_{t}}=\xi_{1t}-\xi_{2t}, y=e_{p_{t}}=p_{1t}-p_{2t}$ and $z=e_{T_{t}}=T_{1t}-T_{2t}$ in (\ref{eq-2-41})-(\ref{eq-2-44}) respectively, we have
	\begin{align}
		&\mu\|\varepsilon(e_{\textbf{u}_{t}})\|^{2}_{L^2(\Omega)}+k_{6}\|e_{\xi_{t}}\|^{2}_{L^2(\Omega)}+k_{5}\|e_{\eta_{t}}\|^{2}_{L^2(\Omega)}+k_{3}\|e_{\gamma_{t}}\|^{2}_{L^2(\Omega)})+2k_{2}(e_{\eta_{t}},e_{\gamma_{t}})\nonumber\\
		&+(\bm{K}\nabla e_p,\nabla e_{p_{t}})+(\bm{\Theta} \nabla e_T,\nabla e_{T_{t}})-(\boldsymbol{L_1}\cdot\nabla e_T,e_{T_{t}})-(\boldsymbol{L_2}\cdot(\bm{K}\nabla e_p),e_{T_{t}})=0.
	\end{align}
	Integrating from 0 to $t$, using Cauchy-Schwarz inequality and Young inequality, we obtain
	\begin{align}
		\int_{0}^{t}\mu\|\varepsilon(e_{\textbf{u}_{t}}(s))\|^{2}_{L^2(\Omega)}+k_{6}\|e_{\xi_{t}}(s)\|^{2}_{L^2(\Omega)}+(k_{5}-k_2)\|e_{\eta_{t}}(s)\|^{2}_{L^2(\Omega)}\nonumber\\
		+(k_{3}-k_2)\|e_{\gamma_{t}}(s)\|^{2}_{L^2(\Omega)}~ds
		+\frac{1}{2}\big(k_{m}\|\nabla e_{p}\|^{2}_{L^2(\Omega)}+\theta_{m}\|\nabla e_{T}\|^{2}_{L^2(\Omega)}\big)\nonumber\\
		-\int_{0}^{t}\epsilon_{26}k^{2}_{1}\|e_{\xi_{t}}(s)\|^{2}_{L^2(\Omega)}+\epsilon_{27}k^{2}_{2}\|e_{\eta_{t}}(s)\|^{2}_{L^2(\Omega)}+\epsilon_{28}k^{2}_{3}\|e_{\gamma_{t}}(s)\|^{2}_{L^2(\Omega)}~ds\nonumber\\
		\leq \int_{0}^{t}(\frac{l_1^2}{2\epsilon_{26}}+\frac{l_1^2}{2\epsilon_{27}}+\frac{l_1^2}{2\epsilon_{28}})\|\nabla e_T(s)\|^{2}_{L^2(\Omega)} \no\\
		+(\frac{l_2^2k_M^2}{2\epsilon_{26}}+\frac{l_2^2k_M^2}{2\epsilon_{27}}+\frac{l_2^2k_M^2}{2\epsilon_{28}})\|\nabla e_p(s)\|^{2}_{L^2(\Omega)}~ds.
	\end{align}
	Using Gronwall inequality, choosing $\epsilon_{26} = \frac{k_{6}}{2k^{2}_{1}}, \epsilon_{27} = \frac{(k_{5}-k_{3})}{2k^{2}_{2}}, \epsilon_{28} = \frac{(k_{3}-k_{2})}{2k^{2}_{3}}$ we yields $e_{\textbf{u}_{t}}(t)=e_{\xi_{t}}(t)=e_{\eta_{t}}(t)=e_{\gamma_{t}}(t)=0$, which implies $e_{\textbf{u}}(t)$, $e_{\xi}(t)$, $e_{\eta}(t)$, $e_{\gamma}(t)$ are constants for any $\in[0,\tau]$. Since $e_{\textbf{u}}(0)=e_{\xi}(0)=e_{\eta}(0)=e_{\gamma}(0)=0$, we obtain the uniqueness of a weak solution to problem \reff{eq-2-23}-\reff{eq-2-26}. The proof is complete.
\end{proof}
\subsection{ Analysis of  the non-linear problem (\ref{eq-2-20})-(\ref{eq-2-23-1})}
We use the Newton iterative algorithm as follows: let $i \geq 1$, and at the iteration $i$, we solve for
$(\textbf{u}^{i}, \xi^{i}, \eta^{i}, \gamma^{i}, p^{i}, T^{i}) \in \bm{{\rm H}}_{\perp}^{1}(\Omega)\times L^{2}(\Omega)\times L^{2}(\Omega)\times L^{2}(\Omega) \times  H^{1}(\Omega) \times  H^{1}(\Omega)$
\begin{align}
	\mu(\varepsilon(\bm{{\rm u}}^{i}),\varepsilon(\bm{{\rm v}}))-(\xi^{i},\nabla\cdot \bm{{\rm v}})=(\bm{{\rm f}},\bm{{\rm v}})+\langle \bm{{\rm f}}_{1},\bm{{\rm v}}\rangle \quad&\forall \bm{{\rm v}}\in \bm{{\rm H}}^{1}(\Omega),\label{eq-2-48}\\
	k_{6}(\xi^{i},\varphi)+(\nabla\cdot \bm{{\rm u}}^{i},\varphi) =k_{4}(\eta^{i},\varphi)+k_{1}(\gamma^{i},\varphi) \quad&\forall \varphi\in L^{2}(\Omega),\label{eq-2-49}\\
	(\eta^{i}_{t},y)+(\bm{K}\nabla p^{i},\nabla y)=(g,y)+\langle g_{1},y\rangle \quad&\forall y\in H^{1}(\Omega),\label{eq-2-50}\\
	(\gamma^{i}_{t},z)+(\bm{\Theta}\nabla T^{i},\nabla z)-(\nabla T^{i}\cdot(\bm{K}\nabla p^{i-1}),z)\nonumber\\-(\nabla T^{i-1}\cdot(\bm{K}\nabla p^{i}),z)+(\nabla T^{i-1}\cdot(\bm{K}\nabla p^{i-1}),z)\no\\
	=(\phi,z)+\langle \phi_{1},z\rangle \quad&\forall z\in H^{1}(\Omega)\label{eq-2-51},\\
	T^{i}=k_{1}\xi^{i}+k_{2}\eta^{i}+k_{3}\gamma^{i},~~p^{i}=k_{4}\xi^{i}+k_{5}\eta^{i}+k_{2}\gamma^{i}.\label{eq-22-52}
\end{align}
\begin{remark}
	We suppose that for all $i\geq1$ and $t\in[0,\tau]$, $\nabla p^{i}(t)$, $\nabla T^{i}(t)\in L^{\infty}(\Omega)$. 
	The above hypothesis is reasonable, and it is necessary for the solution to the iterative procedure \reff{eq-2-48}-\reff{eq-22-52} to be well-defined for each $i\geq1$. 
	 This hypothesis is satisfied with sufficiently regular data and domain boundary. From Theorem \ref{th-2-3} and the theory of linear parabolic equations (see \cite{Evans2010}) to get $p^{i},T^{i} \in H^{2}(\Omega)$. Then, we can get $\nabla p^{i}$, $\nabla T^{i}\in L^{\infty}(\Omega)$, from Sobolev space embedding theorem \cite{Evans2010}.
\end{remark}

\begin{theorem}\label{2.6}
	Assume that $\bm{{\rm f}}\in H^{2}(0,\tau;L^2(\Omega))$,$\bm{{\rm f}}_{1}\in H^{1}(0,\tau;L^2(\partial\Omega)$, $g$,  $\phi\in H^{1}(0,\tau;L^2(\Omega))$, $g_{1}$,  $\phi_{1}\in H^{1}(L^2(0,\tau;\partial\Omega)$,
	$p_{0},T_{0} \in H^{1}_{0}$ and $\bm{{\rm u}}\in L^2(\Omega)$, then the \reff{eq-2-48}-\reff{eq-22-52} defines a unique sequence of iterates
	$$
	\begin{array}{l}
		\left(T^{i}, p^{i}\right) \in L^{\infty}\left(0, \tau; H^{1}(\Omega)\right)\cap H^{1}\left(0, \tau; H^{1}(\Omega)\right),\\
		\left(\xi^{i}, \eta^{i}, \gamma^{i}\right) \in W^{1, \infty}\left(0, \tau; L^{2}(\Omega)\right) \cap H^{1}(0, \tau; L^{2}(\Omega)), \\
		\mathbf{u}^{i} \in W^{1, \infty}\left(0, \tau; H^{1}(\Omega)\right)\cap H^{1}\left(0, \tau; H^{1}(\Omega)\right),
	\end{array}
	$$
	that converges to the weak solution $(\textbf{u}, \xi, \eta, \gamma, p, T)$ of \reff{eq-2-20}-\reff{eq-2-23-1}, admitting the following regularity
	$$
	\begin{array}{l}
		(T, p) \in L^{\infty}\left(0, \tau; H^{1}(\Omega)\right)\cap L^{2}\left(0, \tau; H^{1}(\Omega)\right),\\
		(\xi, \eta, \gamma) \in L^{\infty}\left(0, \tau; L^{2}(\Omega)\right) \cap H^{1}\left(0, \tau; L^{2}(\Omega)\right),\\
		\bm{{\rm u}} \in L^{\infty}\left(0, \tau; H^{1}_{\perp}(\Omega)\right) \cap H^{1}\left(0, \tau; H^{1}_{\perp}(\Omega)\right).
	\end{array}
	$$
\end{theorem}
\begin{proof}
	According to Theorem \ref{th-2-3}, the sequence $\left(\textbf{u}^{i}, \xi^{i},
	\eta^{i}, \gamma^{i}, p^{i}, T^{i}\right)$ are well-defined for all $i \geq 1$ and $p^{i}, T^{i} \in H^{1}(0, \tau; H^{1}(\Omega))$, this guarantees continuity in time for the sequence. We define $\sup _{t \in [0,\tau]}\left\|p^{i}(t)\right\|^{2}\leq\delta_{1}$ and $\sup _{t \in [0,\tau]}\left\|T^{i}(t)\right\|^{2}\leq\delta_{1}$. It remains to show the convergence of the iterates to the weak solution of\reff{eq-2-20}-\reff{eq-2-23-1} in suitable norms. To this end, let $i \geq 2$, and take the difference of equations \reff{eq-2-48}-\reff{eq-22-52} at the iteration step $i$ with the corresponding equations at iteration step $i-1$ to obtain the following problem: find $\left(e_{\textbf{u}}^{i}, e_{\xi}^{i}, e_{\gamma}^{i}, e_{\eta}^{i}, e_{p}^{i}, e_{T}^{i}\right)  \in \bm{{\rm H}}_{\perp}^{1}(\Omega)\times L^{2}(\Omega)\times L^{2}(\Omega)\times L^{2}(\Omega) \times  H^{1}(\Omega) \times  H^{1}(\Omega)$ such that 
	\begin{align}
		\mu(\varepsilon(e_{\textbf{u}}^{i}),\varepsilon(\bm{{\rm v}}))-(e_{\xi}^{i},\nabla\cdot \bm{{\rm v}})=0 \quad&\forall \bm{{\rm v}}\in \bm{{\rm H}}^{1}(\Omega),\label{eq-2-52}\\
		k_{6}(e_{\xi}^{i},\varphi)+(\nabla\cdot e_{\textbf{u}}^{i},\varphi) =k_{4}(e_{\eta}^{i},\varphi)+k_{1}(e_{\gamma}^{i},\varphi) \quad&\forall \varphi\in L^{2}(\Omega),\label{eq-2-53}\\
		(e_{\eta_{t}}^{i},y)+(\bm{K}\nabla(e_{p}^{i}),\nabla y)=0\quad&\forall y\in H^{1}(\Omega),\label{eq-2-54}\\
		(e_{\gamma_t}^{i},z)+(\bm{\Theta}\nabla e^i_T,\nabla z)-(\nabla e_{T}^{i}\cdot(\bm{K}\nabla p^{i-1}),z)\nonumber\\
		-(\nabla e_{T}^{i-1}\cdot(\bm{K}\nabla p^{i}),z)+(\nabla e_{T}^{i-1}\cdot(\bm{K}\nabla p^{i-2}),z)\nonumber\\
		-(\nabla T^{i-2}\cdot(\bm{K}\nabla e_{p}^{i}),z)=0 \quad&\forall z\in H^{1}(\Omega).\label{eq-2-55}
	\end{align}
	Differentiating (\ref{eq-2-53}) one time with respect to $t$ and setting $\textbf{v}=e^{i}_{\textbf{u}} , \varphi=e^{i}_{\xi}, y=e^{i}_{p},z=e^{i}_{T}$ in \reff{eq-2-52}-\reff{eq-2-55}, we get
	\begin{align}
		&\frac{1}{2}\frac{d}{dt}(\mu\|\varepsilon(e^{i}_{\textbf{u}})\|^{2}_{L^2(\Omega)}+k_{6}\|e_{\xi}^{i}\|^{2}_{L^2(\Omega)}+k_{5}\|e_{\eta}^{i}\|^{2}_{L^2(\Omega)}+k_{3}\|e_{\gamma}^{i}\|^{2}_{L^2(\Omega)})\nonumber\\
		&+k_{2}(e_{\eta_{t}}^{i},e_{\gamma}^{i})+k_{2}(e_{\gamma_{t}}^{i},e^{i}_{\eta})+(\bm{K}\nabla e_{p}^{i},\nabla e_{p}^{i})+(\bm{\Theta}\nabla e_{T}^{i},\nabla e_{T}^{i})\nonumber\\
		&=(\nabla e_{T}^{i}\cdot(\bm{K}\nabla p^{i-1}),e_{T}^{i})
		+(\nabla e_{T}^{i-1}\cdot(\bm{K}\nabla p^{i}),e_{T}^{i})\nonumber\\
		&-(\nabla e_{T}^{i-1}\cdot(\bm{K}\nabla p^{i-2}),e_{T}^{i})
		+(\nabla T^{i-2}\cdot(\bm{K}\nabla e_{p}^{i}),e_{T}^{i}).
	\end{align}
	Integrating from $0$ to $t$, using Cauchy-Schwarz inequality and Young inequality, we obtain
	\begin{align}
		&\frac{1}{2}(\mu\|\varepsilon(e^{i}_{\textbf{u}})\|^{2}_{L^2(\Omega)}+k_{6}\|e_{\xi}^{i}\|^{2}_{L^2(\Omega)}+(k_{5}-k_{2})\|e_{\eta}^{i}\|^{2}_{L^2(\Omega)}+(k_{3}-k_{2})\|e_{\gamma}^{i}\|^{2}_{L^2(\Omega)})\nonumber\\
		&+\int^{t}_{0} k_{m}\|\nabla e_{p}^{i}(s)\|^{2}_{L^2(\Omega)}+\theta_{m}\|\nabla e_{T}^{i}(s)\|^{2}_{L^2(\Omega)}~ds\nonumber\\
		&\leq \int^{t}_{0}\frac{k_{M}^2\delta_{1}^2}{2\epsilon_{29}}\|\nabla e_{T}^{i}\|^{2}_{L^2(\Omega)}+\frac{\epsilon_{29}}{2}\|e^{i}_{T}\|^{2}_{L^2(\Omega)}+k_{M}^2\delta_{1}^2\|\nabla e_{T}^{i-1}\|^{2}_{L^2(\Omega)}\no\\
		&+\|e^{i}_{T}\|^{2}_{L^2(\Omega)}+\frac{k_{M}^2\delta_{1}^2}{2\epsilon_{30}}\|\nabla e_{p}^{i}\|^{2}_{L^2(\Omega)}+\frac{\epsilon_{30}}{2}\|e^{i}_{T}\|^{2}_{L^2(\Omega)}~ds.
	\end{align}
	Choosing $\epsilon_{29}=\frac{k_{M}^2\delta_{1}^2}{\theta_{m}}$, $\epsilon_{30}=\frac{k_{M}^2\delta_{1}^2}{k_{m}}$, using Gronwall's inequality and  \reff{eq-2-7}, we get
	\begin{align}\label{eq-22-65}
		&\frac{1}{2}(\mu\|\varepsilon(e^{i}_{\textbf{u}})\|^{2}_{L^2(\Omega)}+k_{6}\|e_{\xi}^{i}\|^{2}_{L^2(\Omega)}+(k_{5}-k_{2})\|e_{\eta}^{i}\|^{2}_{L^2(\Omega)}+(k_{3}-k_{2})\|e_{\gamma}^{i}\|^{2}_{L^2(\Omega)})\nonumber\\
		&+\int^{t}_{0} \frac{k_{m}}{2}\|\nabla e_{p}^{i}(s)\|^{2}_{L^2(\Omega)}+\frac{\theta_{m}}{2}\|\nabla e_{T}^{i}(s)\|^{2}_{L^2(\Omega)}~ds\nonumber\\
		&\leq C\int^{t}_{0}\|\nabla e_{T}^{i-1}\|^{2}_{L^2(\Omega)}~ds,
	\end{align}
	where C is a positive constant.\\
	Differentiating \reff{eq-2-52}  and \reff{eq-2-53} with respect to $t$, taking $\textbf{v}=e^{i}_{\textbf{u}_t}, \varphi=e^{i}_{\xi_t}, y=e^{i}_{p_{t}}$ and $z=e^{i}_{T_{t}}$ in \reff{eq-2-52}-\reff{eq-2-55} respectively, we have
	\begin{align}
		&\mu\|\varepsilon(e^{i}_{\textbf{u}_t})\|^{2}_{L^2(\Omega)}+k_{6}\|e_{\xi_{t}}^{i}\|^{2}_{L^2(\Omega)}+k_{5}\|e_{\eta_{t}}^{i}\|^{2}_{L^2(\Omega)}+k_{3}\|e_{\gamma_{t}}^{i}\|^{2}_{L^2(\Omega)}\nonumber\\
		&+k_{2}(e_{\eta_{t}}^{i},e_{\gamma_{t}}^{i})+k_{2}(e_{\gamma_{t}}^{i},e^{i}_{\eta_{t}})+\frac{1}{2}\frac{d}{dt}(\bm{K}\nabla e_{p}^{i},\nabla e_{p}^{i})+\frac{1}{2}\frac{d}{dt}(\bm{\Theta}\nabla e_{T}^{i},\nabla e_{T}^{i})\nonumber\\
		&=(\nabla e_{T}^{i}\cdot(\bm{K}\nabla p^{i-1}),e_{T_t}^{i})
		+(\nabla e_{T}^{i-1}\cdot(\bm{K}\nabla p^{i}),e_{T_t}^{i})\nonumber\\
		&-(\nabla e_{T}^{i-1}\cdot(\bm{K}\nabla p^{i-2}),e_{T_t}^{i})
		+(\nabla T^{i-2}\cdot(\bm{K}\nabla e_{p}^{i}),e_{T_t}^{i}).
	\end{align}
	Integrating in $t$, we get for $t\in[0,\tau]$ and applying Cauchy-Schwarz inequality, Young inequality and \reff{eq-2-7}, we infer
	\begin{align}\label{eq-22-72}
		&\int^{t}_{0}\mu\|\varepsilon(e^{i}_{\textbf{u}_t})\|^{2}_{L^2(\Omega)}+k_{6}\|e_{\xi_{t}}^{i}\|^{2}_{L^2(\Omega)}+(k_{5}-k_{2})\|e_{\eta_{t}}^{i}\|^{2}_{L^2(\Omega)}+(k_{3}-k_{2})\|e_{\gamma_{t}}^{i}\|^{2}_{L^2(\Omega)}~ds\nonumber\\
		&+\frac{k_{m}}{2}\|\nabla e_{p}^{i}\|^{2}_{L^2(\Omega)}+\frac{\theta_{m}}{2}\|\nabla e_{T}^{i}\|^{2}_{L^2(\Omega)}\nonumber\\
		&\leq\int_{0}^{t}(\nabla e_{T}^{i}\cdot(\bm{K}\nabla p^{i-1}), k_{1}e^{i}_{\xi_{t}}+k_{2}e^{i}_{\eta_{t}}+k_{3}e^{i}_{\gamma_{t}})\nonumber\\
		&	+(\nabla e_{T}^{i-1}\cdot(\bm{K}\nabla p^{i}), k_{1}e^{i}_{\xi_{t}}+k_{2}e^{i}_{\eta_{t}}+k_{3}e^{i}_{\gamma_{t}})\nonumber\\
		&-(\nabla e_{T}^{i-1}\cdot(\bm{K}\nabla p^{i-2}), k_{1}e^{i}_{\xi_{t}}+k_{2}e^{i}_{\eta_{t}}+k_{3}e^{i}_{\gamma_{t}})\nonumber\\
		&+(\nabla T^{i-2}\cdot(\bm{K}\nabla p^{i}), k_{1}e^{i}_{\xi_{t}}+k_{2}e^{i}_{\eta_{t}}+k_{3}e^{i}_{\gamma_{t}})~ds\nonumber\\
		&\leq\int_{0}^{t}(\frac{k_{M}^2\delta_{1}^2}{2\epsilon_{14}}+\frac{k_{M}^2\delta_{1}^2}{2\epsilon_{15}}+\frac{k_{M}^2\delta_{1}^2}{2\epsilon_{16}})\|\nabla e^{i}_{T}\|^{2}_{L^2(\Omega)}\nonumber\\
		&+(\frac{k_{M}^2\delta_{1}^2}{\epsilon_{14}}+\frac{k_{M}^2\delta_{1}^2}{\epsilon_{15}}+\frac{k_{M}^2\delta_{1}^2}{\epsilon_{16}})\|\nabla e^{i-1}_{T}\|^{2}_{L^2(\Omega)}\nonumber\\
		&+(\frac{k_{M}^2\delta_{1}^2}{2\epsilon_{14}}+\frac{k_{M}^2\delta_{1}^2}{2\epsilon_{15}}+\frac{k_{M}^2\delta_{1}^2}{2\epsilon_{16}})\|\nabla e^{i}_{p}\|^{2}_{L^2(\Omega)}\nonumber\\
		&+2k^{2}_{1}\epsilon_{14}\|e_{\xi_{t}}^{i}\|^{2}_{L^2(\Omega)}+2k^{2}_{2}\epsilon_{15}\|e_{\eta_{t}}^{i}\|^{2}_{L^2(\Omega)}+2k^{2}_{3}\epsilon_{16}\|e_{\gamma_{t}}^{i}\|^{2}_{L^2(\Omega)}~ds.
	\end{align}
	Choosing $\epsilon_{14}=\frac{k_{6}}{4k^{2}_{1}}, \epsilon_{15}=\frac{k_{5}-k_{2}}{4k^{2}_{2}}, \epsilon_{16}=\frac{k_{3}-k_{2}}{4k^{2}_{3}}$ in (\ref{eq-22-72}), we yield
	\begin{align}\label{eq-22-73}
		&\int^{t}_{0}\mu\|\varepsilon(e^{i}_{\textbf{u}_t})\|^{2}_{L^2(\Omega)}+\frac{k_{6}}{2}\|e_{\xi_{t}}^{i}\|^{2}_{L^2(\Omega)}+\frac{(k_{5}-k_{2})}{2}\|e_{\eta_{t}}^{i}\|^{2}_{L^2(\Omega)}+\frac{(k_{3}-k_{2})}{2}\|e_{\gamma_{t}}^{i}\|^{2}_{L^2(\Omega)}~ds\nonumber\\
		&+\frac{k_{m}}{2}\|\nabla e_{p}^{i}\|^{2}_{L^2(\Omega)}+\frac{\theta_{m}}{2}\|\nabla e_{T}^{i}\|^{2}_{L^2(\Omega)}\nonumber\\
		&\leq\int_{0}^{t}C_{4}(k_{M}^2\delta_{1}^2\|\nabla e^{i}_{T}\|^{2}_{L^2(\Omega)}+2k_{M}^2\delta_{1}^2\|\nabla e^{i-1}_{T}\|^{2}_{L^2(\Omega)}
		+k_{M}^2\delta_{1}^2\|\nabla e^{i}_{p}\|^{2}_{L^2(\Omega)})~ds.
	\end{align}
	Using Gronwall's inequality and \reff{eq-22-73}, we have
	\begin{align}
		&\int^{t}_{0}\mu\|\varepsilon(e^{i}_{\textbf{u}_t})\|^{2}_{L^2(\Omega)}+\frac{k_{6}}{2}\|e_{\xi_{t}}^{i}\|^{2}_{L^2(\Omega)}+\frac{(k_{5}-k_{2})}{2}\|e_{\eta_{t}}^{i}\|^{2}_{L^2(\Omega)}+\frac{(k_{3}-k_{2})}{2}\|e_{\gamma_{t}}^{i}\|^{2}_{L^2(\Omega)}~ds\nonumber\\
		&+\frac{k_{m}}{2}\|\nabla e_{p}^{i}\|^{2}_{L^2(\Omega)}+\frac{\theta_{m}}{2}\|\nabla e_{T}^{i}\|^{2}_{L^2(\Omega)}\nonumber\\
		&\leq2 C_{4}k_{M}^{2}\delta_{1}^{2}\exp(C_{1}k_{M}^{2}\delta_{1}^{2}\tau)\int_{0}^{t}\|\nabla e^{i-1}_{T}\|^{2}_{L^2(\Omega)}~ds\nonumber\\
		&\leq 2C_{4}k_{M}^{2}\delta_{1}^{2}\exp(C_{1}k_{M}^{2}\delta_{1}^{2}\tau)\int_{0}^{t_1
		}\|\nabla e^{i-1}_{T}\|^{2}_{L^2(\Omega)}~ds,
	\end{align}
	for $t\leq t_{1}$ where $t_{1}>0$ will be fixed later,  and where $C_{4}=\frac{2k^{2}_{1}(k_{5}-k_{2})(k_{3}-k_{2})+2k_{6}k^{2}_{2}(k_{3}-k_{2})+2k^{2}_{3}k_{6}(k_{5}-k_{2})}{k_{6}(k_{5}-k_{2})(k_{3}-k_{2})}$. Integrating  in time once more from 0 to $t_{1}$ yields
	\begin{eqnarray}\label{1.111}
		\int^{t_{1}}_{0} \|e^{i}_{T}(s)\|^{2}_{H^1(\Omega)} ds \leq t_{1}C_{k}\int^{t_{1}}_{0} \|e^{i-1}_{T}(s)\|^{2}_{H^1(\Omega)} ds,
	\end{eqnarray}
	where the constant $C_{k} = 2C_{4}k_{M}^{2}\delta_{1}^{2}\exp(C_{1}k_{M}^{2}\delta_{1}^{2}\tau)$ is independent of $i$ and of the local final time $t_{1} .$ Thus, for $t_{1}=\frac{1}{2 C_{k}}$ the above expression implies that the map $\mathbf{e}_{T}^{i-1}(t) \mapsto$ $\mathbf{e}_{T}^{i}(t)$ is a contraction map for $t \in\left(0, t_{1}\right].$ In particular, this implies that as $i \rightarrow \infty$ we have from \reff{eq-22-65}, \reff{1.111} and the Banach Fixed Point Theorem the following convergences  \\  	
	$\mathbf{e}_{\mathbf{\xi}}^{i}, \mathbf{e}_{\mathbf{\eta}}^{i}, \mathbf{e}_{\mathbf{\gamma}}^{i}\rightarrow 0$ in \ $L^{ \infty}\left(0, t_{1} ; L^{2}(\Omega)\right) \cap H^{1}\left(0, t_{1} ; L^{2}(\Omega)\right)$,\\
	$e_{p}^{i}, e_{T}^{i} \rightarrow 0$ in \ $L^{\infty}\left(0, t_{1} ; H^{1}(\Omega)\right)\cap L^{2}\left(0, t_{1} ; H^{1}(\Omega)\right)$,\\
	$\mathbf{e}_{\mathbf{u}}^{i} \rightarrow 0$ in \ $L^{\infty}\left(0, t_{1} ; H^{1}_{\perp}(\Omega)\right)\cap H^{1}\left(0, t_{1} ; H^{1}_{\perp}(\Omega)\right)$.\\
	Observe that the time $t_{1}>0$ depends only upon the several constants, We can therefore repeat the argument above to extend our solution to the time interval $[t_{1}, 2t_{1}]$ Continuing. After finitely many steps, we construct
	a weak solution existing on the full interval $[0, \tau]$. The proof is complete.
\end{proof} 	
	\section{Fully discrete multiphysics finite element method}\label{sec-3}
	Assume that $\Omega\in \mathbb{R}^{d}(d=2,3)$ is a polygonal domain. Let $\mathcal{T}_{h}$ be a uniform triangulation or rectangular partition of $\Omega$ with mesh size $h$ and $\overline{\Omega}=\cup_{E\in\mathcal{T}_{h}}\overline{E}$. Also, let $(\textbf{X}_{h},M_{h})$ be a stable mixed finite element pair, that is, $\textbf{X}_{h}\subset \textbf{H}^{1}(\Omega) $~and~$ M_{h}\subset L^{2}(\Omega)$ satisfy the inf-sup condition
	\begin{eqnarray}\label{eq-3-1}
		\sup_{\textbf{v}_{h}\in \textbf{X}_{h}} \frac{(\operatorname{div} \textbf{v}_{h},\varphi_{h})}{\|v_{h}\|_{H^{1}(\Omega)}}\geq \beta_{0}\|\varphi_{h}\|_{L^{2}(\Omega)} ~~~~~\forall \varphi_{h}\in M_{0h}:=M_{h}\cap L^{2}_{0}(\Omega), \beta_{0}>0.
	\end{eqnarray}
	A well-known example that satisfies ($\ref{eq-3-1}$) is the following  Taylor-Hood element (cf. \cite{Bercovier1979,Roberts1991}):
	\begin{eqnarray*}
		\textbf{X}_{h}=\{\textbf{v}_{h}\in \textbf{C}^{0}(\overline{\Omega});~\textbf{v}_{h}|_{E}\in \bm{P}_{2}(E)~~ \forall E \in \mathcal{T}_{h}\},\\
		M_{h}=\{\varphi_{h}\in C^{0}(\overline{\Omega});~\varphi_{h}|_{E}\in P_{1}(E)~~\forall E\in \mathcal{T}_{h}\}.
	\end{eqnarray*}
	The finite element approximation space $W_h$ for $\eta$ variable and $Z_h$ for $\gamma$ also are $M_h$. Recall the definition of $\textbf{RM}$, it's easy to see that $\textbf{RM} \subset \textbf{X}_h$.
	Moreover, we define
	$$
	\textbf{V}_{h}:=\left\{\textbf{v}_{h} \in \textbf{X}_{h};~\left(\textbf{v}_{h}, \textbf{r}\right)=0 \quad\forall \textbf{r} \in \textbf{RM}\right\}.
	$$
	It is easy to check that $\bm{X}_{h}=\bm{V}_{h} \oplus \textbf{RM}$. It was proved in \cite{Feng2010} that there holds the following alternative version of  inf-sup condition:
	$$
	\sup _{\textbf{v}_{h} \in \bm{V}_{h}} \frac{\left(\operatorname{div} \textbf{v}_{h}, \varphi_{h}\right)}{\left\|\textbf{v}_{h}\right\|_{H^{1}(\Omega)}} \geq \beta_{1}\left\|\varphi_{h}\right\|_{L^{2}(\Omega)} \quad \forall \varphi_{h} \in M_{0 h}, \quad \beta_{1}>0.
	$$
	We recall the following inverse inequality for the polynomial functions (cf. \cite{Ciarlet1978}):
	\begin{eqnarray}\label{eq-3-111}
		\left\|\nabla \varphi_{h}\right\|_{L^{2}(E)} \leq c_{1} h^{-1}\left\|\varphi_{h}\right\|_{L^{2}(E)} \quad \forall \varphi_{h} \in P_{k}(E), E \in\mathcal{T}_{h}.
	\end{eqnarray}
	The cut-off operator $\mathcal{N}$ (cf. \cite{Sun2005}) is defined as is uniformly Lipschitz continuous
		\begin{align}
		\mathcal{N}(c)(x) &=\min (c(x), N) \\
		\mathcal{N}(\mathbf{u})(x) &=\left\{\begin{array}{ll}
			\mathbf{u}(x) & \text { if }|\mathbf{u}(x)| \leqslant N, \\
			N \mathbf{u}(x) /|\mathbf{u}(x)| & \text { if }|\mathbf{u}(x)|>N,
		\end{array}\right.\label{eq-33-4}
	\end{align}
    where $N$ is a large positive constant. 
    
	Next, we propose the multiphysics finite element algorithm as follows:\\
	\textbf{Multiphysics finite element method (MFEM):}\\
	(i) Compute $\textbf{u}^{0}_{h}\in \textbf{V}_{h}$ and $q^{0}_{h}\in W_{h}$ by
	\begin{eqnarray}\label{eq-33-3}
	&&\textbf{u}^{0}_{h}=\textbf{u}_{0},~~~~~~~p^{0}_{h}=p_{0},~~~~~~~~T^{0}_{h}=T_{0}.\nonumber\\
	&&\xi^{0}_{h}=\alpha p^{0}_{h}+\beta T^{0}_{h}-\lambda q^{0}_{h},~~
	\gamma^0_h=a_{0}T^{0}_{h}-b_{0}p^{0}_{h}+\beta q^{0}_{h},\nonumber\\
	&&\eta^0_h=c_{0}p^{0}_{h}-b_{0}T^{0}_{h}+\alpha q^{0}_{h},
\end{eqnarray}
	where $\mathcal{Q}_{h}$ is the $L^2$-projection operator defined by (\ref{eq210607-1}).\\
	(ii) For $n=0,1,2,...,$ do the following two steps.\\
	Step 1: Solve for $(\textbf{u}^{n+1}_{h},\xi^{n+1}_{h},\eta^{n+1}_{h},\gamma^{n+1}_{h})\in \textbf{V}_h\times M_h\times W_h\times Z_h$
		\begin{align}
		&\mu(\varepsilon(\textbf{u}^{n+1}_{h}),\varepsilon(\textbf{v}_{h}))-(\xi^{n+1}_{h},\nabla\cdot \textbf{v}_{h})=(\textbf{f},\textbf{v}_{h})+\langle \textbf{f}_{1},\textbf{v}_h \rangle~~~\forall ~~\textbf{v}_{h}\in \bm{V}_{h},\label{eq-3-113}\\
		&k_{6}(\xi^{n+1}_{h},\varphi_{h})+(\nabla\cdot \textbf{u}^{n+1}_{h},\varphi_{h}) =k_{4}(\eta^{n+\theta}_{h},\varphi_{h})+k_{1}(\gamma^{n+\theta}_{h},\varphi_{h})~~~\forall~~ \varphi_{h}\in M_h,\label{eq-3-114}\\
		&(d_{t}\eta^{n+1}_{h},y_{h})+(\bm{K}\nabla(k_{4}\xi^{n+1}_{h}+k_{5}\eta^{n+1}_{h}+k_{2}\gamma^{n+1}_{h}),\nabla y_{h})\nonumber\\
		&=(g,y_{h})+\langle g_{1},y_{h}\rangle~~~\forall ~~y_{h}\in W_h,\label{eq-3-115}\\
		&(d_{t}\gamma^{n+1}_{h},z_{h})+(\bm{\Theta}\nabla(k_{1}\xi^{n+1}_{h}+k_{2}\eta^{n+1}_{h}+k_{3}\gamma^{n+1}_{h}),\nabla z_{h})\nonumber\\
		&+(-\nabla(k_{1}\xi^{n+1}_{h}+k_{2}\eta^{n+1}_{h}+k_{3}\gamma^{n+1}_{h})\cdot(\bm{K}\nabla(k_{4}\xi^{n+1}_{h}+k_{5}\eta^{n+1}_{h}+k_{2}\gamma^{n+1}_{h}),z_{h})\nonumber\\
		&=(\phi,z_{h})+\langle \phi_{1},z_{h}\rangle~~~\forall ~~z_{h}\in Z_h,\label{eq-3-116}
	\end{align}
	where $\theta=0$ or $1$, $d_{t}\eta^{n}_{h}=\frac{\eta^{n}_{h}-\eta^{n-1}_{h}}{\Delta t}$.\\
	Step 2: Update $p^{n+1}_{h}$, $T^{n+1}_{h}$ and $q^{n+1}_{h}$ by
	\begin{align}\label{100.3}
		&p^{n+1}_{h} = k_{4}\xi^{n+1}_{h}+k_{5}\eta^{n+\theta}_{h}+k_{2}\gamma^{n+\theta}_{h},
		T^{n+1}_{h} = k_{1}\xi^{n+1}_{h}+k_{2}\eta^{n+\theta}_{h}+k_{3}\gamma^{n+\theta}_{h},\nonumber\\
		&q^{n+1}_{h}=-k_{6}\xi^{n+1}_{h}+k_{4}\eta^{n+1}_{h}+k_{1}\gamma^{n+1}_{h}.
	\end{align}

	\begin{remark}
		\rm In the first step of the algorithm, the problem \reff{eq-3-115}-\reff{eq-3-116} is nonlinear and it can be solved by the Newton¡¯s method. Now let ${\eta^{n+1,i}_{h}, \gamma^{n+1,i}_{h}}$ denote that fully discrete solution at the $ith$ step within the Newton method at the time $t_{n+1}$, we can obtain the Newton's method of \reff{eq-3-113}-\reff{eq-3-116}
		\begin{align}
			&\mu(\varepsilon(\textbf{u}^{n+1,i}_{h}),\varepsilon(\textbf{v}_{h}))-(\xi^{n+1,i}_{h},\nabla\cdot \textbf{v}_{h})=(\textbf{f},\textbf{v}_{h})+\langle \textbf{f}_{1},\textbf{v}_h \rangle~~~\forall ~~\textbf{v}_{h}\in \bm{V}_{h},\label{eq-3-3}\\
			&k_{6}(\xi^{n+1,i}_{h},\varphi_{h})+(\nabla\cdot \textbf{u}^{n+1,i}_{h},\varphi_{h}) =k_{4}(\eta^{n+\theta,i}_{h},\varphi_{h})+k_{1}(\gamma^{n+\theta,i}_{h},\varphi_{h})~~~\forall~~ \varphi_{h}\in M_h,\label{eq-3-4}\\
			&(d_{t}\eta^{n+1,i}_{h},y_{h})+(\bm{K}\nabla(k_{4}\xi^{n+1,i}_{h}+k_{5}\eta^{n+1,i}_{h}+k_{2}\gamma^{n+1,i}_{h}),\nabla y_{h})\nonumber\\
			&=(g,y_{h})+\langle g_{1},y_{h}\rangle~~~\forall ~~y_{h}\in W_h,\label{eq-3-5}\\
			&(d_{t}\gamma^{n+1,i}_{h},z_{h})+(\bm{\Theta}\nabla(k_{1}\xi^{n+1,i}_{h}+k_{2}\eta^{n+1,i}_{h}+k_{3}\gamma^{n+1,i}_{h}),\nabla z_{h})\nonumber\\
			&-(\nabla(k_{1}\xi^{n+1,i}_{h}+k_{2}\eta^{n+1,i}_{h}+k_{3}\gamma^{n+1,i}_{h})\cdot(\bm{K}\nabla(k_{4}\xi^{n+1,i-1}_{h}+k_{5}\eta^{n+1,i-1}_{h}+k_{2}\gamma^{n+1,i-1}_{h}),z_{h})\nonumber\no\\
			&-(\nabla(k_{1}\xi^{n+1,i-1}_{h}+k_{2}\eta^{n+1,i-1}_{h}+k_{3}\gamma^{n+1,i-1}_{h})\cdot(\bm{K}\nabla(k_{4}\xi^{n+1,i}_{h}+k_{5}\eta^{n+1,i}_{h}+k_{2}\gamma^{n+1,i}_{h}),z_{h})\nonumber\no\\
			&+(\nabla(k_{1}\xi^{n+1,i-1}_{h}+k_{2}\eta^{n+1,i-1}_{h}+k_{3}\gamma^{n+1,i-1}_{h})\cdot(\bm{K}\nabla(k_{4}\xi^{n+1,i-1}_{h}+k_{5}\eta^{n+1,i-1}_{h}+k_{2}\gamma^{n+1,i-1}_{h}),z_{h})\nonumber\no\\
			&=(\phi,z_{h})+\langle \phi_{1},z_{h}\rangle~~~\forall ~~z_{h}\in Z_h.\label{eq-3-6}
		\end{align}
		The scheme is $L$-type iterative scheme\cite{Brun2020}.
	\end{remark}
\begin{theorem}
\rm	Assume that $A1$--$A4$ hold, the solution of the problem (\ref{eq-3-113})--(\ref{eq-3-116}) is unique. 
\end{theorem}	
	\begin{proof}
		In the above scheme, we use $(\mathcal{N}(\nabla(k_{1}\xi^{n+1}_{h}+k_{2}\eta^{n+1}_{h}+k_{3}\gamma^{n+1}_{h}))\cdot\mathcal{N}((\bm{K}\nabla(k_{4}\xi^{n+1}_{h}+k_{5}\eta^{n+1}_{h}+k_{2}\gamma^{n+1}_{h})),z_{h})$ for the approximation of the convective coupling term instead of the original $(\nabla(k_{1}\xi^{n+1}_{h}+k_{2}\eta^{n+1}_{h}+k_{3}\gamma^{n+1}_{h})\cdot\bm{K}\nabla(k_{4}\xi^{n+1}_{h}+k_{5}\eta^{n+1}_{h}+k_{2}\gamma^{n+1}_{h})),z_{h})$.  Obviously, if the exict fluxes are bounded , i.e., $\nabla(k_{1}\xi^{n+1}_{h}+k_{2}\eta^{n+1}_{h}+k_{3}\gamma^{n+1}_{h}),\bm{K}\nabla(k_{4}\xi^{n+1}_{h}+k_{5}\eta^{n+1}_{h}+k_{2}\gamma^{n+1}_{h}) \in (L^\infty(\Omega))^d$, then if we picked $N$ large enough, we have practically $\mathcal{N}(\nabla(k_{1}\xi^{n+1}_{h}+k_{2}\eta^{n+1}_{h}+k_{3}\gamma^{n+1}_{h})=\nabla(k_{1}\xi^{n+1}_{h}+k_{2}\eta^{n+1}_{h}+k_{3}\gamma^{n+1}_{h})$ and $\mathcal{N}(\bm{K}\nabla(k_{4}\xi^{n+1}_{h}+k_{5}\eta^{n+1}_{h}+k_{2}\gamma^{n+1}_{h}))=\bm{K}\nabla(k_{4}\xi^{n+1}_{h}+k_{5}\eta^{n+1}_{h}+k_{2}\gamma^{n+1}_{h})$.
As for the case of $\theta=1$, the proof ideas is similar to the continuous case, therefore,
we just prove as for the case of $\theta=0$.
		let $\be_\bu^{n+1,i}=\bu_h^{n+1,i}-\bu_h^{n+1,i-1},e_\xi^{n+1,i}=\xi_h^{n+1,i}-\xi_h^{n+1,i-1},e_\eta^{n+1,i}=\eta_h^{n+1,i}-\eta_h^{n+1,i-1},e_\gamma^{n+1,i}=\gamma_h^{n+1,i}-\gamma_h^{n+1,i-1},e_T^{n+1,i}=T_h^{n+1,i}-T_h^{n+1,i-1},	e_p^{n+1,i}=p_h^{n+1,i}-p_h^{n+1,i-1}$£¬we begin by deriving the error equations satisfied by  ($\be_\bu^{n+1,i}$,$e_\xi^{n+1,i}$,$e_\eta^{n+1,i}$,$e_\gamma^{n+1,i}$,$e_p^{n+1,i}$,$e_T^{n+1,i}$), i.e. subtract the equations (\ref{eq-3-3})--(\ref{eq-3-6}) for $i$ from the ones for $i-1$, and obtain 
		\begin{alignat}{2}\label{101.4}
		&	\mu\left(\varepsilon(\be_\bu^{n+1,i}) ,\varepsilon(\bv_h)\right) -\left( e_\xi^{n+1,i},\nab\cdot \bv_h\right)=0
			\qquad\forall \bv_h\in \bV_h ,\\  
		&k_6(e_\xi^{n+1,i},\varphi_h)+\left( \nab\cdot \be_\bu^{n+1,i},\varphi_h\right) =k_4\left( e_\eta^{n,i},\varphi_h\right)+k_1\left(e_\gamma^{n,i},\varphi_h\right)\qquad \forall \varphi_h\in M_h,\label{101.5} \\
		&	\left(d_te_\eta^{n+1,i},y_h \right)+
			\left(\bm{K}\nabla e_p^{n+1,i},\nabla y_h\right)
			=0\qquad \forall\psi_h\in W_h,\label{101.6}\\
		&	\left(d_te_\gamma^{n+1,i},z_h\right)+\left(\bm{\Theta} \nab e_T^{n+1,i},\nabla z_h\right)
			-(\nab e_T^{n+1,i}\cdot\mathcal{N}(\bm{K} \nabla p_h^{n+1,i-1}) ,z_h)\no\\
		&	-(\nab e_T^{n+1,i-1}\cdot\mathcal{N}(\bm{K} \nabla p_h^{n+1,i}) ,z_h)-(\mathcal{N}(\nab T_h^{n+1,i-2})\cdot\bm{K}\nab e_p^{n+1,i},z_h)~~~~~~~~~~~~~~~~~~~~\no\\
		&	+(\nabla e_T^{n+1,i-1}\cdot\mathcal{N}( \bm{K} \nabla p_h^{n+1,i-2}),z_h)=0\qquad\forall z_h\in Z_h\label{101.7}.
		\end{alignat}
From (\ref{100.3}), let $e_p^{0,i}=0$, $e_T^{0,i}=0$ and $e_\xi^{0,i}=0$, we can define
		\begin{alignat}{2}\label{112.4}
			(p_h^{0,i},y_h)=k_4(\xi_h^{0,i},y_h)+k_5(\eta_h^{-1,i},y_h)+k_2(\gamma_h^{-1,i},y_h),\\
			(T_h^{0,i},z_h)=k_1(\xi_h^{0,i},z_h)+k_2(\eta_h^{-1,i},z_h)+k_3(\gamma_h^{-1,i},z_h).
		\end{alignat}
		Setting $\bv_h=d_t\be_u^{n+1,i}$ in (\ref{101.4}), $\varphi_h=e_\xi^{n+1,i}$ in (\ref{101.5}) (after using operator $d_t$), $y_h=e_p^{n+1,i}$ in  (\ref{101.6})
		and $z_h=e_T^{n+1,i}$ in (\ref{101.7}), after lowing the super-index from $n+1$ to $n$ on the both sides of (\ref{101.6}) and (\ref{101.7}), we get
		\begin{alignat}{2}\label{113.4}
		&	\mu\left(\varepsilon(\be_\bu^{n+1,i}) ,\varepsilon(d_t\be_\bu^{n+1,i})\right) -\left( e_\xi^{n+1,i},\nabla\cdot (d_t\be_\bu^{n+1,i})\right)=0,\\  
	&	k_6\left(d_te_\xi^{n+1,i},e_\xi^{n+1,i} \right) +\left( \nabla\cdot (d_t\be_\bu^{n+1,i}),e_\xi^{n+1,i}\right) \label{113.5}\\
	&	=k_4\left(d_te_\eta^{n,i},e_\xi^{n+1,i}\right)+k_1\left(d_te_\gamma^{n,i},e_\xi^{n+1,i}\right),\no \\
		&	\left(d_te_\eta^{n,i},k_5e_\eta^{n,i} \right)+\left(d_te_\eta^{n,i},k_4e_\xi^{n+1,i} \right)+\left(d_te_\eta^{n,i},k_2e_\gamma^{n,i} \right) \label{113.6}\\
		&	+
			\left(\bm{K}\nabla e_p^{n,i},\nabla e_p^{n+1,i}\right) =0,\no\\
		&	\left(d_te_\gamma^{n,i},k_1e_\xi^{n+1,i}\right)+\left(d_te_\gamma^{n,i},k_2e_\eta^{n,i}\right)
			+\left(d_te_\gamma^{n,i},k_3e_\gamma^{n,i}\right)
				+\left(\bm{\Theta} \nab e_T^{n,i},\nabla e_T^{n+1,i}\right)\no\\
			&-(\nab e_T^{n,i}\cdot\mathcal{N}(\bm{K} \nabla p_h^{n,i-1}) ,e_T^{n+1,i})
			-(\nab e_T^{n,i-1}\cdot\mathcal{N}(\bm{K} \nabla p_h^{n,i}) ,e_T^{n+1,i})\no\\
			&-(\mathcal{N}(\nab T_h^{n,i-2})\cdot\bm{K}\nab e_p^{n,i},e_T^{n+1,i})
			+(\nabla e_T^{n,i-1}\cdot\mathcal{N}( \bm{K} \nabla p_h^{n,i-2}),e_T^{n+1,i})=0\label{113.7}.
		\end{alignat}
	The first term of the left-hand of (\ref {113.4}) can be rewritten as
	\begin{alignat}{2}\label{103.4}
		\mu\left(\varepsilon(\be_\bu^{n+1,i}) ,\varepsilon(d_t\be_\bu^{n+1,i})\right)=\frac{\mu}{2}d_t\lVert\varepsilon(\be_\bu^{n+1,i})\rVert_{L^2(\Omega)}^2+\frac{\mu}{2}\Delta t\lVert\varepsilon(d_t\be_\bu^{n+1,i})\rVert_{L^2(\Omega)}^2.
	\end{alignat}
	Using (\ref {113.5}), (\ref {113.6}) and (\ref {113.7}), we have
	\begin{alignat}{2}\label{103.5}
	&	k_6(d_te_\xi^{n+1,i},e_\xi^{n+1,i})=\frac{k_6}{2}d_t\lVert e_\xi^{n+1,i}\rVert_{L^2(\Omega)}^2+\frac{k_6}{2}\Delta t\lVert d_te_\xi^{n+1,i}\rVert_{L^2(\Omega)}^2,\\
	&	k_5(d_te_\eta^{n,i},e_\eta^{n,i})=\frac{k_5}{2}d_t\lVert e_\eta^{n,i}\rVert_{L^2(\Omega)}^2+\frac{k_5}{2}\Delta t\lVert d_te_\eta^{n,i}\rVert_{L^2(\Omega)}^2,\\\label{103.6}
	&	k_3(d_te_\gamma^{n,i},e_\gamma^{n,i})=\frac{k_3}{2}d_t\lVert e_\gamma^{n,i}\rVert_{L^2(\Omega)}^2+\frac{k_3}{2}\Delta t\lVert d_te_\gamma^{n,i}\rVert_{L^2(\Omega)}^2,\\\label{103.7}
	&
	k_2(d_te_\eta^{n,i},e_\gamma^{n,i})+k_2(e_\eta^{n,i},d_te_\gamma^{n,i})=k_{2}d_t(e_\eta^{n,i},e_\gamma^{n,i})+k_{2}\Delta t(d_{t}e_\eta^{n,i},d_{t}e_\gamma^{n,i}).
	\end{alignat}
		Moreover, it is easy to check that
		\begin{alignat}{2}\label{114.4}
			(\bm{K}\nabla e_p^{n,i},\nabla e_p^{n+1,i})=(\bm{K}\nabla e_p^{n,i},\nabla e_p^{n,i})+k_4\Delta t(\bm{K}\nabla e_p^{n,i},d_t\nab e_\xi^{n+1,i}),\\\label{114.5}
			(\bm{\Theta}\nabla e_T^{n,i},\nabla e_T^{n+1,i})=(\bm{\Theta}\nabla e_T^{n,i},\nabla e_T^{n,i})+k_1\Delta t(\bm{\Theta}\nabla e_T^{n,i},d_t\nab e_\xi^{n+1,i}).
		\end{alignat}
		Adding (\ref{113.4})--(\ref{113.7}), using (\ref{103.4})--(\ref{103.7}) and (\ref{114.4})--(\ref{114.5}), Cauchy-Schwarz inequality and Young inequality, we get
		\begin{alignat}{2}\label{115.4}
			&	\frac{\mu}{2}d_t\lVert\varepsilon(\be_\bu^{n+1,i})\rVert _{L^2(\Omega)}^2+\frac{\mu}{2}\Delta t\lVert\varepsilon(d_t\be_\bu^{n+1,i})\rVert_{L^2(\Omega)}^2+\frac{k_6}{2}d_t\lVert e_\xi^{n+1,i}\rVert_{L^2(\Omega)}^2+\frac{k_6}{2}\Delta t\lVert d_te_\xi^{n+1,i}\rVert_{L^2(\Omega)}^2\no\\
			&	+\frac{k_5-k_2}{2}d_t\lVert e_\eta^{n,i}\rVert_{L^2(\Omega)}^2+\frac{k_5-k_2}{2}\Delta t\lVert d_te_\eta^{n,i}\rVert_{L^2(\Omega)}^2+\frac{k_3-k_2}{2}d_t\lVert e_\gamma^{n,i}\rVert_{L^2(\Omega)}^2~~~~~~~~~~~~~~\no\\
			&	+\frac{k_3-k_2}{2}\Delta t\lVert d_te_\gamma^{n,i}\rVert_{L^2(\Omega)}^2
			+(\bm{K}\nabla e_p^{n,i},\nabla e_p^{n,i})+k_4\Delta t(\bm{K}\nabla e_p^{n,i},d_t\nab e_\xi^{n+1,i})\no\\
			&+(\bm{\Theta}\nabla e_T^{n,i},\nabla e_T^{n,i})+k_1\Delta t(\bm{\Theta}\nabla e_T^{n,i},d_t\nab e_\xi^{n+1,i})\no\\
			&\leq(\nab e_T^{n,i}\cdot\mathcal{N}(\bm{K} \nabla p_h^{n,i-1}) ,e_T^{n+1,i})+(\nab e_T^{n,i-1}\cdot\mathcal{N}(\bm{K} \nabla p_h^{n,i}) ,e_T^{n+1,i})\no\\
			&	+(\mathcal{N}(\nab T_h^{n,i-2})\cdot\bm{K}\nab e_p^{n,i},e_T^{n+1,i})
			-(\nabla e_T^{n,i-1}\cdot\mathcal{N}( \bm{K} \nabla p_h^{n,i-2}),e_T^{n+1,i}).
		\end{alignat}
		Applying the summation operator $\Delta t\sum_{n=0}^l$ to the both sides of (\ref {115.4}), we get
		\begin{alignat}{2}\label{116.4}
			&	\frac{\mu}{2}\lVert\varepsilon(\be_\bu^{l+1,i})\rVert_{L^2(\Omega)}^2+\frac{k_6}{2}\lVert e_\xi^{l+1,i}\rVert_{L^2(\Omega)}^2
			+\frac{k_5-k_2}{2}\lVert e_\eta^{l,i}\rVert_{L^2(\Omega)}^2+\frac{k_3-k_2}{2}\lVert e_\gamma^{l,i}\rVert_{L^2(\Omega)}^2\\
			&+\Delta t\sum_{n=0}^l[\frac{\mu}{2}\Delta t\lVert\varepsilon(d_t\be_\bu^{n+1,i})\rVert_{L^2(\Omega)}^2+\frac{k_6}{2}\Delta t\lVert d_te_\xi^{n+1,i}\rVert_{L^2(\Omega)}^2+\frac{k_5-k_2}{2}\Delta t\lVert d_te_\eta^{n,i}\rVert_{L^2(\Omega)}^2\no\\
			&+\frac{k_3-k_2}{2}\Delta t\lVert d_te_\gamma^{n,i}\rVert_{L^2(\Omega)}^2]+\Delta t\sum_{n=0}^l
			[k_m\lVert\nabla e_p^{n,i}\rVert_{L^2(\Omega)}^2+\theta_m\lVert \nab e_T^{n,i}\rVert_{L^2(\Omega)}^2]\no\\
			&\leq\Delta t\sum_{n=0}^l[-k_4\Delta t(\bm{K}\nabla e_p^{n,i},d_t\nab e_\xi^{n+1,i})-k_1\Delta t(\bm{\Theta}\nabla e_T^{n,i},d_t\nab e_\xi^{n+1,i})+
			\frac{N^2}{2\epsilon_{31}}\lVert\nabla e_T^{n,i}\rVert_{L^2(\Omega)}^2\no\\
			&
			+\frac{\epsilon_{31}}{2}\lVert e_T^{n+1,i}\rVert_{L^2(\Omega)}^2+N^2\lVert\nabla e_T^{n,i-1}\rVert_{L^2(\Omega)}^2+\lVert e_T^{n+1,i}\rVert_{L^2(\Omega)}^2+\frac{N^2}{2\epsilon_{32}}\lVert\nabla e_p^{n,i}\rVert_{L^2(\Omega)}^2+\frac{\epsilon_{32}}{2}\lVert e_T^{n+1,i}\rVert_{L^2(\Omega)}^2].
			\no
		\end{alignat}
		Using Cauchy-Schwarz inequality, Young inequality and (\ref{eq-3-111}), we get
		\begin{alignat}{2}\label{117.5}
			k_4\Delta t(\bm{K}\nab e_p^{n,i},d_t\nab e_\xi^{n+1,i})\leq\frac{k_4^2}{2\epsilon_{33}}\lVert\nab e_\xi^{n+1,i}-\nab e_\xi^{n,i}\rVert_{L^2(\Omega)}^2
			+\frac{\epsilon_{33}}{2}\lVert \bm{K}\nab e_p^{n,i}\rVert_{L^2(\Omega)}^2\no\\
			\leq\frac{c_1^2k_4^2}{2\epsilon_{33}h^2}\lVert e_\xi^{n+1,i}- e_\xi^{n,i}\rVert_{L^2(\Omega)}^2
			+\frac{\epsilon_{33}k_M^2}{2}\lVert \nab e_p^{n,i}\rVert_{L^2(\Omega)}^2,\\
			k_1\Delta t(\bm{\Theta}\nab e_T^{n,i},d_t\nab e_\xi^{n+1,i})\leq\frac{k_1^2}{2\epsilon_{34}}\lVert\nab e_\xi^{n+1,i}-\nab e_\xi^{n,i}\rVert_{L^2(\Omega)}^2
			+\frac{\epsilon_{34}}{2}\lVert \bm{\Theta}\nab e_T^{n,i}\rVert_{L^2(\Omega)}^2\no\\
			\leq\frac{c_1^2\kappa_1^2}{2\epsilon_{34}h^2}\lVert e_\xi^{n+1,i}- e_\xi^{n,i}\rVert_{L^2(\Omega)}^2
			+\frac{\epsilon_{34}\theta_M^2}{2}\lVert \nab e_T^{n,i}\rVert_{L^2(\Omega)}^2.\label{117.6}
		\end{alignat}
		To bound the first term on the right-hand side of (\ref{117.5}) and (\ref{117.6}), we use the inf-sup condition \reff{eq-3-1} and get
		\begin{alignat}{2}\label{118.6}
			&\|e_\xi^{n+1,i}-e_\xi^{n,i}\|_{L^{2}(\Omega)}\no\\
			&\leq\frac{1}{\beta_1}\sup_{\bv_h\in\bV_h}\frac{\left( \div\bv_h,e_\xi^{n+1,i}-e_\xi^{n,i}\right) }{\|\bv_h\|_{H^1(\Omega)}}\no\\
			&\leq\frac{1}{\beta_1}\sup_{\bv_h\in\bV_h}\frac{\left( \div\bv_h,e_\xi^{n+1,i}-e_\xi^{n,i}\right) }{\|\nabla\bv_h\|_{L^2(\Omega)}}\no\\
			&=\frac{1}{\beta_1}\sup_{\bv_h\in\bV_h}\frac{\mu\left( \varepsilon(\be_\bu^{n+1,i}),\varepsilon(\bv_h)\right) -\mu\left( \varepsilon(\be_\bu^{n,i}),\varepsilon(\bv_h)\right) }{\|\nabla\bv_h\|_{L^2(\Omega)}}\no\\
			&=\frac{\mu}{\beta_1}\sup_{\bv_h\in\bV_h}\frac{\left( \varepsilon(\be_\bu^{n+1,i})-\varepsilon(\be_\bu^{n,i}),\varepsilon(\bv_h)\right) }{\|\nabla\bv_h\|_{L^2(\Omega)}}\no\\
			&=\frac{\mu\Delta t}{\beta_1}\sup_{\bv_h\in\bV_h}\frac{\left( d_t\varepsilon(\be_\bu^{n+1,i}),\varepsilon(\bv_h)\right) }{\|\nabla\bv_h\|_{L^2(\Omega)}}\no\\
			&\leq\frac{\mu\Delta t}{\beta_1}\|d_t\varepsilon(\be_\bu^{n+1,i})\|_{L^2(\Omega)}.
		\end{alignat}
		Substituting (\ref{118.6}) into  (\ref{117.5}) and (\ref{117.6}), we have
		\begin{alignat}{2}\label{119.6}
			k_4\Delta t(K\nab e_p^{n,i},d_t \nab e_\xi^{n+1,i})\leq\frac{c_1^2k_4^2\mu^2\Delta t^2}{2\epsilon_{33}h^2\beta_1^2}\lVert\varepsilon(d_t\be_{\bu }^{n+1,i})\rVert_{L^2(\Omega)}^2+\frac{\epsilon_{33}\kappa_M^2}{2}\lVert\nab e_p^{n,i}\rVert_{L^2(\Omega)}^2,\\\label{120.4}
			k_1\Delta t(\Theta\nab e_T^{n,i},d_t\nab e_\xi^{n+1,i})\leq\frac{c_1^2k_1^2\mu^2\Delta t^2}{2\epsilon_{34}h^2\beta_1^2}\lVert\varepsilon(d_t\be_{\bu }^{n+1,i})\rVert_{L^2(\Omega)}^2+\frac{\epsilon_{34}\theta_M^2}{2}\lVert\nab e_T^{n,i}\rVert_{L^2(\Omega)}^2.
		\end{alignat}
		Let $\epsilon_{31}=\frac{2N^2}{\theta_m}$, $\epsilon_{32}=\frac{2N^2}{k_m}$, $\epsilon_{33}=\frac{k_m}{2k_M^2}$ and $\epsilon_{34}=\frac{\theta_m}{2\theta_M^2}$ , combining (\ref{119.6}) and (\ref{120.4}) with (\ref{116.4}), applying  Gronwall's inequality, we get
		\begin{alignat}{2}\label{121.5}
			\frac{\mu}{2}\lVert\varepsilon(\be_\bu^{l+1,i})\rVert_{L^2(\Omega)}^2+\frac{k_6}{2}\lVert e_\xi^{l+1,i}\rVert_{L^2(\Omega)}^2
			+\frac{k_5-k_2}{2}\lVert e_\eta^{l,i}\rVert_{L^2(\Omega)}^2+\frac{k_3-k_2}{2}\lVert e_\gamma^{l,i}\rVert_{L^2(\Omega)}^2\no\\
			+\Delta t\sum_{n=0}^l[(\frac{\mu}{2}\Delta t-\frac{c_1^2k_4^2\mu^2\Delta t^2k_M^2}{h^2\beta_1^2k_m}
			-\frac{c_1^2k_1^2\mu^2\Delta t^2\theta_M^2}{h^2\beta_1^2\theta_m})
			\lVert\varepsilon(d_t\be_\bu^{n+1,i})\rVert_{L^2(\Omega)}^2~~~~~~~~~~~~~~\no\\
			+\frac{k_6}{2}\Delta t\lVert d_te_\xi^{n+1,i}\rVert_{L^2(\Omega)}^2+\frac{k_5-k_2}{2}\Delta t\lVert d_te_\eta^{n,i}\rVert_{L^2(\Omega)}^2
			+\frac{k_3-k_2}{2}\Delta t\lVert d_te_\gamma^{n,i}\rVert_{L^2(\Omega)}^2]\no\\
			+\Delta t\sum_{n=0}^l
			[\frac{k_m}{2}\lVert\nabla e_p^{n,i}\rVert_{L^2(\Omega)}^2+\frac{\theta_m}{2}\lVert \nabla e_T^{n,i}\rVert_{L^2(\Omega)}^2]
			\leq C\Delta t\sum_{n=0}^l\lVert\nab e_T^{n,i-1}\rVert_{L^2(\Omega)}^2,
		\end{alignat}
		where $C$ is a positive constant, we deduce that (\ref{121.5}) holds if
		$\Delta t<\frac{h^2\beta_1^2k_m\theta_m}{2c_1^2\mu(k_4^2k_M^2\theta_m+k_1^2\theta_M^2k_m)}$.\\
		
		Set $\bv_h=d_t\be_u^{n+1,i}$ in (\ref{101.4})(after using operator $d_t$), $\varphi_{h}=d_te_\xi^{n+1,i}$ in (\ref{101.5})(after using operator $d_t$), $y_h=d_te_p^{n+1,i}$ in  (\ref{101.6})
		and $z_h=d_te_T^{n+1,i}$,  after lowing the super-index from $n+1$ to $n$ on the both sides of (\ref{101.6}) and (\ref{101.7}), we yield
		\begin{alignat}{2}\label{122.4}
			&\mu\lVert\varepsilon(d_t\be_\bu^{n+1,i})\rVert^2+k_6 \lVert d_te_\xi^{n+1,i}\rVert^2+k_5\lVert d_te_\eta^{n,i}\rVert^2\lVert +k_3\lVert d_te_\gamma^{n,i}\rVert^2\
			+k_2(d_te_\eta^{n,i},d_te_\gamma^{n,i})\no\\	
			&	+k_2(d_te_\gamma^{n,i},d_te_\eta^{n,i})
			+(K\nabla e_p^{n,i},d_t\nabla e_p^{n+1,i})
			+(\Theta\nabla e_T^{n,i},d_t\nabla e_T^{n+1,i})~~~~~~~~~~~~~\no\\
			&	=(\nab e_T^{n,i}\cdot\mathcal{N}(\bm{K} \nabla p_h^{n,i-1}) ,d_te_T^{n+1,i})+(\nab e_T^{n,i-1}\cdot\mathcal{N}(\bm{K} \nabla p_h^{n,i}) ,d_te_T^{n+1,i})\no\\
			&	+(\mathcal{N}(\nab T_h^{n,i-2})\cdot\bm{K}\nab e_p^{n,i},d_te_T^{n+1,i})
			-(\nabla e_T^{n,i-1}\cdot\mathcal{N}( \bm{K} \nabla p_h^{n,i-2}),d_te_T^{n+1,i}).
		\end{alignat}
		Moreover, it is easy to check that
		\begin{alignat}{2}\label{122.5}
			(\bm{K}\nabla e_p^{n,i},d_t\nabla e_p^{n+1,i})=(\bm{K}\nabla e_p^{n,i},d_t\nabla e_p^{n,i})+k_4\Delta t(\bm{K}\nabla e_p^{n,i},d_t^2\nab e_{\xi}^{n+1,i}),\\\label{122.6}
			(\bm{\Theta}\nabla e_T^{n,i},d_t\nabla e_T^{n+1,i})=(\bm{\Theta}\nabla e_T^{n,i},d_t\nabla e_T^{n,i})+k_1\Delta t(\bm{\Theta}\nabla e_T^{n,i},d_t^2\nab e_{\xi}^{n+1,i}).
		\end{alignat}
		Adding (\ref{122.5})--(\ref{122.6}) and (\ref{122.4}), we get
		\begin{alignat}{2}\label{122.114}
			&\mu\lVert\varepsilon(d_t\be_\bu^{n+1,i})\rVert^2+k_6 \lVert d_te_\xi^{n+1,i}\rVert^2+k_5\lVert d_te_\eta^{n,i}\rVert^2\lVert +k_3\lVert d_te_\gamma^{n,i}\rVert^2\
			+k_2(d_te_\eta^{n,i},d_te_\gamma^{n,i})\no\\	
			&	+k_2(d_te_\gamma^{n,i},d_te_\eta^{n,i})
			+(\bm{K}\nabla e_p^{n,i},d_t\nabla e_p^{n,i})+k_4\Delta t(\bm{K}\nabla e_p^{n,i},d_t^2\nab \xi_h^{n+1,i})\no\\
			&	+(\bm{\Theta}\nabla e_T^{n,i},d_t\nabla e_T^{n,i})+k_1\Delta t(\bm{\Theta}\nabla e_T^{n,i},d_t^2\nab \xi_h^{n+1,i})\no\\
			&	=(\nab e_T^{n,i}\cdot\mathcal{N}(\bm{K} \nabla p_h^{n,i-1}) ,d_te_T^{n+1,i})+(\nab e_T^{n,i-1}\cdot\mathcal{N}(\bm{K} \nabla p_h^{n,i}) ,d_te_T^{n+1,i})\no\\
			&	+(\mathcal{N}(\nab T_h^{n,i-2})\cdot\bm{K}\nab e_p^{n,i},d_te_T^{n+1,i})
			-(\nabla e_T^{n,i-1}\cdot\mathcal{N}( \bm{K} \nabla p_h^{n,i-2}),d_te_T^{n+1,i}).
		\end{alignat}
		Applying the summation operator $\Delta t\sum_{n=0}^l$ to the both sides of (\ref {122.114}), using Cauchy-Schwarz inequality, Young inequality and (\ref{100.3}), we get
		\begin{alignat}{2}\label{123.4}
			&	\Delta t\sum_{n=0}^l[\mu \lVert\varepsilon(d_t\be_{\bu}^{n+1,i})\rVert_{L^2(\Omega)}^2+k_6\lVert d_te_\xi^{n+1,i}\rVert_{L^2(\Omega)}^2+(k_5-k_2)\lVert d_te_\eta^{n,i}\rVert_{L^2(\Omega)}^2+(k_3-k_2)\lVert d_te_\gamma^{n,i}\rVert_{L^2(\Omega)}^2\no\\
			&+\frac{k_m\Delta t}{2}\lVert d_t\nab e_p^{n,i}\rVert_{L^2(\Omega)}^2+\frac{\theta_m\Delta t}{2}\lVert d_t\nab e_T^{n,i}\rVert_{L^2(\Omega)}^2]+\frac{k_m}{2}\lVert \nab e_p^{n,i}\rVert_{L^2(\Omega)}^2+\frac{\theta_m}{2}\lVert \nab e_T^{n,i}\rVert_{L^2(\Omega)}^2\no\\
			&\leq\Delta t\sum_{n=0}^l[-k_4\Delta t(\bm{K}\nab e_p^{n,i},d_t^2\nab e_\xi^{n+1,i})
			-k_1\Delta t(\Theta\nab e_T^{n,i},d_t^2\nab e_\xi^{n+1,i})\no\\
			&	+(\frac{N^2}{2\epsilon_{35}}
			+\frac{N^2}{2\epsilon_{15}}+\frac{N^2}{2\epsilon_{16}})\lVert\nabla e_T^{n,i}\rVert_{L^2(\Omega)}^2+(\frac{k_M^2N^2}{2\epsilon_{35}}+\frac{k_M^2N^2}{2\epsilon_{15}}+\frac{k_M^2N^2}{2\epsilon_{16}})\lVert\nabla e_p^{n,i}\rVert_{L^2(\Omega)}^2\no\\
			&	+(\frac{N^2}{\epsilon_{35}}+\frac{N^2}{\epsilon_{15}}+\frac{N^2}{\epsilon_{16}})\lVert\nabla e_T^{n,i-1}\rVert_{L^2(\Omega)}^2+2k_1^2\epsilon_{35}\lVert d_te_\xi^{n+1,i}\rVert_{L^2(\Omega)}^2\no\\
			&+2k_2^2\epsilon_{15}\lVert d_te_\eta^{n,i}\rVert_{L^2(\Omega)}^2+
			2k_3^2\epsilon_{16}\lVert d_te_\gamma^{n,i}\rVert_{L^2(\Omega)}^2.
		\end{alignat}
		Moreover, it is easy to check that
		\begin{alignat}{2}\label{123.5}
			&\sum_{n=0}^l(\bm{K}\nab e_p^{n,i},d_t^2\nab e_\xi^{n+1,i})=\sum_{n=0}^l(\bm{K}\nab e_p^{n,i},\frac{d_t\nab e_\xi^{n+1,i}-d_t\nab e_\xi^{n,i}}{\Delta t})\no\\
			&	=\frac{1}{\Delta t}\sum_{n=0}^l(\bm{K}\nab e_p^{n,i},d_t\nab e_\xi^{n+1,i})-\frac{1}{\Delta t}\sum_{n=0}^l(\bm{K}\nab e_p^{n,i}-\bm{K}\nab e_p^{n-1,i}+\bm{K}\nab e_p^{n-1,i},d_t\nab e_\xi^{n,i})\no\\
			&=\frac{1}{\Delta t}\sum_{n=0}^l(\bm{K}\nab e_p^{n,i},d_t\nab e_\xi^{n+1,i})-\frac{1}{\Delta t}\sum_{n=0}^l(\bm{K}\nab e_p^{n-1,i},d_t\nab e_\xi^{n,i})-\sum_{n=0}^l(\bm{K}d_t\nab e_p^{n,i},d_t\nab e_\xi^{n,i})\no\\
			&=\frac{1}{\Delta t}(\bm{K}\nab e_p^{l,i},d_t\nab e_\xi^{l+1,i})-\sum_{n=0}^l(\bm{K}d_t\nab e_p^{n,i},d_t\nab e_\xi^{n,i}).
		\end{alignat}
		It is easy to check that
		\begin{alignat}{2}\label{123.6}
			\sum_{n=0}^l(\bm{\Theta}\nab e_T^{n,i},d_t^2\nab e_\xi^{n+1,i})=\frac{1}{\Delta t}(\bm{\Theta}\nab e_T^{l,i},d_t\nab e_\xi^{l+1,i})-\sum_{n=0}^l(\bm{\Theta} d_t\nab e_T^{n,i},d_t\nab e_\xi^{n,i}).
		\end{alignat}
		So, we get 
		\begin{alignat}{2}\label{128.5}
			\Delta t\sum_{n=0}^l k_4\Delta t(\bm{K}\nab e_p^{n,i},d_t^2\nab e_\xi^{n+1,i})
			&	=\Delta t k_4(\bm{K}\nab e_p^{l,i},d_t\nab e_\xi^{l+1,i})
			-\Delta t\sum_{n=0}^l\Delta tk_4(\bm{K}d_t\nab e_p^{n,i},d_t\nab e_\xi^{n,i})\no\\
			&\leq \lVert\sqrt{\Delta t}k_4\bm{K}\nab e_p^{l,i}\rVert_{L^2(\Omega)}\lVert\sqrt{\Delta t}d_t\nab e_\xi^{l+1,i}\rVert_{L^2(\Omega)}+\Delta t\sum_{n=0}^l\lVert\Delta tk_4\bm{K}d_t\nab e_p^{n,i}\rVert_{L^2(\Omega)}\no\\
			&\lVert d_t\nab e_\xi^{n,i}\rVert_{L^2(\Omega)}\leq\frac{k_4^2\epsilon\Delta tk_M^2}{2}\lVert\nab e_p^{l,i}\rVert_{L^2(\Omega)}^2+\frac{\Delta t}{2\epsilon}\lVert d_t\nab e_\xi^{l+1,i}\rVert_{L^2(\Omega)}^2\no\\
			&+\Delta t\sum_{n=0}^l[\frac{k_4^2\epsilon\Delta t^2k_M^2}{2}\lVert d_t\nab e_p^{n,i}\rVert_{L^2(\Omega)}^2+\frac{1}{2\epsilon}\lVert d_t\nab e_\xi^{n,i}\rVert_{L^2(\Omega)}^2]\no\\
			&\leq\frac{k_4^2\epsilon k_M^2\Delta t}{2}\lVert\nab e_p^{l,i}\rVert_{L^2(\Omega)}^2+\Delta t\sum_{n=0}^l[\frac{k_4^2\epsilon k_M^2}{2}\lVert \nab e_p^{n,i}\rVert_{L^2(\Omega)}^2\no\\
			&+\frac{k_4^2\epsilon k_M^2}{2}\lVert \nab e_p^{n-1,i}\rVert_{L^2(\Omega)}^2]
			+\Delta t\sum_{n=0}^l\frac{1}{2\epsilon}\lVert d_t\nab e_\xi^{n+1,i}\rVert_{L^2(\Omega)}^2\no\\
			&\leq\Delta t\sum_{n=0}^l[k_4^2\epsilon k_M^2\lVert\nab e_p^{n,i}\rVert_{L^2(\Omega)}^2+\frac{1}{2\epsilon}\lVert d_t\nab e_\xi^{n+1,i}\rVert_{L^2(\Omega)}^2]\no\\
			&\leq\Delta t\sum_{n=0}^l[k_4^2\epsilon k_M^2\lVert\nab e_p^{n,i}\rVert_{L^2(\Omega)}^2+\frac{1}{2\epsilon}\cdot\frac{c_{1}^2}{h^2}\lVert d_t e_\xi^{n+1,i}\rVert_{L^2(\Omega)}^2].
		\end{alignat}
		Similarly, we have
		\begin{alignat}{2}\label{128.6}
			\Delta t\sum_{n=0}^l k_1\Delta t(\Theta\nab e_T^{n,i},d_t^2\nab e_\xi^{n+1,i})\leq\Delta t\sum_{n=0}^l[k_1^2\epsilon\theta_M^2\lVert\nab e_T^{n,i}\rVert_{L^2(\Omega)}^2+\frac{1}{2\epsilon}\cdot\frac{c_{1}^2}{h^2}\lVert d_t e_\xi^{n+1,i}\rVert_{L^2(\Omega)}^2].
		\end{alignat}
		Let $\epsilon_{35}=\frac{k_6}{8k_1^2}$, $\epsilon_{15}=\frac{k_5-k_2}{4k_2^2}$, $\epsilon_{16}=\frac{k_3-k_2}{4k_3^2}$ and $\epsilon=\frac{4c_{1}^2}{h^2k_6}$,
		using 
	(\ref{123.4})--(\ref{128.6}), applying the Gronwall's inequality, we get
		\begin{alignat}{2}\label{129.4}
			&	\Delta t\sum_{n=0}^l[\mu\lVert\varepsilon(d_t\be_\bu^{n+1,i})\rVert_{L^2(\Omega)}^2+\frac{k_6}{2} \lVert d_te_\xi^{n+1,i}\rVert_{L^2(\Omega)}^2+\frac{k_5-k_2}{2} \lVert d_te_\eta^{n,i}\rVert_{L^2(\Omega)}^2 \\
			&+\frac{k_3-k_2}{2} \lVert d_te_\gamma^{n,i}\rVert_{L^2(\Omega)}^2]	
			+\Delta t\sum_{n=1}^l[\frac{k_m}{2}\Delta t\lVert d_t\nabla e_p^{n,i}\rVert_{L^2(\Omega)}^2+\frac{\theta_m}{2}\Delta t\lVert \nabla e_T^{n,i}\rVert_{L^2(\Omega)}^2]\no\\
			&	+\frac{k_m}{2}\lVert \nabla e_T^{l,i}\rVert_{L^2(\Omega)}^2
			+\frac{\theta_m}{2}\lVert \nabla e_T^{l,i}\rVert_{L^2(\Omega)}^2\no\\
			&	\leq
			C \Delta t\sum_{n=0}^{l}\lVert\nabla e_T^{n,i-1}\rVert_{L^2(\Omega)}^2\no\\
			&\leq
			C \Delta t\sum_{n=0}^{l_1}\lVert\nabla e_T^{n,i-1}\rVert_{L^2(\Omega)}^2\no.
		\end{alignat}
		It is true for any $l\leq l_1$,($l_1>0$ will be fixed below),  by $\Delta t\sum_{n=0}^{l_1}$ again, we can get
		\begin{alignat}{2}\label{139.4}
			\Delta t\sum_{n=0}^{l_1}\left\|e_{T}^{n,i}\right\|_{H^{1}(\Omega)}^{2} \leq l_{1} C \Delta t^2\sum_{n=0}^{l_1}\left\|e_{T}^{n,i-1}\right\|_{H^{1}(\Omega)}^{2}. 
		\end{alignat}
		Let $\Delta t=\frac{1}{2l_1C}$, this shows a contraction of the residuals from the Banach Fixed Point Theorem and therefore completes the proof.
	\end{proof}

	To derive the optimal order error estimates of the fully discrete multiphysics finite element method for any $\varphi \in L^{2}(\Omega)$, we firstly define $L^{2}(\Omega)$-projection operators $\mathcal{Q}_{h}: L^{2}(\Omega)\rightarrow X^{k}_{h}$  by
	\begin{eqnarray}
		(\mathcal{Q}_{h}\varphi,\psi_{h})=(\varphi,\psi_{h})~~~~~\psi_{h}\in X^{k}_{h},\label{eq210607-1}
	\end{eqnarray}
	where $X^{k}_{h}:=\{\psi_{h}\in C^{0};~\psi_{h}|_{E}\in P_{k}(E)~\forall E\in\mathcal{T}_{h}\}$, $k$ is the degree of piecewise polynomial on $E$.
	
	Next, for any $\varphi\in H^{1}(\Omega)$, we define its elliptic projection $\mathcal{S}_{h}: H^{1}(\Omega)\rightarrow X^{k}_{h}$ by
	\begin{align}
		(\bm{K}\nabla \mathcal{S}_{h}\varphi,&\nabla\varphi_{h})=(\bm{K}\nabla\varphi,\nabla\varphi_{h})~~~~~\forall \varphi_{h}\in X^{k}_{h},\\
		&(\mathcal{S}_{h}\varphi,1)=(\varphi,1).
	\end{align}
	Finally, for any $\textbf{v}\in\textbf{H}^{1}(\Omega)$, we define its elliptic projection $\mathcal{R}_{h}:\textbf{H}^{1}(\Omega)\rightarrow\textbf{V}^{k}_{h}$ by
	\begin{eqnarray}
		(\varepsilon(\mathcal{R}_{h}\textbf{v}),\varepsilon(\textbf{w}_{h}))=(\varepsilon(\textbf{v}),\varepsilon(\textbf{w}_{h}))~~~~~
		\forall\textbf{w}_{h}\in \textbf{V}^{k}_{h},
	\end{eqnarray}
	where $\textbf{V}^k_{h}:=\{\textbf{v}_{h} \in \textbf{C}^{0};~\textbf{v}_{h}|_{E}\in \textbf{P}_{k}(E), (\textbf{v}_{h}, \textbf{r})=0 ~\forall \textbf{r} \in \textbf{RM}\}$.
	From \cite{Quarteroni1997}, we know that  $\mathcal{Q}_{h},\mathcal{S}_{h}$ and $\mathcal{R}_{h}$ satisfy
	\begin{align}
		\|\mathcal{Q}_{h}\varphi-\varphi\|_{L^{2}(\Omega)}+&h\|\nabla(\mathcal{Q}_{h}\varphi-\varphi)\|_{L^{2}(\Omega)}\\\nonumber
		&\leq Ch^{s+1}\|\varphi\|_{H^{s+1}(\Omega)}~~\forall \varphi\in H^{s+1}(\Omega),~~  0\leq s\leq k,\label{eq-3-47}\\	\|\mathcal{S}_{h}\varphi-\varphi\|_{L^{2}(\Omega)}+&h\|\nabla(\mathcal{S}_{h}\varphi-\varphi)\|_{L^{2}(\Omega)}\\\nonumber	&\leq Ch^{s+1}\|\varphi\|_{H^{s+1}(\Omega)}~~\forall \varphi\in H^{s+1}(\Omega),~~  0\leq s\leq k,\label{eq-3-47-2}\\
		\|\mathcal{R}_{h}\textbf{v}-\textbf{v}\|_{L^{2}(\Omega)}+&h\|\nabla(\mathcal{R}_{h}\textbf{v}-\textbf{v})\|_{L^{2}(\Omega)}\\\nonumber
		&\leq Ch^{s+1}\|\textbf{v}\|_{H^{s+1}(\Omega)}~~\forall \textbf{v}\in \textbf{H}^{s+1}(\Omega),~~  0\leq s\leq k.\label{eq-3-47-3}
	\end{align}
	To derive error estimates, we introduce the following notations:
	\begin{eqnarray*}
		E^{n}_{\textbf{u}}:=\textbf{u}(t_{n})-\textbf{u}^{n}_{h},~~~~ E^{n}_{\xi}:=\xi(t_{n})-\xi^{n}_{h},~~~~E^{n}_{\eta}:=\eta(t_{n})-\eta^{n}_{h},\\
		E^{n}_{p}:=p(t_{n})-p^{n}_{h},~~~~E^{n}_{\gamma}:=\gamma(t_{n})-\gamma^{n}_{h},~~~~E^{n}_{q}:=q(t_{n})-q^{n}_{h}.
	\end{eqnarray*}
	It is easy to check that
	\begin{eqnarray}
		&&E^{n}_{p}=k_{4}E^{n}_{\xi}+k_{5}E^{n}_{\eta}+k_{2}E^{n}_{\gamma},~~
		E^{n}_{T}=k_{1}E^{n}_{\xi}+k_{2}E^{n}_{\eta}+k_{3}E^{n}_{\gamma},\nonumber\\
		&&E^{n}_{q}=-k_{6}E^{n}_{\xi}+k_{4}E^{n}_{\eta}+k_{1}E^{n}_{\gamma},
		\widehat{E}^{n+1}_{p}=k_{4}E^{n+1}_{\xi}+k_{5}E^{n+\theta}_{\eta}+k_{2}E^{n+\theta}_{\gamma},\nonumber\\
	&&\widehat{E}^{n+1}_{T}=k_{1}E^{n+1}_{\xi}+k_{2}E^{n+\theta}_{\eta}+k_{3}E^{n+\theta}_{\gamma}.
	\end{eqnarray}
	Also, we denote
	\begin{eqnarray*}
		&&E^{n}_{\textbf{u}}=\textbf{u}(t_{n})-\mathcal{R}_{h}(\textbf{u}(t_{n}))+\mathcal{R}_{h}(\textbf{u}(t_{n}))-\textbf{u}^{n}_{h}:=\varLambda^{n}_{\textbf{u}}+\varPi^{n}_{\textbf{u}},\\
		&&E^{n}_{\xi}=\xi(t_{n})-\mathcal{Q}_{h}(\xi(t_{n}))+\mathcal{Q}_{h}(\xi(t_{n}))-\xi^{n}_{h}:=\Phi^{n}_{\xi}+\Psi^{n}_{\xi},\\
		&&E^{n}_{\xi}=\xi(t_{n})-\mathcal{S}_{h}(\xi(t_{n}))+\mathcal{S}_{h}(\xi(t_{n}))-\xi^{n}_{h}:=\varLambda^{n}_{\xi}+\varPi^{n}_{\xi},\\
		&&E^{n}_{\eta}=\eta(t_{n})-\mathcal{Q}_{h}(\eta(t_{n}))+\mathcal{Q}_{h}(\eta(t_{n}))-\eta^{n}_{h}:=\Phi^{n}_{\eta}+\Psi^{n}_{\eta},\\
		&&E^{n}_{\gamma}:=\gamma(t_{n})-\mathcal{Q}_{h}(\gamma(t_{n}))+\mathcal{Q}_{h}(\gamma(t_{n}))-\gamma^{n}_{h}:=\Phi^{n}_{\gamma}+\Psi^{n}_{\gamma},\\	
		&&E^{n}_{p}=p(t_{n})-\mathcal{Q}_{h}(p(t_{n}))+\mathcal{Q}_{h}(p(t_{n}))-p^{n}_{h}:=\Phi^{n}_{p}+\Psi^{n}_{p},\\
		&&E^{n}_{p}=p(t_{n})-\mathcal{S}_{h}(p(t_{n}))+\mathcal{S}_{h}(p(t_{n}))-p^{n}_{h}:=\varLambda^{n}_{p}+\varPi^{n}_{p},\\
		&&E^{n}_{T}=T(t_{n})-\mathcal{Q}_{h}(T(t_{n}))+\mathcal{Q}_{h}(T(t_{n}))-T^{n}_{h}:=\Phi^{n}_{T}+\Psi^{n}_{T},\\
		&&E^{n}_{T}=T(t_{n})-\mathcal{S}_{h}(T(t_{n}))+\mathcal{S}_{h}(T(t_{n}))-T^{n}_{h}:=\varLambda^{n}_{T}+\varPi^{n}_{T}.
	\end{eqnarray*}
	\begin{lemma}\label{lem-3-4}
	{\rm	Let ${(\bm{{\rm u}}^{n}_{h},\xi^{n}_{h},\eta^{n}_{h},\gamma^{n}_{h})}$ be generated by the MFEM, then we have
		\begin{eqnarray}\label{eq-3-54}
			&&\varSigma^{l}_{h}+\Delta t\sum_{n=0}^{l}[(\bm{K}\nabla\widehat{\varPi}^{n+1}_{p},\nabla\widehat{\varPi}^{n+1}_{p})
			+(\bm{\Theta}\nabla\widehat{\varPi}^{n+1}_{T},\nabla\widehat{\varPi}^{n+1}_{T})
			+\frac{\Delta t}{2}(\mu\|d_{t}\varepsilon(\varPi^{n+1}_{\bm{{\rm u}}})\|^2_{L^{2}(\Omega)}\nonumber\\
			&&+k_{6}\|d_{t}\Psi^{n+1}_{\xi}\|^2_{L^{2}(\Omega)}
			+(k_{5}-k_{2})\|d_{t}\Psi^{n+\theta}_{\eta}\|^2_{L^{2}(\Omega)}+(k_{3}-k_{2})\|d_{t}\Psi^{n+\theta}_{\gamma}\|^2_{L^{2}(\Omega)})]\nonumber\\
			&&\leq \widehat{\varSigma}^{-1}_{h}+\Delta t\sum_{n=0}^{l} [(\Phi^{n+1}_{\xi},\operatorname{div} d_{t}\varPi^{n+1}_{\bm{{\rm u}}})-(\operatorname{div} d_{t}\varLambda^{n+1}_{\bm{{\rm u}}},\Psi^{n+1}_{\xi})]\nonumber\\
			&&+\Delta t\sum_{n=0}^{l}[(d_{t}\Psi^{n+\theta}_{\eta},\widehat{\varLambda}^{n+1}_{p}-\widehat{\Phi}^{n+1}_{p})+(d_{t}\Psi^{n+\theta}_{\gamma},\widehat{\varLambda}^{n+1}_{T}-\widehat{\Phi}^{n+1}_{T})]\nonumber\\
			&&+(1-\theta)(\Delta t)^{2}\sum_{n=0}^{l}[k_{4}(d^{2}_{t}\eta(t_{n+1}),\Psi^{n+1}_{\xi})
			+k_{1}(d^{2}_{t}\gamma(t_{n+1}),\Psi^{n+1}_{\xi})]\nonumber\\
			&&+(1-\theta)(\Delta t)^{2}\sum_{n=0}^{l}[k_{4}\Delta t(\bm{K}\nabla\widehat{\varPi}^{n+1}_{p}\,d_{t}\nabla\varPi^{n+1}_{\xi})
			+k_{1}\Delta t(\bm{\Theta}\nabla\widehat{\varPi}^{n+1}_{T}, d_{t}\nabla\varPi^{n+1}_{\xi})]\nonumber\\
			&&+\Delta t\sum_{n=0}^{l}[(\mathcal{N}(\nabla T^{n+\theta}_h)\cdot(\bm{K}\nabla E^{n+\theta}_{p}),\widehat{\varPi}^{n+1}_{T})+(\nabla E^{n+\theta}_{T}\cdot(\bm{K}\nabla p(t_{n+\theta})),\widehat{\varPi}^{n+1}_{T})]\nonumber\\
			&&+\Delta t\sum_{n=0}^{l}[(R^{n+\theta}_{\eta},\widehat{\varPi}^{n+1}_{p})+(R^{n+\theta}_{\gamma},\widehat{\varPi}^{n+1}_{T})],
		\end{eqnarray}
		where
		\begin{align}
			\widehat{\varPi}^{n+1}_{p}&:=k_{4} \varPi^{n+1}_{\xi}+k_{5}\varPi^{n+\theta}_{\eta}+k_{2} \varPi^{n+\theta}_{\gamma},\\
			\widehat{\varPi}^{n+1}_{T}&:=k_{1} \varPi^{n+1}_{\xi}+k_{2}\varPi^{n+\theta}_{\eta}+k_{3} \varPi^{n+\theta}_{\gamma},\\
			\varSigma^{l}_{h}:&=\frac{1}{2}(\mu\|\varepsilon(\varPi^{l+1}_{\bm{{\rm u}}})\|^2_{L^{2}(\Omega)}+k_{6}\|\Psi^{l+1}_{\xi}\|^2_{L^{2}(\Omega)}\\\nonumber
			&+(k_{5}-k_{2})\|\Psi^{l+\theta}_{\eta}\|^2_{L^{2}(\Omega)}+(k_{3}-k_{2})\|\Psi^{l+\theta}_{\gamma}\|^2_{L^{2}(\Omega)}),\\
			\widehat{\varSigma}^{-1}_{h}:&=\frac{1}{2}(\mu\|\varepsilon(\varPi^{0}_{\bm{{\rm u}}})\|^2_{L^{2}(\Omega)}+k_{6}\|\Psi^{0}_{\xi}\|^2_{L^{2}(\Omega)}\\\nonumber
			&+(k_{5}+k_{2})\|\Psi^{\theta-1}_{\eta}\|^2_{L^{2}(\Omega)}+(k_{3}+k_{2})\|\Psi^{\theta-1}_{\gamma}\|^2_{L^{2}(\Omega)}),\\
			R^{n+1}_{\eta}:&=-\frac{1}{\Delta t}\int_{t_{n}}^{t_{n+1}}(s-t_{n})\eta_{tt}(s)ds,~R^{n+1}_{\gamma}:=-\frac{1}{\Delta t}\int_{t_{n}}^{t_{n+1}}(s-t_{n})\gamma_{tt}(s)ds.
		\end{align}}
	\end{lemma}
	\begin{proof}
		Subtracting (\ref{eq-3-113}) form (\ref{eq-2-20}), (\ref{eq-3-114}) from (\ref{eq-2-21}), (\ref{eq-3-115}) from  (\ref{eq-2-22}), (\ref{eq-3-116}) from (\ref{eq-2-23-1}), respectively, we get the following error equations
		\begin{align}
			\mu(\varepsilon(E_{\textbf{u}}^{n+1}),\varepsilon(\textbf{v}_{h}))-(E_{\xi}^{n+1},\nabla\cdot \textbf{v}_{h})=0 \quad\quad \forall ~~\textbf{v}_{h}\in \textbf{V}_h,\label{eq-3-58}\\
			k_{6}(E_{\xi}^{n+1},\varphi_{h})+(\nabla\cdot E_{\textbf{u}}^{n+1},\varphi_{h})=k_{4}(E_{\eta}^{n+\theta},\varphi_{h})
			+(1-\theta)k_{4}\Delta t(d_{t}\eta(t_{n+1}),\varphi_{h})\nonumber\\
			+k_{1}(E^{n+\theta}_{\gamma},\varphi_{h})+(1-\theta)k_{1}\Delta t(d_{t}\gamma(t_{n+1}),\varphi_{h}) \quad\quad \forall~~ \varphi_{h}\in M_h,\label{eq-3-59}\\
			(d_{t}E_{\eta}^{n+\theta},y_{h})+(\bm{K}\nabla(\widehat{E}^{n+1}_{p}),\nabla y_{h})-(1-\theta)k_4\Delta t(\bm{K}d_t\nabla E^{n+1}_\xi,\nabla y_h)\nonumber\\
			=(R^{n+\theta}_{\eta},y_{h}) \quad\quad\forall~~y_{h}\in W_h,\label{eq-3-60}\\
			(d_{t}E_{\gamma}^{n+\theta},z_{h})+(\bm{\Theta}\nabla(\widehat{E}^{n+1}_{T}),\nabla z_{h})-(1-\theta)k_1\Delta t(\bm{\Theta}d_t\nabla E^{n+1}_\xi,\nabla z_h)\no\\
			-(\nabla E^{n+\theta}_{T}\cdot(\bm{K}\nabla p(t_{n+\theta})),z_{h})
			-(\nabla T^{n+\theta}_{h}\cdot(\bm{K}\nabla E^{n+\theta}_{p}),z_{h})=(R^{n+\theta}_{\gamma},z_{h})\quad \forall~~z_{h}\in Z_h.\label{eq-3-61}
		\end{align}
		
		Using the definition of the projection operators $\mathcal{Q}_{h},\mathcal{S}_{h},\mathcal{R}_{h}$ and the above error equations, we have
		\begin{align}
			\mu(\varepsilon(\varPi_{\textbf{u}}^{n+1}),\varepsilon(\textbf{v}_{h}))-(\Psi_{\xi}^{n+1},\nabla\cdot \textbf{v}_{h})=(\Phi_{\xi}^{n+1},\nabla\cdot \textbf{v}_{h}) \quad\quad\forall ~~\textbf{v}_{h}\in V_h,\label{eq-3-63}\\
			k_{6}(\Psi_{\xi}^{n+1},\varphi_{h})+(\nabla\cdot \varPi_{\textbf{u}}^{n+1},\varphi_{h})
			=-(\nabla\cdot \varLambda^{n+1}_{\textbf{u}},\varphi_{h})+k_{4}(\Psi_{\eta}^{n+\theta},\varphi_{h})\nonumber\\
			+(1-\theta)k_{4}\Delta t(d_{t}\eta(t_{n+1}),\varphi_{h})
			+k_{1}(\Psi^{n+\theta}_{\gamma},\varphi_{h})\no \\
			(1-\theta)k_{1}\Delta t(d_{t}\gamma(t_{n+1}),\varphi_{h})  \quad\quad\forall~~ \varphi_{h}\in M_h,\label{eq-3-64}\\
			(d_{t}\Psi_{\eta}^{n+\theta},y_{h})+(\bm{K}\nabla(\widehat{\varPi}^{n+1}_{p}),\nabla y_{h})-(1-\theta)k_4\Delta t(\bm{K}d_t\nabla \varPi^{n+1}_\xi,\nabla y_h)\no\\
			=(R^{n+\theta}_{\eta},y_{h}) \quad\quad\forall~~y_{h}\in W_h,\label{eq-3-65}\\
		(d_{t}\Psi_{\gamma}^{n+\theta},z_{h})+(\bm{\Theta}\nabla(\widehat{\varPi}^{n+1}_{T}),\nabla z_{h})-(1-\theta)k_1\Delta t(\bm{\Theta}d_t\nabla \varPi^{n+1}_\xi,\nabla z_h)\no\\
			-(\nabla E^{n+\theta}_{T}\cdot(\bm{K}\nabla p(t_{n+\theta})),z_{h})
			-(\nabla T^{n+\theta}_{h}\cdot(\bm{K}\nabla E^{n+\theta}_{p}),z_{h})\nonumber\\=(R^{n+\theta}_{\gamma},z_{h}) \quad\quad\forall~~z_{h}\in Z_h.\label{eq-3-66}
		\end{align}
	
		Taking $\textbf{v}_{h}=d_{t}\varPi^{n+1}_{\textbf{u}}$ in (\ref{eq-3-63}), $\varphi_{h}=\Psi^{n+1}_{\xi}$ (after applying the difference operator $d_{t}$ to the equation (\ref{eq-3-64})), $y_{h}=\widehat{\varPi}^{n+1}_{p}=\widehat{\Phi}^{n+1}_{p}-\widehat{\varLambda}^{n+1}_{p}+k_{4}\Psi^{n+1}_{\xi}+k_{5}\Psi^{n+\theta}_{\eta}+k_{2}\Psi^{n+\theta}_{\gamma}$ in (\ref{eq-3-65}) and $z_{h}=\widehat{\varPi}^{n+1}_{T}=\widehat{\Phi}^{n+1}_{T}-\widehat{\varLambda}^{n+1}_{T}+k_{1}\Psi^{n+1}_{\xi}+k_{2}\Psi^{n+\theta}_{\eta}+k_{3}\Psi^{n+\theta}_{\gamma}$ in (\ref{eq-3-66}), we get
		\begin{align}\label{eq-3-68}
			\mu(\varepsilon(\varPi^{n+1}_{\textbf{u}}),\varepsilon(d_{t}\varPi^{n+1}_{\textbf{u}}))+k_{6}(d_{t}\Psi^{n+1}_{\xi},\Psi^{n+1}_{\xi})+k_{5}(d_{t}\Psi^{n+\theta}_{\eta},\Psi^{n+\theta}_{\eta})\nonumber\\
			+k_{3}(d_{t}\Psi^{n+\theta}_{\gamma},\Psi^{n+\theta}_{\gamma})+k_{2}(d_{t}\Psi^{n+\theta}_{\eta},\Psi^{n+\theta}_{\gamma})+k_{2}(d_{t}\Psi^{n+\theta}_{\gamma},\Psi^{n+\theta}_{\eta})\nonumber\\
			+(\bm{K}\nabla\widehat{\varPi}^{n+1}_{p},\nabla\widehat{\varPi}^{n+1}_{p})+(\bm{\Theta}\nabla\widehat{\varPi}^{n+1}_{T},\nabla\widehat{\varPi}^{n+1}_{T})\no\\
			-(1-\theta)k_{4}\Delta t(\bm{K}d_t\nabla \varPi^{n+1}_{\xi},\nabla\widehat{\varPi}^{n+1}_{p})-(1-\theta)k_{1}\Delta t(\bm{\Theta}d_t\nabla \varPi^{n+1}_{\xi},\nabla\widehat{\varPi}^{n+1}_{T})\nonumber\\
		=(\Phi^{n+1}_{\xi},\operatorname{div} d_{t}\varPi^{n+1}_{\textbf{u}})-(\operatorname{div} d_{t}\varLambda^{n+1}_{\textbf{u}},\Psi^{n+1}_{\xi})+(d_{t}\Psi^{n+\theta}_{\eta},\widehat{\varLambda}^{n+1}_{p}-\widehat{\Phi}^{n+1}_{p})\nonumber\\
			+(d_{t}\Psi^{n+\theta}_{\gamma},\widehat{\varLambda}^{n+1}_{T}-\widehat{\Phi}^{n+1}_{T})+(1-\theta)k_{4}\Delta t(d^{2}_{t}\eta(t_{n+1}),\Psi^{n+1}_{\xi})\nonumber\\
			+(1-\theta)k_{1}\Delta t(d^{2}_{t}\gamma(t_{n+1}),\Psi^{n+1}_{\xi})+(\nabla E^{n+\theta}_{T}\cdot(\bm{K}\nabla p(t_{n+\theta})), \widehat{\varPi}^{n+1}_{T})\nonumber\\
			+(\nabla T^{n+\theta}_{h}\cdot(\bm{K}\nabla E^{n+ \theta}_{p}), \widehat{\varPi}^{n+1}_{T})
			+(R^{n+\theta}_{\eta},\widehat{\varPi}^{n+1}_{p})+(R^{n+\theta}_{\gamma},\widehat{\varPi}^{n+1}_{T}).
		\end{align}
		It is easy to check
		\begin{eqnarray}
			\mu(\varepsilon(\varPi^{n+1}_{\textbf{u}}),\varepsilon(d_{t}\varPi^{n+1}_{\textbf{u}}))=\frac{\mu}{2}(d_{t}\|\varepsilon(\varPi^{n+1}_{\textbf{u}})\|^{2}_{L^{2}(\Omega)}+\Delta t\|d_{t}\varepsilon(\varPi^{n+1}_{\textbf{u}})\|^{2}_{L^{2}(\Omega)}),\label{eq-3-76-1}\\
			k_{6}(d_{t}\Psi^{n+1}_{\xi},\Psi^{n+1}_{\xi})=\frac{k_{6}}{2}(d_{t}\|\Psi^{n+1}_{\xi}\|^{2}_{L^{2}(\Omega)}+\Delta t\|d_{t}\Psi^{n+1}_{\xi}\|^{2}_{L^{2}(\Omega)}),\\
			k_{5}(d_{t}\Psi^{n+\theta}_{\eta},\Psi^{n+\theta}_{\eta})=\frac{k_{5}}{2}(d_{t}\|\Psi^{n+\theta}_{\eta}\|^{2}_{L^{2}(\Omega)}+\Delta t\|d_{t}\Psi^{n+\theta}_{\eta}\|^{2}_{L^{2}(\Omega)}),\\
			k_{3}(d_{t}\Psi^{n+\theta}_{\gamma},\Psi^{n+\theta}_{\gamma})=\frac{k_{3}}{2}(d_{t}\|\Psi^{n+\theta}_{\gamma}\|^{2}_{L^{2}(\Omega)}+\Delta t\|d_{t}\Psi^{n+\theta}_{\gamma}\|^{2}_{L^{2}(\Omega)}),\\
			k_{2}(d_{t}\Psi^{n+\theta}_{\eta},\Psi^{n}_{\gamma})+k_{2}(d_{t}\Psi^{n+\theta}_{\gamma},\Psi^{n+\theta}_{\eta})\leq
			\frac{k_{2}}{2}(\Delta t\|d_{t}\Psi^{n+\theta}_{\eta}\|^{2}_{L^{2}(\Omega)}\nonumber\\\label{eq-3-80-1}+\Delta t\|d_{t}\Psi^{n+\theta}_{\gamma}\|^{2}_{L^{2}(\Omega)})+k_{2}d_{t}(\Psi^{n+\theta}_{\eta},\Psi^{n}_{\gamma}).
		\end{eqnarray}
		
		Substituting (\ref{eq-3-76-1})-(\ref{eq-3-80-1}) into (\ref{eq-3-68}), we have
		{\allowdisplaybreaks[2]
			\begin{align}
				&\frac{1}{2}(\mu d_{t}\|\varepsilon(\varPi^{n+1}_{\textbf{u}})\|^2_{L^{2}(\Omega)}+k_{6}d_{t}\|\Psi^{n+1}_{\xi}\|^2_{L^{2}(\Omega)}+k_{5}d_{t}\|\Psi^{n+\theta}_{\eta}\|^2_{L^{2}(\Omega)}+k_{3}d_{t}\|\Psi^{n+\theta}_{\gamma}\|^2_{L^{2}(\Omega)})\nonumber\\
				&+\frac{\Delta t}{2}(\mu\|d_{t}\varepsilon(\varPi^{n+1}_{\textbf{u}})\|^2_{L^{2}(\Omega)}
				+k_{6}\|d_{t}\Psi^{n+1}_{\xi}\|^2_{L^{2}(\Omega)}+(k_{5}-k_{2})\|d_{t}\Psi^{n+\theta}_{\eta}\|^2_{L^{2}(\Omega)}\nonumber\\
				&+(k_{3}-k_{2})\|d_{t}\Psi^{n+\theta}_{\gamma}\|^2_{L^{2}(\Omega)})+k_{2}d_{t}(\Psi^{n+\theta}_{\eta},\Psi^{n+\theta}_{\gamma})	+(\bm{K}\nabla\widehat{\varPi}^{n+1}_{p},\nabla\widehat{\varPi}^{n+1}_{p})\nonumber\\
				&+(\bm{\Theta}\nabla\widehat{\varPi}^{n+1}_{T},\nabla\widehat{\varPi}^{n+1}_{T})-(1-\theta)k_{4}\Delta t(\bm{K}\nabla\widehat{\varPi}^{n+1}_{p},d_{t}\nabla\varPi^{n+1}_{\xi})\no\\
				&-(1-\theta)k_{1}\Delta t(\bm{\Theta}\nabla\widehat{\varPi}^{n+1}_{T}, d_{t}\nabla\varPi^{n+1}_{\xi})
				\leq(\Phi^{n+1}_{\xi},\operatorname{div} d_{t}\varPi^{n+1}_{\textbf{u}})-(\operatorname{div} d_{t}\varLambda^{n+1}_{\textbf{u}},\Psi^{n+1}_{\xi})\nonumber\\
				&+(d_{t}\Psi^{n+\theta}_{\eta},\widehat{\varLambda}^{n+1}_{p}-\widehat{\Phi}^{n+1}_{p})
				+(d_{t}\Psi^{n+\theta}_{\gamma},\widehat{\varLambda}^{n+1}_{T}-\widehat{\Phi}^{n+1}_{T})
				+(1-\theta)\Delta t k_{4}(d^{2}_{t}\eta(t_{n+1}),\Psi^{n+1}_{\xi})\nonumber\\
				&+(1-\theta)\Delta t k_{1}(d^{2}_{t}\gamma(t_{n+1}),\Psi^{n+1}_{\xi})
				+(\nabla E^{n+\theta}_{T}\cdot(\bm{K}\nabla p(t_{n+\theta})), \widehat{\varPi}^{n+1}_{T})\nonumber\\
				&+(\mathcal{N}(\nabla T^{n+\theta}_{h})\cdot(\bm{K}\nabla E^{n+\theta}_{p}), \widehat{\varPi}^{n+1}_{T})
				+(R^{n+\theta}_{\eta},\widehat{\varPi}^{n+1}_{p})+(R^{n+\theta}_{\gamma},\widehat{\varPi}^{n+1}_{T}).\label{eq-3-80-11}
		\end{align}}
		Applying the summation operation $\Delta t\sum_{n=0}^{l}$ to both sides of (\ref{eq-3-80-11}), we implies that  (\ref{eq-3-54}) holds. The proof is complete.
	\end{proof}
	\begin{theorem}\label{th-3-5}
		Assume that $\Delta t=O(h^{2})$ when $\theta=0$ and $\Delta t>0$ when $\theta=1$, then there holds the error estimate 
		
		\begin{eqnarray}\label{eq-3-75}
			&&\max_{0\leq n\leq l}[\sqrt{\mu}\|\varepsilon(\varPi^{n+1}_{\bm{{\rm u}}})\|_{L^{2}(\Omega)}+\sqrt{k_{6}}\|\Psi^{n+1}_{\xi}\|_{L^{2}(\Omega)}+\sqrt{k_{5}-k_{2}}\|\Psi^{n+\theta}_{\eta}\|_{L^{2}(\Omega)}\\
			&&+\sqrt{k_{3}-k_{2}}\|\Psi^{n+\theta}_{\gamma}\|_{L^{2}(\Omega)}]
			+[\Delta t \sum_{n=0}^{l}\|\widehat{\varPi}^{n+1}_{p}\|^{2}_{L^{2}(\Omega)}+\|\widehat{\varPi}^{n+1}_{T}\|^{2}_{L^{2}(\Omega)}]^{\frac{1}{2}}\nonumber\\
			&&\leq C_{1}(\tau)\Delta t+C_{2}(\tau)h^{2},\nonumber
		\end{eqnarray}
		where
		\begin{eqnarray*}
			&&C_{1}(\tau)=C(\|\eta_{t}\|^{2}_{L^{2}((0,\tau);L^{2}(\Omega))}+\|\gamma_{t}\|^{2}_{L^{2}((0,\tau);L^{2}(\Omega))})\nonumber\\
			&&+C(\|\eta_{tt}\|^{2}_{L^{2}((0,\tau);H^{1}(\Omega)')}+\|\gamma_{tt}\|^{2}_{L^{2}((0,\tau);H^{1}(\Omega)')}),\\
			&&C_{2}(\tau)=C(\|\xi\|_{L^{\infty}((0,\tau);H^{2}(\Omega))}+\|p\|_{L^{\infty}((0,\tau);H^{2}(\Omega))}+\|T\|_{L^{\infty}((0,\tau);H^{2}(\Omega))})\nonumber\\
			&&+C(\|\xi_{t}\|_{L^{2}((0,\tau);H^{2}(\Omega))}+\|p\|_{L^{2}((0,\tau);H^{2}(\Omega))}\nonumber\\
			&&+\|T\|_{L^{2}((0,\tau);H^{2}(\Omega))}
			+\|\operatorname{div}(\bm{{\rm u}})_{t}\|_{L^{2}((0,\tau);H^{2}(\Omega))})\nonumber\\
			&&+C(\|p_{t}\|_{L^{2}((0,\tau);H^{2}(\Omega))}+\|T_{t}\|_{L^{2}((0,\tau);H^{2}(\Omega))}).
		\end{eqnarray*}	
	\end{theorem}
	\begin{proof}
		To derive the above inequality, we need to bound each term on the right-hand side of (\ref{eq-3-54}). Using the fact $\varPi^{0}_{\textbf{u}}, \varPi^{0}_{\xi}=0, \varPi^{-1}_{\eta}=0$ and $\varPi^{-1}_{\gamma}=0$, we have
		{\allowdisplaybreaks[2]
			\begin{align}\label{eq-3-77}
				&\varSigma^{l}_{h}+\Delta t\sum_{n=0}^{l}[(\bm{K}\nabla\widehat{\varPi}^{n+1}_{p},\nabla\widehat{\varPi}^{n+1}_{p})
				+(\bm{\Theta}\nabla\widehat{\varPi}^{n+1}_{T},\nabla\widehat{\varPi}^{n+1}_{T})\nonumber\\
				&+\frac{\Delta t}{2}(\mu\|d_{t}\varepsilon(\varPi^{n+1}_{\bm{{\rm u}}})\|^2_{L^{2}(\Omega)}+k_{6}\|d_{t}\Psi^{n+1}_{\xi}\|^2_{L^{2}(\Omega)}\nonumber\\
				&+(k_{5}-k_{2})\|d_{t}\Psi^{n+\theta}_{\eta}\|^2_{L^{2}(\Omega)}+(k_{3}-k_{2})\|d_{t}\Psi^{n+\theta}_{\gamma}\|^2_{L^{2}(\Omega)})]\nonumber\\
				&\leq \Delta t\sum_{n=0}^{l} [(\Phi^{n+1}_{\xi},\operatorname{div} d_{t}\varPi^{n+1}_{\bm{{\rm u}}})-(\operatorname{div} d_{t}\varLambda^{n+1}_{\bm{{\rm u}}},\Psi^{n+1}_{\xi})]\nonumber\\
				&+\Delta t\sum_{n=0}^{l}[(d_{t}\Psi^{n+\theta}_{\eta},\widehat{\varLambda}^{n+1}_{p}-\widehat{\Phi}^{n+1}_{p})+(d_{t}\Psi^{n+\theta}_{\gamma},\widehat{\varLambda}^{n+1}_{T}-\widehat{\Phi}^{n+1}_{T})]\nonumber\\
				&+(1-\theta)(\Delta t)^{2}\sum_{n=0}^{l}[k_{4}(d^{2}_{t}\eta(t_{n+1}),\Psi^{n+1}_{\xi})
				+k_{1}(d^{2}_{t}\gamma(t_{n+1}),\Psi^{n+1}_{\xi})]\nonumber\\
				&+(1-\theta)(\Delta t)^{2}\sum_{n=0}^{l}[k_{4}(\bm{K}d_{t}\nabla \varPi^{n+1}_{\xi},\nabla\widehat{\varPi}^{n+1}_{p})+k_{1}(\bm{\Theta} d_{t}\nabla \varPi^{n+1}_{\xi},\nabla\widehat{\varPi}^{n+1}_{T})]\nonumber\\
				&+\Delta t\sum_{n=0}^{l}[(\nabla E^{n+\theta}_{T}\cdot(\bm{K}\nabla p(t_{n+\theta})), \widehat{\varPi}^{n+1}_{T})
				+(\nabla T^{n+\theta}_{h}\cdot(\bm{K}\nabla E^{n+\theta}_{p}), \widehat{\varPi}^{n+1}_{T})]\nonumber\\
				&+\Delta t\sum_{n=0}^{l}[(R^{n+\theta}_{\eta},\widehat{\varPi}^{n+1}_{p})+(R^{n+\theta}_{\gamma},\widehat{\varPi}^{n+1}_{T})].
		\end{align}}
		We now estimate each term on the right-hand side of (\ref{eq-3-77}). The last term on the right-hand side of (\ref{eq-3-77}) can be bounded by
		\begin{eqnarray}\label{eq-3-78}
			&&|(R^{n+\theta}_{\eta},\widehat{\varPi}^{n+1}_{p})|=|(R^{n+\theta}_{\eta}, \widehat{\Phi}^{n+1}_{p}+\widehat{\Psi}^{n+1}_{p}-\widehat{\varLambda}^{n+1}_{p})|\leq\frac{3}{2}\|R^{n+\theta}_{\eta}\|^{2}_{L^{2}(\Omega)}\nonumber\\
			&&+\frac{1}{2}\|\widehat{\Phi}^{n+1}_{p}\|^{2}_{L^{2}(\Omega)}+\frac{1}{2}\|\widehat{\Psi}^{n+1}_{p}\|^{2}_{L^{2}(\Omega)}+\frac{1}{2}\|\widehat{\varLambda}^{n+1}_{p}\|^{2}_{L^{2}(\Omega)}\nonumber\\
			&&\leq \frac{\Delta t}{2}\|\eta_{tt}\|^{2}_{L^{2}((t_{n+\theta-1},t_{n+\theta});H^{1}(\Omega)')}+\frac{1}{2}\|\widehat{\Phi}^{n+1}_{p}\|^{2}_{L^{2}(\Omega)}\nonumber\\
			&&+\frac{1}{2}\|\widehat{\Psi}^{n+1}_{p}\|^{2}_{L^{2}(\Omega)}+\frac{1}{2}\|\widehat{\varLambda}^{n+1}_{p}\|^{2}_{L^{2}(\Omega)},
		\end{eqnarray}
		where we have used the fact that
		\begin{eqnarray}
			\|R^{n+\theta}_{\eta}\|^{2}_{H^{1}(\Omega)'}\leq\frac{\Delta t}{3}\int^{t_{n+\theta}}_{t_{n+\theta-1}}\|\eta_{tt}\|^{2}_{H^{1}(\Omega)'} dt.
		\end{eqnarray}
		Similarly, we have
		\begin{eqnarray}
			&&|(R^{n+\theta}_{\gamma},\widehat{\varPi}^{n+1}_{T})|=|(R^{n+\theta}_{\gamma}, \widehat{\Phi}^{n+1}_{T}+\widehat{\Psi}^{n+1}_{T}-\widehat{\varLambda}^{n+1}_{T})|\nonumber\\
			&&\leq \frac{\Delta t}{2}\|\gamma_{tt}\|^{2}_{L^{2}((t_{n+\theta-1},t_{n+\theta});H^{1}(\Omega)')}+\frac{1}{2}\|\widehat{\Phi}^{n+1}_{T}\|^{2}_{L^{2}(\Omega)}\nonumber\\
			&&+\frac{1}{2}\|\widehat{\Psi}^{n+1}_{T}\|^{2}_{L^{2}(\Omega)}+\frac{1}{2}\|\widehat{\varLambda}^{n+1}_{T}\|^{2}_{L^{2}(\Omega)}.
		\end{eqnarray}
		The first term on the right-hand side of (\ref{eq-3-77}) can be bounded by
		\begin{eqnarray}
			&&\Delta t\sum_{n=0}^{l} [(\Phi^{n+1}_{\xi},\operatorname{div} d_{t}\varPi^{n+1}_{\textbf{u}})-(\operatorname{div} d_{t}\varLambda^{n+1}_{\textbf{u}},\Psi^{n+1}_{\xi})]\nonumber\\
			&&=(\Phi^{l+1}_{\xi},\operatorname{div}\varPi^{l+1}_{\textbf{u}})-\Delta t\sum_{n=0}^{l}[(d_{t}\Phi^{n+1}_{\xi},\operatorname{div} \varPi^{n+1}_{\textbf{u}})+(\operatorname{div} d_{t}\varLambda^{n+1}_{\textbf{u}},\Psi^{n+1}_{\xi})]\nonumber\\
			&&\leq\frac{C}{\mu}\|\Phi^{l+1}_{\xi}\|^{2}_{L^{2}(\Omega)}+\frac{\mu}{4}\|\varepsilon(\varPi^{l+1}_{\textbf{u}})\|^{2}_{L^{2}(\Omega)}+\Delta t\sum_{n=0}^{l}\Big[\frac{C}{\mu}\|d_{t}\Phi^{n+1}_{\xi}\|^{2}_{L^{2}(\Omega)}\nonumber\\
			&&+\frac{\mu}{4}\|\varepsilon(\varPi^{n+1}_{\textbf{u}})\|^{2}_{L^{2}(\Omega)}+\frac{1}{k_{6}}\|\operatorname{div} d_{t}\varLambda^{n+1}_{\textbf{u}}\|^{2}_{L^{2}(\Omega)}+\frac{k_{6}}{4}\|\Psi^{n+1}_{\xi}\|^{2}_{L^{2}(\Omega)}\Big].
		\end{eqnarray}
		The second term on the right-hand side of (\ref{eq-3-77}) can be bounded by
		\begin{align}
			&\Delta t\sum_{n=0}^{l}(d_{t}\Psi^{n+\theta}_{\eta},\widehat{\varLambda}^{n+1}_{p}-\widehat{\Phi}^{n+1}_{p})\nonumber\\
			&=\Delta t\sum_{n=0}^{l}[\Delta t(d_t\Psi_{\eta}^{n+\theta},d_t(\widehat{\varLambda}^{n+1}_{p}-\widehat{\Phi}^{n+1}_{p}))
			+d_t(\Psi^{n+\theta}_{\eta},\widehat{\varLambda}^{n+1}_{p}-\widehat{\Phi}^{n+1}_{p})\nonumber\\
			&-(\Psi^{n+\theta}_{\eta},d_t(\widehat{\varLambda}^{n+1}_{p}-\widehat{\Phi}^{n+1}_{p}))]
			\leq (\Delta t)^{2}\sum^{l}_{n=0}\big[\frac{k_5-k_2}{4}\|d_t\Psi^{n+\theta}_{\eta}\|^{2}_{L^{2}(\Omega)}\nonumber\\
			&+\frac{2}{k_5-k_2}(\|d_t\widehat{\varLambda}^{n+1}_{p}\|^{2}_{L^{2}(\Omega)}+\|d_t\widehat{\Phi}^{n+1}_{p}\|^{2}_{L^{2}(\Omega)})\big]+\frac{k_5-k_2}{4}\|\Psi^{l+\theta}_{\eta}\|^{2}_{L^{2}(\Omega)}\nonumber\\
			&+\frac{2}{k_5-k_2}(\|\widehat{\varLambda}^{l+1}_{p}\|^{2}_{L^{2}(\Omega)}+\|\widehat{\Phi}^{l+1}_{p}\|^{2}_{L^{2}(\Omega)})+\Delta t\sum^{l}_{n=0}\big[\frac{k_5-k_2}{4}\|\Psi^{n+\theta}_{\eta}\|^{2}_{L^{2}(\Omega)}\nonumber\\
			&+\frac{2}{k_5-k_2}(\|d_t\widehat{\varLambda}^{n+1}_{p}\|^{2}_{L^{2}(\Omega)}+\|d_t\widehat{\Phi}^{n+1}_{p}\|^{2}_{L^{2}(\Omega)})\big],	
		\end{align}
		and
		\begin{align}
			&\Delta t\sum_{n=0}^{l}(d_{t}\Psi^{n+\theta}_{\gamma},\widehat{\varLambda}^{n+1}_{T}-\widehat{\Phi}^{n+1}_{T})\nonumber\\
			&\leq (\Delta t)^{2}\sum^{l}_{n=0}\big[\frac{k_3-k_2}{4}\|d_t\Psi^{n+\theta}_{\gamma}\|^{2}_{L^{2}(\Omega)}+\frac{2}{(k_3-k_2)}(\|d_t\widehat{\varLambda}^{n+1}_{T}\|^{2}_{L^{2}(\Omega)}\nonumber\\
			&+\|d_t\widehat{\Phi}^{n+1}_{T}\|^{2}_{L^{2}(\Omega)})\big]+\frac{k_3-k_2}{4}\|\Psi^{l+\theta}_{\gamma}\|^{2}_{L^{2}(\Omega)}\nonumber\\
			&
			+\frac{2}{k_3-k_2}(\|\widehat{\varLambda}^{l+1}_{T}\|^{2}_{L^{2}(\Omega)}+\|\widehat{\Phi}^{l+1}_{T}\|^{2}_{L^{2}(\Omega)})+\Delta t\sum^{l}_{n=0}\big[\frac{k_3-k_2}{4}\|\Psi^{n+\theta}_{\gamma}\|^{2}_{L^{2}(\Omega)}\nonumber\\
			&+\frac{2}{k_3-k_2}(\|d_t\widehat{\varLambda}^{n+1}_{T}\|^{2}_{L^{2}(\Omega)}+\|d_t\widehat{\Phi}^{n+1}_{T}\|^{2}_{L^{2}(\Omega)})\big].	
		\end{align}
	When $\theta=0$, the third and fourth terms of formula (\ref{eq-3-77}), to bound the third term on the right-hand side of (\ref{eq-3-77}), we firstly use the summation by parts formula and $d_{t}\eta_{h}(t_{0})=0$ and $d_{t}\gamma_{h}(t_{0})=0$ to get
		\begin{eqnarray}\label{eq-3-82}
			\sum_{n=0}^{l}(d^{2}_{t}\eta(t_{n+1}),\Psi^{n+1}_{\xi})=\frac{1}{\Delta t}(d_{t}\eta(t_{l+1}),\Psi^{l+1}_{\xi})-\sum^{l}_{n=1}(d_{t}\eta(t_{n}),d_{t}\Psi^{n+1}_{\xi}),
		\end{eqnarray}
		\begin{eqnarray}\label{eq-3-82-2}
			\sum_{n=0}^{l}(d^{2}_{t}\gamma(t_{n+1}),\Psi^{n+1}_{\xi})=\frac{1}{\Delta t}(d_{t}\gamma(t_{l+1}),\Psi^{l+1}_{\xi})-\sum^{l}_{n=1}(d_{t}\gamma(t_{n}),d_{t}\Psi^{n+1}_{\xi}).
		\end{eqnarray}
		Now, we bound each term on the right-hand side of (\ref{eq-3-82}) and (\ref{eq-3-82-2}) as follows:
		\begin{eqnarray}
			&&\frac{1}{\Delta t}(d_{t}\eta(t_{l+1}),\Psi^{l+1}_{\xi})\leq\frac{1}{\Delta t}\|d_{t}\eta(t_{l+1})\|_{L^{2}(\Omega)}\|\Psi^{l+1}_{\xi}\|_{L^{2}(\Omega)}\nonumber\\
			&&\leq\frac{2k_{4}}{k_{6}}\|\eta_{t}\|^{2}_{L^{2}((t_{l},t_{l+1});L^{2}(\Omega))}
			+\frac{k_{6}}{8k_{4}(\Delta t)^{2}}\|\Psi^{l+1}_{\xi}\|^{2}_{L^{2}(\Omega)},
		\end{eqnarray}
		and
		\begin{eqnarray}
			\frac{1}{\Delta t}(d_{t}\gamma(t_{l+1}),\Psi^{l+1}_{\xi}) \leq\frac{2k_{1}}{k_{6}}\|\gamma_{t}\|^{2}_{L^{2}((t_{l},t_{l+1});L^{2}(\Omega))}+\frac{k_{6}}{8k_{1}(\Delta t)^{2}}\|\Psi^{l+1}_{\xi}\|^{2}_{L^{2}(\Omega)}.
		\end{eqnarray}
		\begin{eqnarray}
			&&\sum^{l}_{n=1}(d_{t}\eta(t_{n}),d_{t}\Psi^{n+1}_{\xi})\leq\sum_{n=1}^{l}\|d_{t}\eta(t_{n})\|_{L^{2}(\Omega)}\|d_{t}\Psi^{n+1}_{\xi}\|_{L^{2}(\Omega)}\nonumber\\
			&&\leq \sum_{n=1}^{l}\big(\frac{k_{4}}{k_{6}}\|d_{t}\eta(t_{n})\|^{2}_{L^{2}(\Omega)}+\frac{k_{6}}{4k_{4}}\|d_{t}\Psi^{n+1}_{\xi}\|^{2}_{L^{2}(\Omega)}\big)\nonumber\\
			&&\leq\frac{k_{4}}{k_{6}}\|\eta_{t}\|^{2}_{L^{2}((0,\tau);L^{2}(\Omega))}+\sum_{n=1}^{l}\frac{k_{6}}{4k_{4}}\|d_{t}\Psi^{n+1}_{\xi}\|^{2}_{L^{2}(\Omega)},
		\end{eqnarray}
		and
		\begin{eqnarray}
			\sum^{l}_{n=1}(d_{t}\gamma(t_{n}),d_{t}\Psi^{n+1}_{\xi})
			\leq\frac{k_{1}}{k_{6}}\|\gamma_{t}\|^{2}_{L^{2}((0,\tau);L^{2}(\Omega))}+\sum_{n=1}^{l}\frac{k_{6}}{4k_{1}}\|d_{t}\Psi^{n+1}_{\xi}\|^{2}_{L^{2}(\Omega)}.
		\end{eqnarray}
		The fourth term of the right-hand side of (\ref{eq-3-77}) can be bounded by
		\begin{align}
			&\Delta t\sum_{n=0}^{l}k_{4}(\bm{K}\nabla\widehat{\varPi}^{n+1}_{p},d_{t}\nabla \varPi^{n+1}_{\xi})\nonumber\\
			&\leq \sum_{n=0}^{l}\frac{c_{1}k_{4}k_{M}\Delta t}{h}\|d_{t}\varPi^{n+1}_{\xi}\|_{L^{2}(\Omega)}\|\nabla\widehat{\varPi}^{n+1}_{p}\|_{L^{2}(\Omega)}\nonumber\\
			&\leq\frac{c_{1}k_{4}k_{M}\Delta t}{h\beta_{1}}\sum_{n=0}^{l}\sup_{\textbf{v}_{h}\in \textbf{V}_{h}}\frac{\mu(d_{t}\varepsilon(\varPi_{\textbf{u}}^{n+1}),\varepsilon(\textbf{v}_{h}))-(d_{t}\varLambda_{\xi}^{n+1},\nabla\cdot\textbf{v}_{h})}{\|\nabla\textbf{v}_{h}\|_{L^{2}(\Omega)}}\|\nabla\widehat{\varPi}^{n+1}_{p}\|_{L^{2}(\Omega)}\nonumber\\
			&\leq\sum_{n=0}^{l}\Big[\frac{4c^{2}_{1}k^{2}_{4}k^{2}_{M}\Delta t^{2}\mu^{2}}{k_{m}h^{2}\beta^{2}_{1}}\|d_{t}\varepsilon(\varPi_{\textbf{u}}^{n+1})\|^{2}_{L^{2}(\Omega)}+\frac{4c^{2}_{1}k^{2}_{4}k^{2}_{M}\Delta t^{2}}{k_{m}h^{2}\beta^{2}_{1}}\|d_{t}\varLambda_{\xi}^{n+1}\|^{2}_{L^{2}(\Omega)}\nonumber\\
			&+\frac{k_{m}}{8}\|\nabla\widehat{\varPi}^{n+1}_{p}\|^{2}_{L^{2}(\Omega)}\Big],
		\end{align}
		and
		\begin{eqnarray}\label{eq-3-91}
			&&\Delta t\sum_{n=0}^{l}k_{1}(\bm{\Theta}\nabla\widehat{\varPi}^{n+1}_{T}, d_{t}\nabla \varPi^{n+1}_{\xi})\nonumber\\
			&&\leq\sum_{n=0}^{l}\Big[\frac{4c^{2}_{1}k^{2}_{1}\theta^{2}_{M}\Delta t^{2}\mu^{2}}{\theta_{m}h^{2}\beta^{2}_{1}}\|d_{t}\varepsilon(\varPi_{\textbf{u}}^{n+1})\|^{2}_{L^{2}(\Omega)}+\frac{4c^{2}_{1}k^{2}_{1}\theta^{2}_{M}\Delta t^{2}}{\theta_{m}h^{2}\beta^{2}_{1}}\|d_{t}\varLambda_{\xi}^{n+1}\|^{2}_{L^{2}(\Omega)}\nonumber\\
			&&+\frac{\theta_{m}}{8}\|\nabla\widehat{\varPi}^{n+1}_{T}\|^{2}_{L^{2}(\Omega)}\Big].
		\end{eqnarray}
 	The fiveth term can be bounded by using \reff{eq-33-4}.
	\begin{align}
		&(\mathcal{N}(\nabla T^{n
		+\theta}_{h})\cdot(\bm{K}\nabla E^{n+\theta}_{p}), \widehat{\varPi}^{n+1}_{T})\nonumber\\
		&\leqslant(\mathcal{N}(\nabla T^{n+\theta}_{h})\cdot\bm{K}\nabla\widehat{E}^{n+1}_{p}, \widehat{\Phi}^{n+1}_{T}+\widehat{\Psi}^{n+1}_{T}-\widehat{\varLambda}^{n+1}_{T})\nonumber\\
		&=(\mathcal{N}(\nabla T^{n+\theta}_{h})\cdot\bm{K}\nabla\widehat{\varPi}^{n+1}_{p}, \widehat{\Phi}^{n+1}_{T}+\widehat{\Psi}^{n+1}_{T}-\widehat{\varLambda}^{n+1}_{T})\nonumber\\	
		&+(\mathcal{N}(\nabla T^{n+\theta}_{h})\cdot\bm{K}\nabla \widehat{\varLambda}^{n+1}_{p}, \widehat{\Phi}^{n+1}_{T}+\widehat{\Psi}^{n+1}_{T}-\widehat{\varLambda}^{n+1}_{T})\nonumber\\
		&\leq \frac{k_{m}}{8}\|\nabla\widehat{\varPi}^{n+1}_{p}\|^{2}_{L^{2}(\Omega)}
		+\frac{6N^2k^2_{M}}{k_{m}}(\|\widehat{\Phi}^{n+1}_{T}\|^{2}_{L^{2}(\Omega)}
		+\|\widehat{\Psi}^{n+1}_{T}\|^{2}_{L^{2}(\Omega)}+\|\widehat{\varLambda}^{n+1}_{T}\|^{2}_{L^{2}(\Omega)})\nonumber\\
		&+\frac{3N^2k^2_{M}}{2}\|\nabla\widehat{\varLambda}^{n+1}_{p}\|^{2}_{L^{2}(\Omega)}
		+\frac{1}{2}(\|\widehat{\Phi}^{n+1}_{T}\|^{2}_{L^{2}(\Omega)}
		+\|\widehat{\Psi}^{n+1}_{T}\|^{2}_{L^{2}(\Omega)}+\|\widehat{\varLambda}^{n+1}_{T}\|^{2}_{L^{2}(\Omega)}),
	\end{align}
and
		\begin{eqnarray}
		   	&&(\nabla E^{n+\theta}_{T}\cdot\bm{K}\nabla p(t_{n+\theta}), \widehat{\varPi}^{n+1}_{T})\nonumber\\
			&&\leqslant(\nabla \widehat{E}^{n+1}_{T}\cdot\bm{K}\nabla p(t_{n+\theta}),\widehat{\Phi}^{n+1}_{T}+\widehat{\Psi}^{n+1}_{T}-\widehat{\varLambda}^{n+1}_{T})\nonumber\\
			&&=(\nabla \widehat{\varPi}^{n+1}_{T}\cdot\bm{K}\nabla p(t_{n+\theta}),\widehat{\Phi}^{n+1}_{T}+\widehat{\Psi}^{n+1}_{T}-\widehat{\varLambda}^{n+1}_{T})\nonumber\\
			&&+(\nabla \widehat{\varLambda}^{n+1}_{T}\cdot\bm{K}\nabla p(t_{n+\theta}),\widehat{\Phi}^{n+1}_{T}+\widehat{\Psi}^{n+1}_{T}-\widehat{\varLambda}^{n+1}_{T})\nonumber\\
			&&\leq \|\nabla \widehat{\varPi}^{n+1}_{T}\|_{L^{2}(\Omega)}\|\bm{K}\nabla p(t_{n+\theta})\|_{L^{\infty}(\Omega)}\big(\|\widehat{\Phi}^{n+1}_{T}\|_{L^{2}(\Omega)}
			+\|\widehat{\Psi}^{n+1}_{T}\|_{L^{2}(\Omega)}+\|\widehat{\varLambda}^{n+1}_{T}\|_{L^{2}(\Omega)}\big)\nonumber\\
			&&+ \|\nabla \widehat{\varLambda}^{n+1}_{T}\|_{L^{2}(\Omega)}\|\bm{K}\nabla p(t_{n+\theta})\|_{L^{\infty}(\Omega)}\big(\|\widehat{\Phi}^{n+1}_{T}\|_{L^{2}(\Omega)}
			+\|\widehat{\Psi}^{n+1}_{T}\|_{L^{2}(\Omega)}+\|\widehat{\varLambda}^{n+1}_{T}\|_{L^{2}(\Omega)}\big)\nonumber\\
			&&\leq \frac{\theta_{m}}{8}\|\nabla\widehat{\varPi}^{n+1}_{T}\|^{2}_{L^{2}(\Omega)}
			+\frac{6\delta^2_{1}k^2_{M}}{\theta_{m}}(\|\widehat{\Phi}^{n+1}_{T}\|^{2}_{L^{2}(\Omega)}
			+\|\widehat{\Psi}^{n+1}_{T}\|^{2}_{L^{2}(\Omega)}+\|\widehat{\varLambda}^{n+1}_{T}\|^{2}_{L^{2}(\Omega)})\nonumber\\
			&&+ \frac{3\delta^2_{1}k^2_{M}}{2}\|\nabla\widehat{\varLambda}^{n+1}_{T}\|^{2}_{L^{2}(\Omega)}
			+\frac{1}{2}(\|\widehat{\Phi}^{n+1}_{T}\|^{2}_{L^{2}(\Omega)}
			+\|\widehat{\Psi}^{n+1}_{T}\|^{2}_{L^{2}(\Omega)}+\|\widehat{\varLambda}^{n+1}_{T}\|^{2}_{L^{2}(\Omega)}).\label{1.1111}
		\end{eqnarray}
		Substituting (\ref{eq-3-78})-(\ref{1.1111}) into (\ref{eq-3-77}), using the discrete Gronwall lemma, we get
		{\allowdisplaybreaks[2]
			\begin{align}
				&\mu\|\varepsilon(\varPi^{l+1}_{\textbf{u}})\|^{2}_{L^{2}(\Omega)}+k_{6}\|\Psi^{l+1}_{\xi}\|^{2}_{L^{2}(\Omega)}+(k_{5}-k_{2})\|\Psi^{l+\theta}_{\eta}\|^{2}_{L^{2}(\Omega)}+(k_{3}-k_{2})\|\Psi^{l+\theta}_{\gamma}\|^{2}_{L^{2}(\Omega)}\nonumber\\
				&+\Delta t\sum_{n=0}^{l}(\|\bm{K}\nabla\widehat{\varPi}^{n+1}_{p}\|^{2}_{L^{2}(\Omega)}
				+\|\bm{\Theta}\nabla\widehat{\varPi}^{n+1}_{T}\|^{2}_{L^{2}(\Omega)})\nonumber\\
				&\leq \frac{\Delta t^{2}}{2}\|\eta_{tt}\|^{2}_{L^{2}((0,\tau);H^{1}(\Omega)^{'})}+\frac{\Delta t^{2}}{2}\|\gamma_{tt}\|^{2}_{L^{2}((0,\tau);H^{1}(\Omega)^{'})}\nonumber\\
				&+\Delta t^2(\frac{2k^{2}_{4}}{k_{6}}\|\eta_{t}\|^{2}_{L^{2}((0,\tau);L^{2}(\Omega))}+\frac{2k^{2}_{1}}{k_{6}}\|\gamma_{t}\|^{2}_{L^{2}((0,\tau);L^{2}(\Omega))})\nonumber\\
				&+C\big[\|\Phi^{l+1}_{\xi}\|^{2}_{L^{2}(\Omega)}
				+\|\widehat{\varLambda}^{l+1}_{p}\|^{2}_{L^{2}(\Omega)}+\|\widehat{\Phi}^{l+1}_{p}\|^{2}_{L^{2}(\Omega)}+\|\widehat{\varLambda}^{l+1}_{T}\|^{2}_{L^{2}(\Omega)}+\|\widehat{\Phi}^{l+1}_{T}\|^{2}_{L^{2}(\Omega)}\nonumber\\
				&+\Delta t\sum_{n=1}^{l}(\|d_{t}\varLambda^{n+1}_{\xi}\|^2_{L^{2}(\Omega)}+\|d_{t}\Phi^{n+1}_{\xi}\|^2_{L^{2}(\Omega)})+
				\Delta t\sum_{n=0}^{l}\|\operatorname{div} d_{t}\varLambda^{n+1}_{\textbf{u}}\|^{2}_{L^{2}(\Omega)}\nonumber\\
				&+\Delta t\sum_{n=0}^{l}(\|d_{t}\widehat{\varLambda}^{n+1}_{p}\|^{2}_{L^{2}(\Omega)}+\|d_{t}\widehat{\Phi}^{l+1}_{p}\|^{2}_{L^{2}(\Omega)}+\|d_{t}\widehat{\varLambda}^{n+1}_{T}\|^{2}_{L^{2}(\Omega)}+\|d_{{t}}\widehat{\Phi}^{n+1}_{T}\|^{2}_{L^{2}(\Omega)})\no\\
				&+\Delta t\sum_{n=0}^{l}(\|\widehat{\varLambda}^{n+1}_{p}\|^2_{L^{2}(\Omega)}+\|\widehat{\Phi}^{n+1}_{p}\|^2_{L^{2}(\Omega)}
				+\|\widehat{\varLambda}^{n+1}_{T}\|^2_{L^{2}(\Omega)}+\|\widehat{\Phi}^{n+1}_{T}\|^2_{L^{2}(\Omega)})\nonumber\\
				&+(\Delta t)^{2}\sum_{n=0}^{l}(\|\widehat{\varLambda}^{n+1}_{p}\|^2_{L^{2}(\Omega)}+\|\widehat{\Phi}^{n+1}_{p}\|^2_{L^{2}(\Omega)}
				+\|\widehat{\varLambda}^{n+1}_{T}\|^2_{L^{2}(\Omega)}+\|\widehat{\Phi}^{n+1}_{T}\|^2_{L^{2}(\Omega)})\big]
		\end{align}}
		provided that $\Delta t <
		\frac{k_{m}\theta_{m}\beta^{2}_{1}h^{2}}{8c^{2}_{1}\mu(k^{2}_{M}\theta_{m}k^{2}_{4}+k^{2}_{1}\theta^{2}_{M}k_{m})}$ when $\theta=0$, but it hold for all $\Delta t>0$ when $\theta=1$. Hence, (\ref{eq-3-75}) follows from the approximation properties of the projection operators $\mathcal{Q}_{h},\mathcal{R}_{h}$ and $\mathcal{S}_{h}$. The proof is complete. 
	\end{proof}

	We conclude this section by stating the main theorem as follows.
\begin{theorem}\label{th-3-6}
	Under the assumption of Theorem \ref{th-3-5}, the solution of the MFEM satisfies the following error estimates
	\begin{align}
		&\max_{0\leq n\leq l}[\sqrt{\mu}\|\nabla(\bm{{\rm u}}(t_{n})-\bm{{\rm u}}^{n}_{h})\|_{L^{2}(\Omega)}+\sqrt{k_{6}}\|\xi(t_{n})-\xi^{n}_{h}\|_{L^{2}(\Omega)}\nonumber\\
		&+\sqrt{k_{5}-k_{2}}\|\eta(t_{n})-\eta^{n}_{h}\|_{L^{2}(\Omega)}+\sqrt{k_{3}-k_{2}}\|\gamma(t_{n})-\gamma^{n}_{h}\|_{L^{2}(\Omega)}]\nonumber\\
		&\leq\widehat{C}_{1}(\tau)\Delta t+\widehat{C}_{2}(\tau)h^{2},\\
		&[\Delta t\sum_{n=0}^{l}\|\nabla(p(t_{n})-p^{n}_{h})\|^2_{L^{2}(\Omega)}]+\|\nabla(T(t_{n})-T^{n}_{h})\|^2_{L^{2}(\Omega)}]^{\frac{1}{2}}\nonumber\\
		&\leq \widehat{C}_{1}(\tau)\Delta t+\widehat{C}_{2}(\tau)h
	\end{align}
	provided that $\Delta t=O(h^{2})$ when $\theta=0$ and $\Delta t>0$ when $\theta=1$, where
	\begin{align*}
		\widehat{C}_{1}(\tau):&=C_{1}(\tau),\\
		\widehat{C}_{2}(\tau):&=C_{2}(\tau)+\|\xi\|_{L^{\infty}((0,\tau);H^{2}(\Omega))}+\|\eta\|_{L^{\infty}((0,\tau);H^{2}(\Omega))}\\
		&+\|\gamma\|_{L^{\infty}((0,\tau);H^{2}(\Omega))}+\|\nabla \bm{{\rm u}}\|_{L^{\infty}((0,\tau);H^{2}(\Omega))}.
	\end{align*}
\end{theorem}
\begin{proof}
	The above estimates follow immediately from an application of the triangle inequality on
	\begin{eqnarray*}
		&&\textbf{u}(t_{n})-\textbf{u}^{n}_{h}=\varLambda^{n}_{\textbf{u}}+\varPi^{n}_{\textbf{u}},
		\quad\xi(t_{n})-\xi^{n}_{h}=\Phi^{n}_{\xi}+\Psi^{n}_{\xi},\\
		&&\xi(t_{n})-\xi^{n}_{h}=\varLambda^{n}_{\xi}+\varPi^{n}_{\xi},
		\quad\eta(t_{n})-\eta^{n}_{h}=\Phi^{n}_{\eta}+\Psi^{n}_{\eta},\\
		&&\gamma(t_{n})-\gamma^{n}_{h}=\Phi^{n}_{\gamma}+\Psi^{n}_{\gamma},
		\quad p(t_{n})-p^{n}_{h}=\Phi^{n}_{p}+\Psi^{n}_{p},\\
		&&p(t_{n})-p^{n}_{h}=\varLambda^{n}_{p}+\varPi^{n}_{p},
		\quad T(t_{n})-T^{n}_{h}=\Phi^{n}_{T}+\Psi^{n}_{T},\\
		&&T(t_{n})-T^{n}_{h}=\varLambda^{n}_{T}+\varPi^{n}_{T}
	\end{eqnarray*}
	and appealing to (\ref{eq-3-47}), (\ref{eq-3-47-2}), (\ref{eq-3-47-3}) and Theorem \ref{th-3-5}. The proof is complete.
\end{proof}

	\section{ Numerical tests}\label{sec-4}		
  	 In this section, we present three numerical tests to verify
	the theoretical results for the proposed numerical methods.
	
		\textbf{Test 1.} This test problem is same as one of \cite{Brun2020}, we take $\Omega=[0,1]\times[0,1]$ , $\Gamma_{1} = \{(x,0);~0\leq x\leq1\}$,\ $~\Gamma_{2}= \{(1,y);~0\leq y\leq1 \}$,\ $\Gamma_{3}= \{(x,1);~0\leq x\leq1 \}$,\ $~\Gamma_{4}= \{ (0,y);~0\leq y\leq1 \}$, 
		 and prescribe the following smooth solutions for the temperature,  pressure and displacement:
	\begin{equation}
		\begin{array}{l}
			T(x,y, t)=t x\left(1-x\right) y\left(1-y\right), \\
			p(x,y, t)=t x\left(1-x\right) y\left(1-y\right) ,\\
			\bm{{\rm u}}(x,y, t)=t x\left(1-x\right) y\left(1-y\right)(1,1)^{'},
		\end{array}
	\end{equation}
\begin{table}[H]
	\centering
	\caption{ Physical parameters}\label{table1}
	\begin{tabular}{c l c }
		\hline
		Parameter   &  \quad Description      &   \quad   Value    \\
		\hline
		$a_0$       &\quad Effective thermal capacity               &  \quad 2e5\\
		$b_0$       &\quad Thermal dilation coefficient              &  \quad 1e5 \\
		$c_0$       &\quad Constrained specific storage coefficient &  \quad 2e5 \\
		$\alpha$    &\quad Biot-Willis constant                     &  \quad  0.01 \\
		$\beta$     &\quad Thermal stress coefficient.               &  \quad  0.01\\
		$\bm{K}$         &\quad Permeability tensor                      &  \quad $0.1I$\\
		$\bm{\Theta}$     &\quad Effective thermal conductivity           &  \quad $0.1I$\\
		$E$         &\quad Young's modulus                          &  \quad  1.25e5\\
		$\nu$       &\quad Poisson ratio                            &  \quad  0.25\\
		\hline
	\end{tabular}
\end{table}
	 We consider the problem (\ref{eq-2-1})-(\ref{eq-2-3}) with the following source functions:
	 	\begin{align*}
	 	\bm{{\rm f}_1} &=2t(\mu+\lambda)(y-y^2)-\frac{\mu }{2}(-2t(x-x^2)+t(1-2x)(1-2y))\no\\
	 	&-\lambda t(1-2x)(1-2y)+(\alpha+\beta)t(1-2x)(y-y^2),\\
	 	\bm{{\rm f}_2} &=2t(\mu+\lambda)(x-x^2)-\frac{\mu }{2}(-2t(y-y^2)+t(1-2x)(1-2y))\no\\
	 	&-\lambda t(1-2x)(1-2y)+(\alpha+\beta)t(1-2y)(x-x^2),\\
	 	\phi &=(a_{0}-b_{0}(x-x^{2})(y-y^{2})+\beta(1-2x)(y-y^{2})\no\\
	 	&+\beta(x-x^{2})(1-2y)+0.2t((y-y^{2})+(x-x^{2}))\no\\
	 	&-0.1(t^{2}(1-2x)^{2}(y-y^{2})^{2}+t^{2}(x-x^{2})^{2})(1-2y)^{2}),\\
	 	g &=(c_{0}-b_{0})(x-x^{2})(y-y^{2})+\alpha(1-2x)(y-y^{2})\no\\
	 	&+\alpha(x-x^{2})(1-2y)+0.2t((y-y^{2})+(x-x^{2})).
	 	\end{align*}
	The boundary and initial conditions are given by
	\begin{eqnarray}
		p=tx(1-x)y(1-y)  \quad &\mathrm{on} & \quad\partial\Omega_{\tau},\nonumber\\
		T=tx(1-x)y(1-y)  \quad&\mathrm{on} & \quad\partial\Omega_{\tau},\nonumber\\
		u_1=tx(1-x)y(1-y) \quad&\mathrm{on}& \quad\Gamma_j\times (0,\tau),~j=3,~4,\nonumber\\
		u_2=tx(1-x)y(1-y)  \quad&\mathrm{on} & \quad\Gamma_j\times(0,\tau),~j=3,~4, \nonumber\\
		\sigma(\pmb{\tau})\bm{n}-\alpha pI\bm{n}-\beta TI\bm{n}=\textbf{f}_1 \quad &\mathrm{on}& \quad\partial\Omega_{\tau},\nonumber\\
		\textbf{u}(x,y,0)=\bm{0},~ p(x,y,0)=0,~T(x,y,0)=0 \quad&\mathrm{in}& \quad\Omega.\nonumber
	\end{eqnarray}
where $\bm{{\rm f}}_{1}=(((\lambda+\mu)t(1-2x)(y-y^{2})+\lambda t(1-2y)(x-x^{2}))n_{1}+(\frac{\mu}{2}t(x-x^{2})(1-2y)+\frac{\mu}{2}t(1-2x)(y-y^{2}))n_{2},\\
~~~~~~~~~~~~~~~~~((\lambda+\mu)t(1-2y)(x-x^{2})+\lambda t(1-2x)(y-y^{2}))n_{1}+(\frac{\mu}{2}t(x-x^{2})(1-2y)+\frac{\mu}{2}t(1-2x)(y-y^{2}))n_{2})^{'}$.

	\begin{table}[htbp]
	\vspace{-2.0em}
	\begin{center}
		\caption{Error and convergence rates of $u_h^n$, $p_h^n$, $T_h^n$}\label{table2}
		\resizebox{\textwidth}{12mm}{
			\begin{tabular}{ccccccccccccc}
				\hline
				$h$  & $\frac{\|e_u\|_{L^2(\Omega)}}{\|u\|_{L^2(\Omega)}}$  &  CR  &  $\frac{\|e_u\|_{H^1(\Omega)}}{\|u\|_{H^1(\Omega)}}$  &  CR & $\frac{\|e_p\|_{L^2(\Omega)}}{\|p\|_{L^2(\Omega)}}$ & CR  &  $\frac{\|e_p\|_{H^1(\Omega)}}{\|p\|_{H^1(\Omega)}}$  &  CR  &  $\frac{\|e_T\|_{L^2(\Omega)}}{\|T\|_{L^2(\Omega)}}$  &  CR  & $\frac{\|e_T\|_{H^1(\Omega)}}{\|T\|_{H^1(\Omega)}}$ & CR \\ 
				\hline
				$1/4$   &0.0079&      &0.0541   &      &0.0779    &      &0.4324&      &0.0779&      &0.4324&    \\
				$1/8$   &9.3526e-04&  3.0829    &0.0139    &  1.9590    &0.0168    & 2.2123     &0.2099&  1.0424    &0.0168& 2.2123     &0.2099& 1.0424   \\
				$1/16$  &1.1146e-04&3.0689 &0.0035    &1.9821&0.0038    &2.1379&0.1034&1.0215&0.0038&2.1379&0.1034&1.0215\\
				$1/32$  &1.3597e-05&3.0352&8.8573e-04&1.9910&9.0225e-04&2.0817&0.0514&1.0099&9.0225e-04&2.0817&0.0514&1.0099\\
				$1/64$  &1.6818e-06&3.0152&2.2214e-04&1.9954&2.1868e-04&2.0451&0.0256&1.0046&2.1868e-04&2.0451&0.0256&1.0046\\
				\hline
		\end{tabular}}
	\end{center}
\end{table}
\begin{table}[htbp]
	\vspace{-1.0em}
	\begin{center}
		\caption{Order of convergence of time discretization of Test 1 }\label{table211}
		\resizebox{\textwidth}{12mm}{
			\begin{tabular}{ccccccccccc}
				\hline
				$\Delta t$ & $\left\|\mathbf{u}_{h}^{\Delta t}-\mathbf{u}_{h}^{\frac{1}{2} \Delta t}\right\|_{L^2(\Omega)}$ & $\rho_{\Delta t, \mathbf{u}_{h}}$ & $\left\|p_{h}^{\Delta t}-p_{h}^{ \frac{1}{2} \Delta t}\right\|_{L^2(\Omega)}$ & $\rho_{\Delta t, p_h}$ & $\left\|T_{h}^{\Delta t}-T_{h}^{\frac{1}{2} \Delta t}\right\|_{L^2(\Omega)}$ & $\rho_{\Delta t, T_h}$ \\
				\hline
				$\frac{1}{10}$ & $1.4312e-10$ & $2.0000$ & $2.8438e-09$ & $1.9997$ & $2.8310e-09$ & $1.9995$ \\
				$\frac{1}{20}$ & $7.1559e-11$ & $2.0000$ & $1.4221e-09$ & $1.9999$ & $1.4159e-09$ & $1.9998$ \\
				$\frac{1}{40}$ & $3.5779e-11$ & $2.0000$ & $7.1109e-10$ & $1.9999$ & $7.0802e-10$ & $1.9999$ \\
				$\frac{1}{80}$ & $1.7890e-11$ &          & $3.5556e-10$ &          & $3.5403e-10$ & \\
				\hline
		\end{tabular}}
	\end{center}
\end{table}
		\begin{figure}[htbp]
		\subfigure[]{
			\centering
			\includegraphics[width=2.5in]{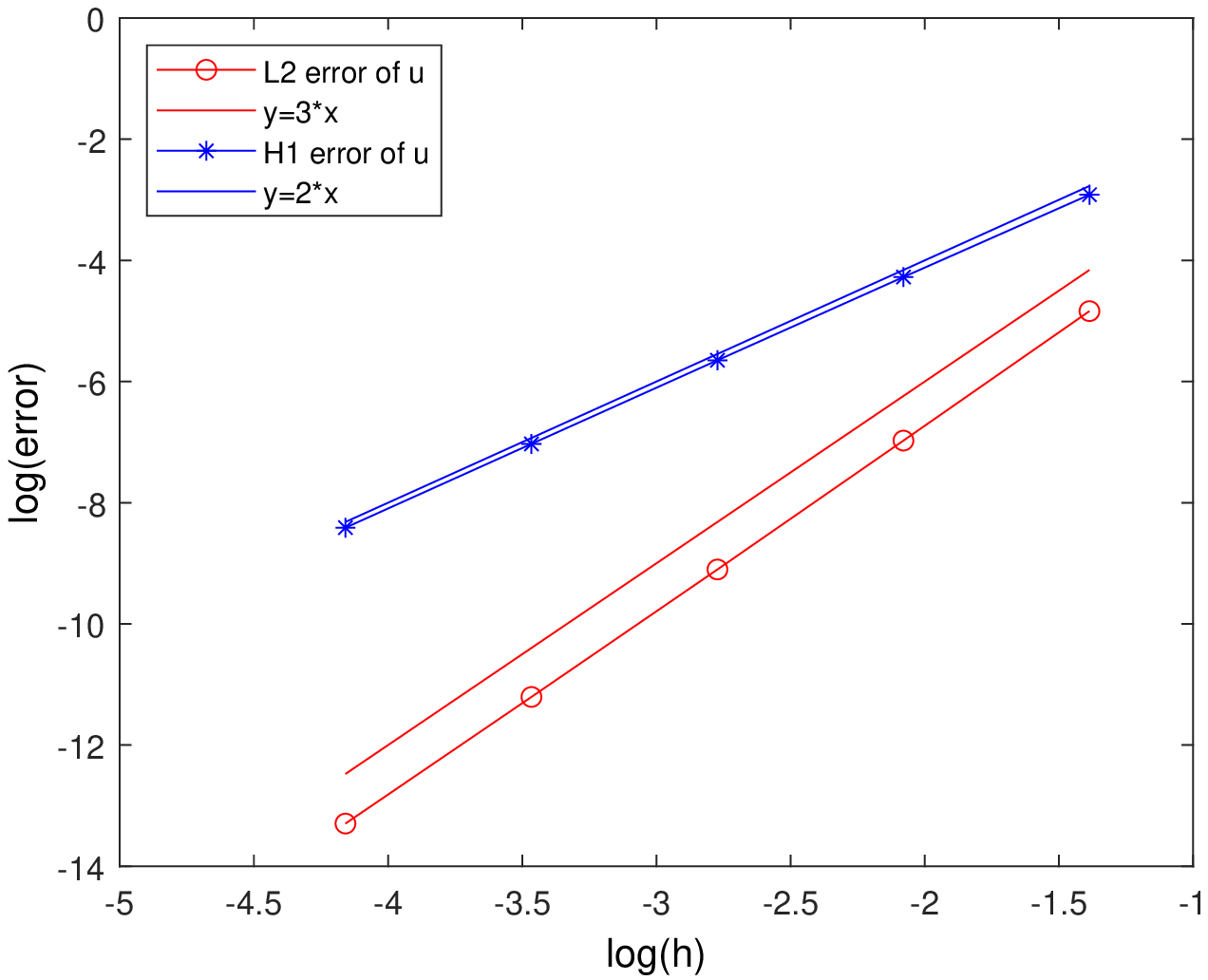}
			\label{fig111}
		}
		\subfigure[]{
			\centering
			\includegraphics[width=2.5in]{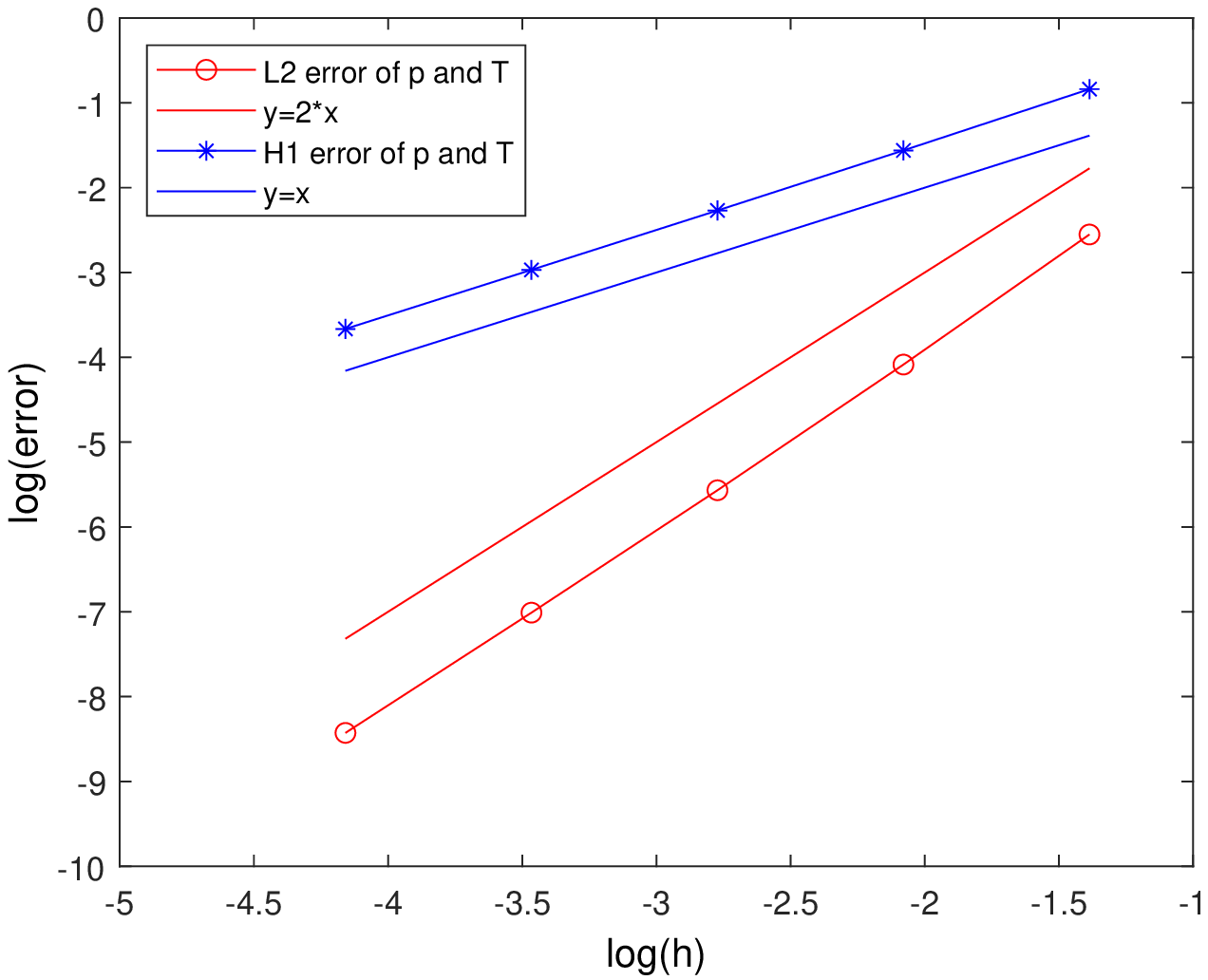}
			\label{fig211}
		}
		\caption{(a) spatial convergence order for $u^{n}_{h}$, (b) space convergence rate for $p^{n}_{h}, T^{n}_{h}$.}
	\end{figure}
	\begin{figure}[H]
		\subfigure[]{
			\centering
			\includegraphics[width=2.5in]{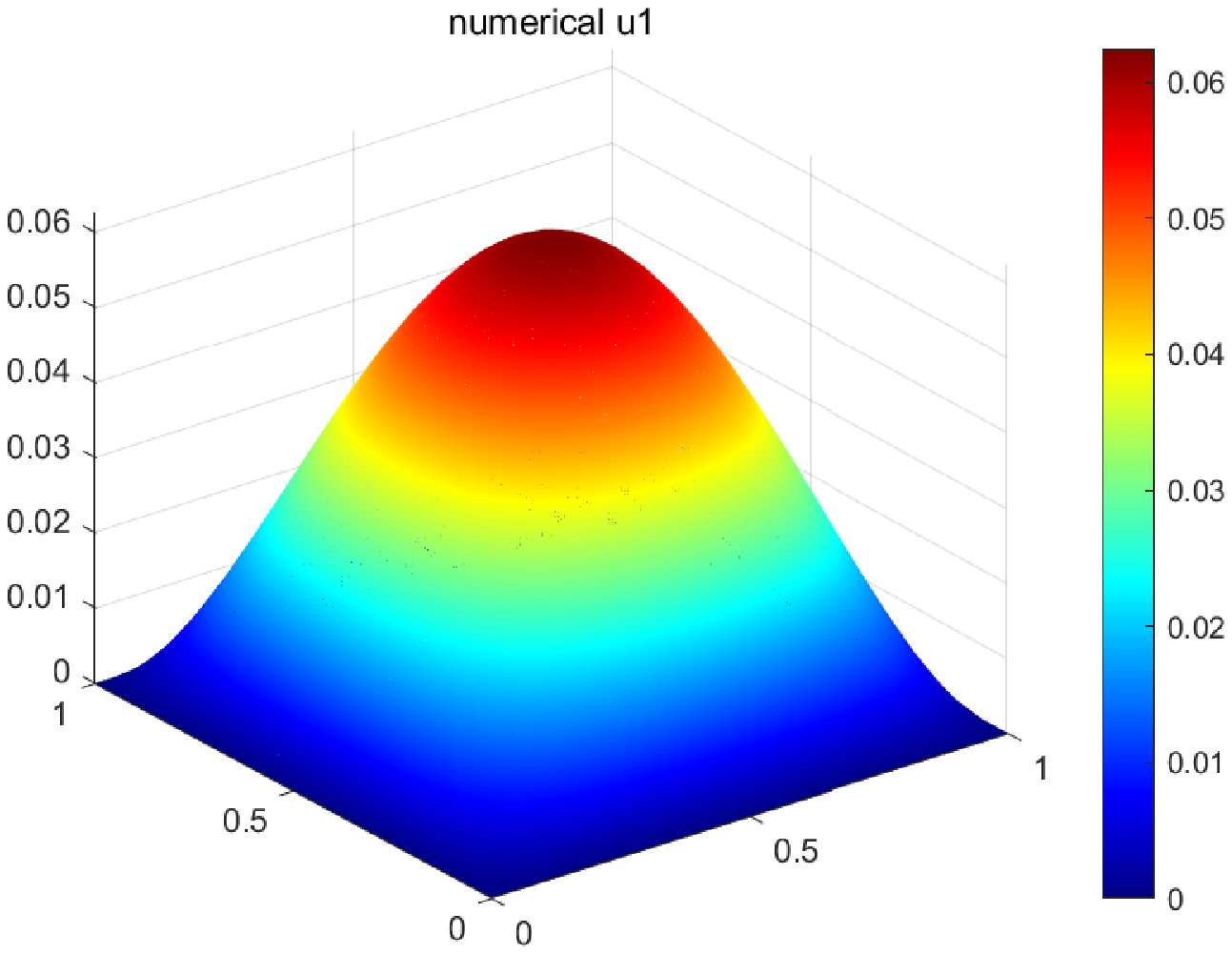}
			\label{fig6}
		}
		\subfigure[]{
			\centering
			\includegraphics[width=2.5in]{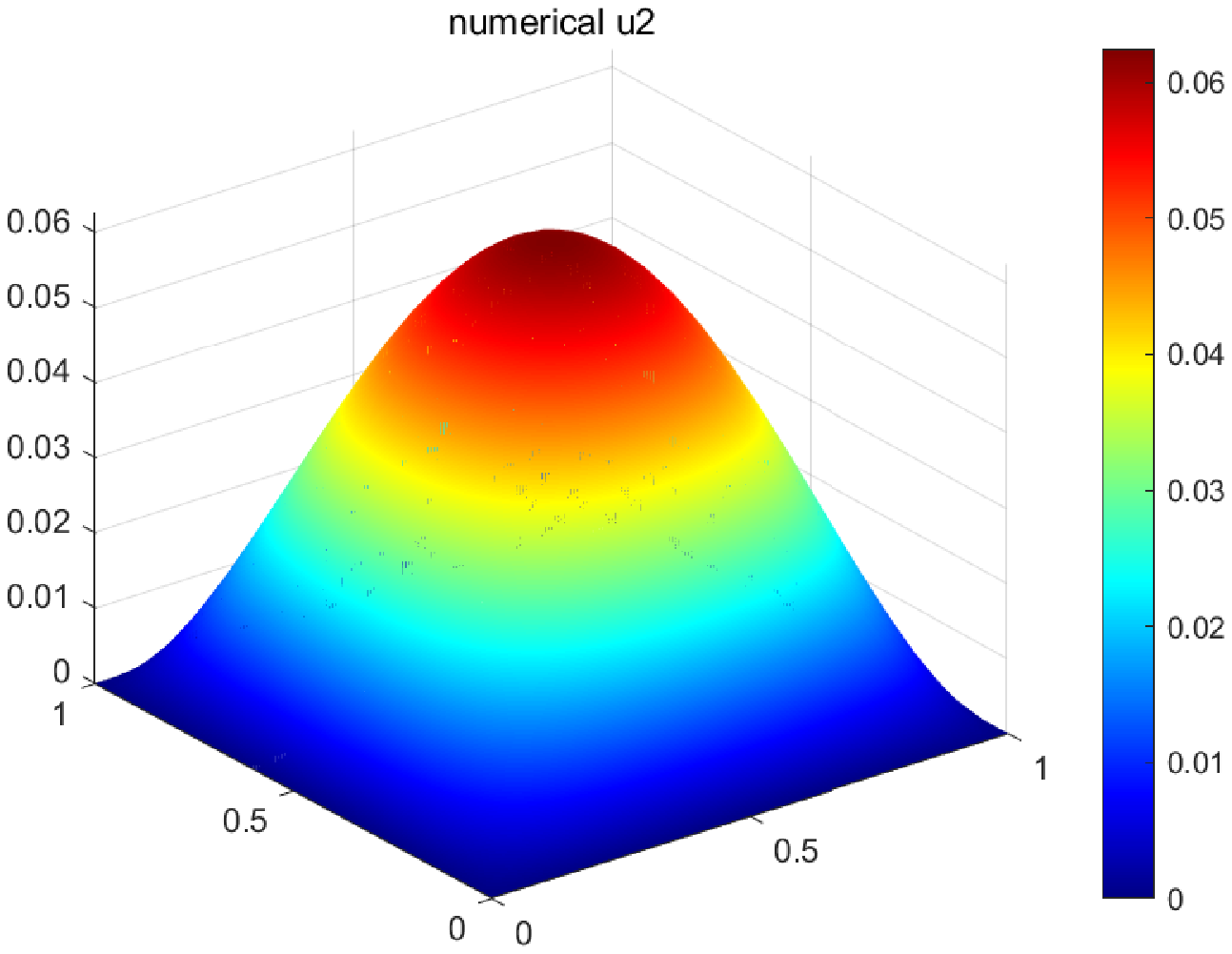}
			\label{fig7}
		}
		\caption{(a) and (b) are Surface plot of $u_{1h}^{n}$ and $u_{2h}^{n}$ at the terminal time $\tau$ respectively.}
	\end{figure}
	
	\begin{figure}[H]
		\centering
		\includegraphics[height=5cm,width=7cm]{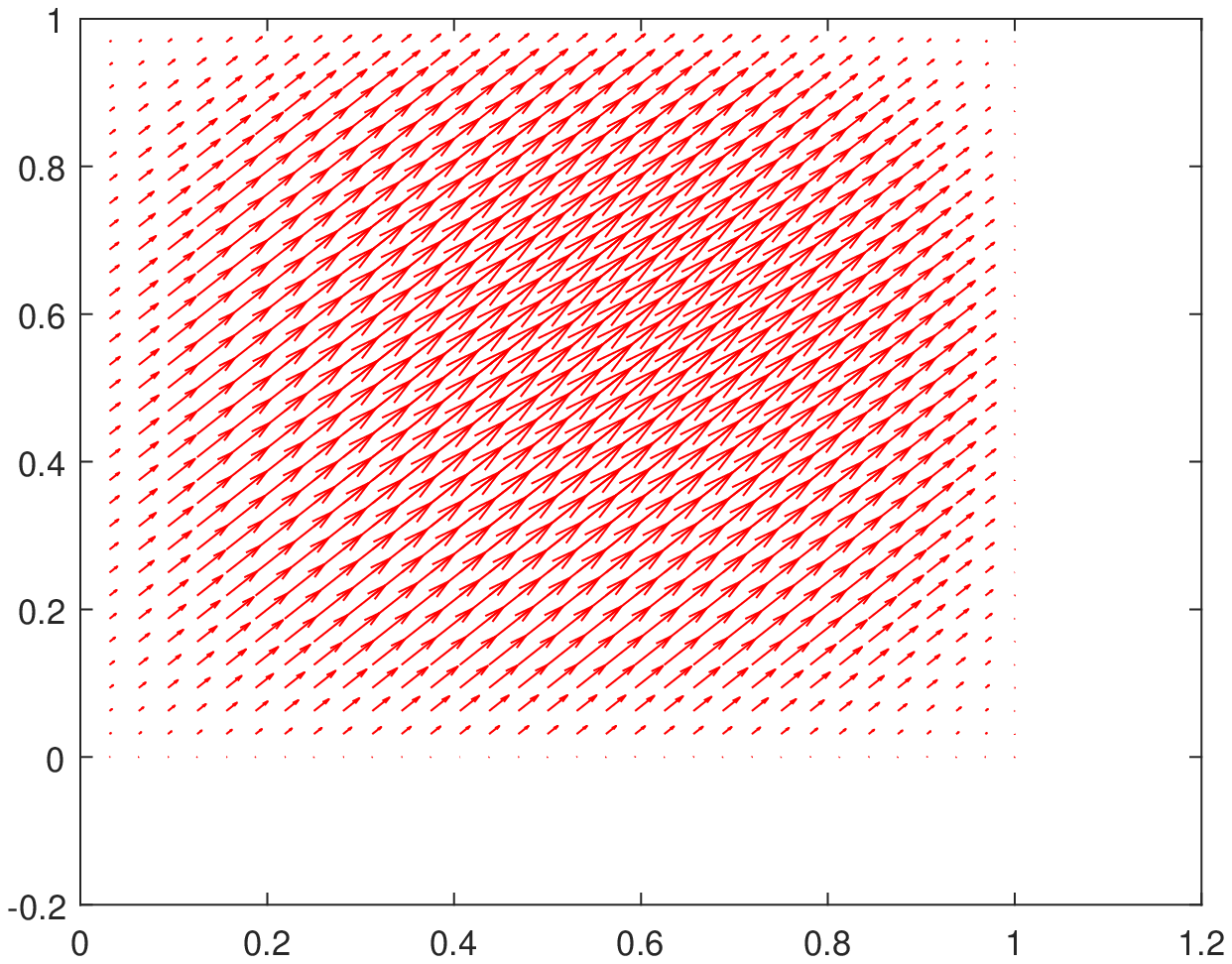}
		\caption{Arrow plot of the computed displacement $\textbf{u}_h^n$.}
		\label{fig8}
	\end{figure}
	\vspace{-0.8cm}
	\begin{figure}[H]
		\centering
		\subfigure[]{
			\centering
			\includegraphics[width=2.5in]{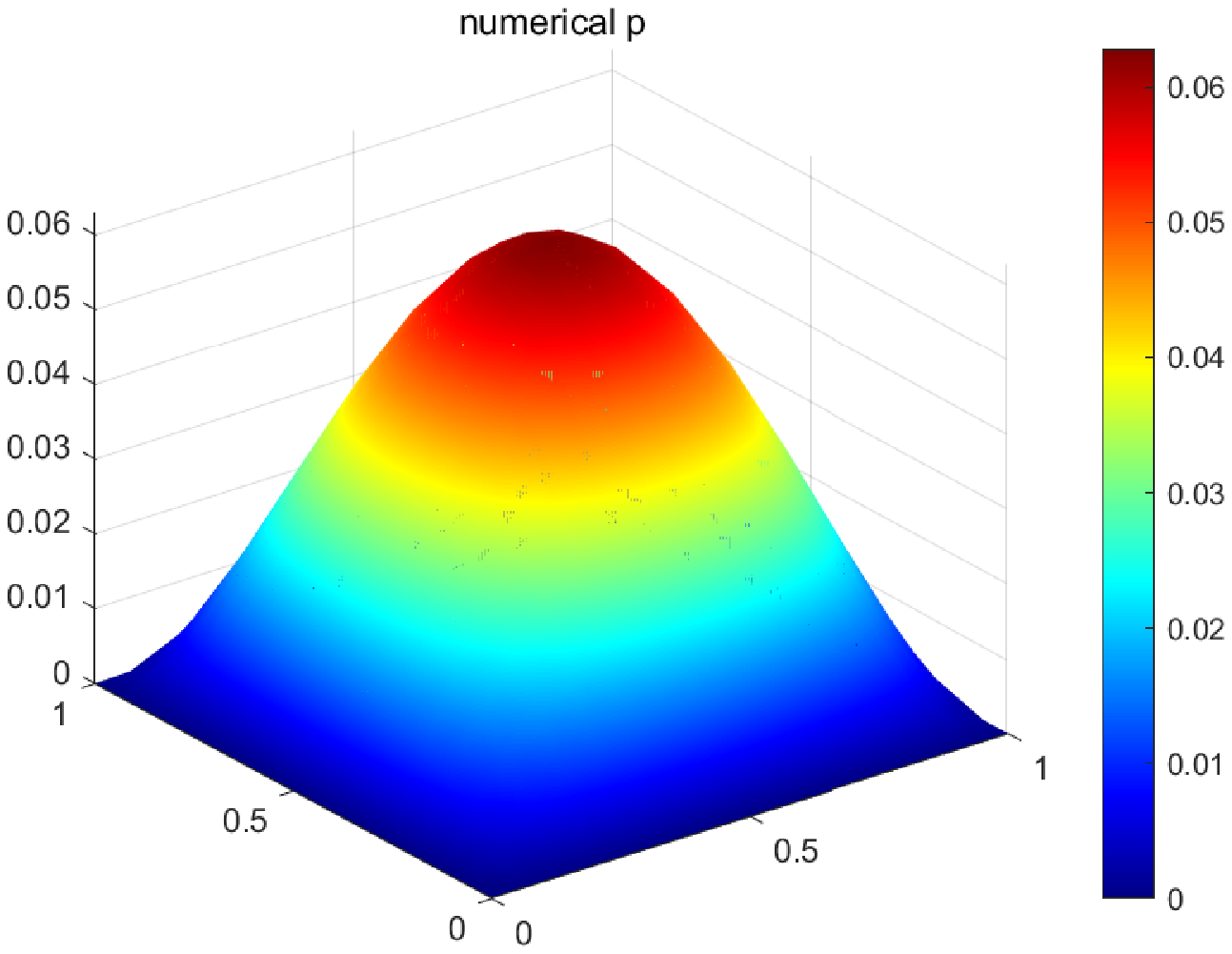}
			\label{fig9}}%
		\subfigure[]{
			\centering
			\includegraphics[width=2.5in]{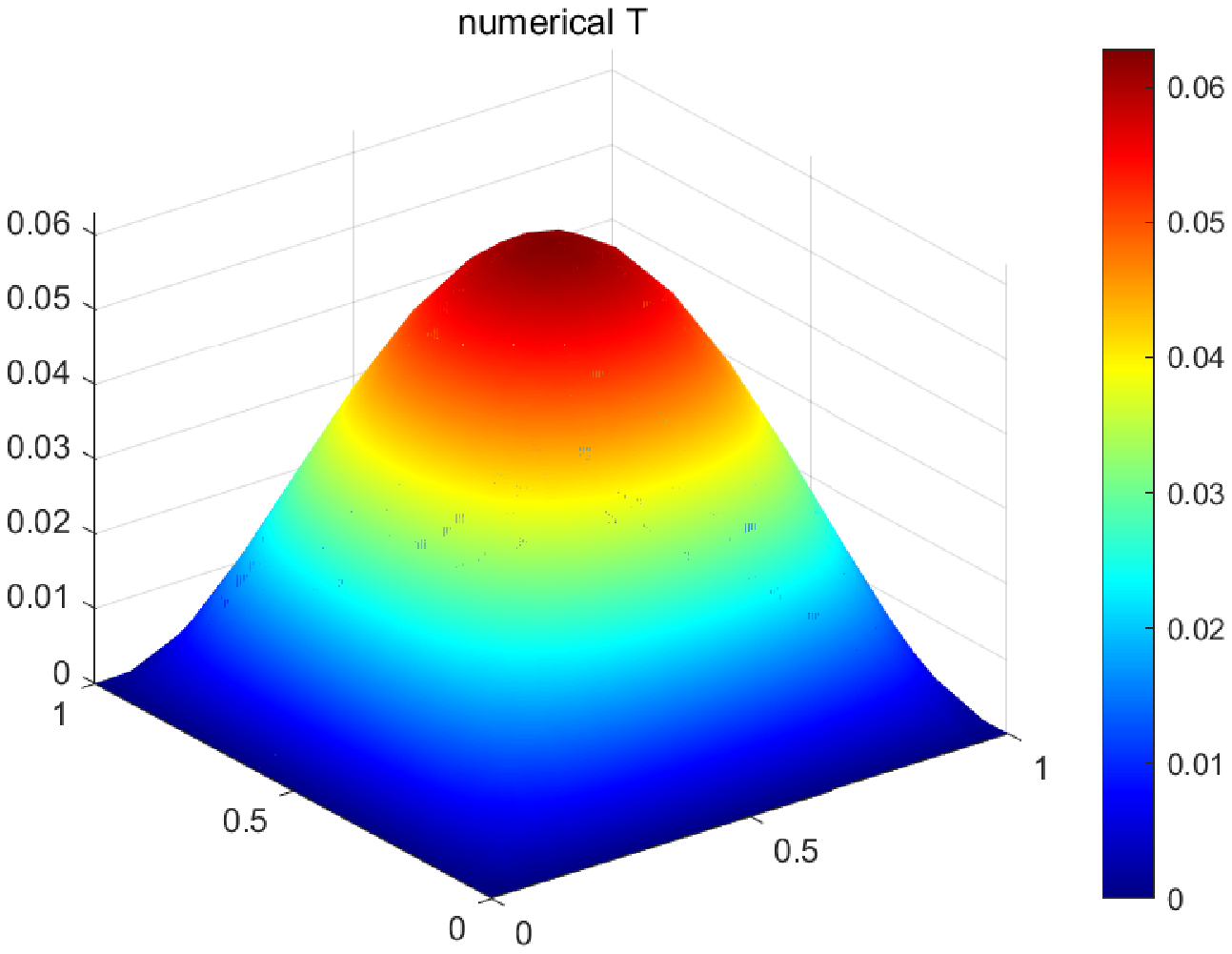}
			\label{fig10}}%
		\centering
		\caption{ (a) and (b) are surface plot of the pressure $p_h^n$ and temperature $T_h^n$ at the terminal time $\tau$ respectively.}
	\end{figure}
	Table \ref{table2} displays the $L^{\infty}(0,\tau;L^2(\Omega))$ and $L^{\infty}(0,\tau;H^1(\Omega))$-norm errors of $\textbf{u}$, $p$ and $T$ and shows that the convergence order with respect to $h$ is optimal, which verify the Theorem \ref{th-3-6} and Table \ref{table211} give the convergence order with respect to $\Delta t$ is optimal when $h=\frac{1}{8}$ and $\tau = 1$.

		Figure \ref{fig111} and Figure \ref{fig211} describe the spatial convergence order of $u^{n}_{h}, p^{n}_{h}$, $T^{n}_{h}$.
	Figure \ref{fig6}, Figure \ref{fig7} and Figure \ref{fig9} and Figure \ref{fig10} show, respectively, the surface plot of the computed $p^{n}_h$, $T^{n}_h$, $u^{n}_{1h}$ and $u^{n}_{2h}$ at the terminal time $\tau$, Figure \ref{fig8} show the arrow plot of the computed displacement $\textbf{u}$.
	
	\textbf{Test 2.}
	Let \ $\Omega= [0,1]\times[0,1]$. Let $\Gamma_{j}$ be same as in \textbf{Test 1} and $\tau=1e-4,\Delta t=1e-5$. We consider the problem (\ref{eq-2-7})-(\ref{eq-2-12}) with the following source functions:
	\begin{table}[H]
		\centering
		\caption{ Physical parameters}\label{table 5}
		\begin{tabular}{c l c }
			\hline
			Parameter   &  \quad Description      &   \quad   Value    \\
			\hline
			$a_0$       &\quad Effective thermal capacity               &  \quad 2e-1\\
			$b_0$       &\quad Thermal dilation coefficient              &  \quad 1e-1 \\
			$c_0$       &\quad Constrained specific storage coefficient &  \quad 2e-1 \\
			$\alpha$    &\quad Biot-Willis constant                     &  \quad  0.01 \\
			$\beta$     &\quad Thermal stress coefficient.               &  \quad  0.01\\
			$\bm{K}$         &\quad Permeability tensor                      &  \quad 1e-5$I$\\
			$\bm{\Theta}$     &\quad Effective thermal conductivity           &  \quad 1e-5$I$\\
			$E$         &\quad Young's modulus                          &  \quad  1.25e4\\
			$\nu$       &\quad Poisson ratio                            &  \quad  0.25\\
			\hline
		\end{tabular}
	\end{table}
	\begin{align*}
		\bm{{\rm f}_1} &=(\mu\pi^3+\frac{3\lambda\pi^3}{4}+(\alpha+\beta)\pi)e^t\cos(\pi x)\cos(\frac{\pi y}{2}),\\
		\bm{{\rm f}_2} &=(\frac{\mu\pi^3}{8}-\frac{3\lambda\pi^3}{8}-\frac{\pi}{2}(\alpha+\beta))e^t\sin(\pi x)\sin(\frac{\pi y}{2}),\\
		\phi &=(a_{0}-b_{0}+(\frac{3\pi^2\beta}{4})+\frac{5\pi^2\times10^{-5}}{4})e^t\sin(\pi x)\cos(\frac{\pi y}{2})\\
		&-10^{-5}\times((\pi e^t\cos(\pi x)\cos(\frac{\pi y}{2}))^2+(\frac{\pi}{2}e^t\sin(\pi x)\sin(\frac{\pi y}{2}))^2),\\
		g 
		&=(c_0-b_0-\frac{3\pi^2\alpha}{4}+\frac{5\pi^2\times10^{-5}}{4})e^t\sin(\pi x)\cos(\frac{\pi y}{2}),
	\end{align*}
	and the following boundary and initial conditions:
	\begin{eqnarray}
		p=e^t\mathrm{sin}(\pi x)\mathrm{cos}(\frac{\pi y}{2}) \quad  &\mathrm{on} & \quad\partial\Omega_{\tau},\nonumber\\
		T=e^t\mathrm{sin}(\pi x)\mathrm{cos}(\frac{\pi y}{2})  \quad &\mathrm{on} & \quad\partial\Omega_{\tau},\nonumber\\
		u_1=\pi e^t\mathrm{cos}(\pi x)\mathrm{cos}(\frac{\pi y}{2}) \quad &\mathrm{on}& \quad\Gamma_j\times (0,\tau),~j=2,~4,\nonumber\\
		u_2=\frac{\pi}{2} e^t\mathrm{sin}(\pi x)\mathrm{sin}(\frac{\pi y}{2}) \quad &\mathrm{on} & \quad\Gamma_j\times(0,\tau),~j=1,~3, \nonumber\\
		\sigma(\pmb{u})\bm{n}-\alpha pI\bm{n}=\textbf{f}_1 \quad  &\mathrm{on}& \quad\partial\Omega_{\tau},\nonumber\\
		\textbf{u}(x,y,0)=\pi \big(\mathrm{cos}(\pi x)\mathrm{cos}(\frac{\pi y}{2}),\frac{1}{2} \mathrm{sin}(\pi x)\mathrm{sin}(\frac{\pi y}{2}) \big)^{'}\quad&\mathrm{in}& \quad\Omega, \nonumber\\
		p(x,y,0)=\mathrm{sin}(\pi x)\mathrm{cos}(\frac{\pi y}{2}),~T(x,y,0)=\mathrm{sin}(\pi x)\mathrm{cos}(\frac{\pi y}{2}) \quad&\mathrm{in}& \quad\Omega,\nonumber
	\end{eqnarray}
	where $\bm{{\rm f}}_{1}=e^t\mathrm{sin}(\pi x)\mathrm{cos}(\frac{\pi y}{2})((-\mu \pi^2-\frac{3}{4}\pi^{2}\lambda-(\alpha+\beta))n_1,(\frac{\pi^{2}\mu}{4}-\frac{3\pi^2\lambda}{4}-(\alpha+\beta))n_2)^{'}$.
	
	It is easy to check that the exact solution are
	\begin{align*}
		\bm{{\rm u}}=(\pi e^t\mathrm{cos}(\pi x)\mathrm{cos}(\frac{\pi y}{2}), \frac{\pi}{2} e^t\mathrm{sin}(\pi x)\mathrm{sin}(\frac{\pi y}{2}))^{'},\\
		p=e^t\mathrm{sin}(\pi x)\mathrm{cos}(\frac{\pi y}{2}),\quad T=e^t\mathrm{sin}(\pi x)\mathrm{cos}(\frac{\pi y}{2}).
	\end{align*}

\begin{table}[htbp]
	\vspace{-2.0em}
	\begin{center}
		\caption{Error and convergence rates of $u_h^n$, $p_h^n$, $T_h^n$}\label{table 4}
		\resizebox{\textwidth}{12mm}{
			\begin{tabular}{ccccccccccccc}
				\hline
				$h$  & $\frac{\|e_u\|_{L^2(\Omega)}}{\|u\|_{L^2(\Omega)}}$  &  CR  &  $\frac{\|e_u\|_{H^1(\Omega)}}{\|u\|_{H^1(\Omega)}}$  &  CR & $\frac{\|e_p\|_{L^2(\Omega)}}{\|p\|_{L^2(\Omega)}}$ & CR  &  $\frac{\|e_p\|_{H^1(\Omega)}}{\|p\|_{H^1(\Omega)}}$  &  CR  &  $\frac{\|e_T\|_{L^2(\Omega)}}{\|T\|_{L^2(\Omega)}}$  &  CR  & $\frac{\|e_T\|_{H^1(\Omega)}}{\|T\|_{H^1(\Omega)}}$ & CR \\ 
				\hline
				$1/4$   &0.0115&      &0.0390    &      &0.0488    &      &0.2983&      &0.0488    &      &0.2983&    \\
				$1/8$   &0.0014&  3.0164    &0.0093    &   2.0713   &0.0106    &   2.2047   &0.1475& 1.0158     &0.0106    &  2.2047    &0.1475& 1.0158   \\
				$1/16$  &1.7799e-04&3.0018&0.0023    &2.0412&0.0024    & 2.1152 &0.0733&1.0100&0.0024    & 2.1152&0.0733&1.0100\\
				$1/32$  &2.2250e-05&2.9999&5.5480e-04& 2.0218&5.8491e-04& 2.0616 &0.0365&1.0047&5.8491e-04&2.0616&0.0365&1.0047\\
				$1/64$  &2.7819e-06&2.9997 &1.3763e-04& 2.0111  &1.4302e-04&2.0320&0.0182& 1.0022 &1.4302e-04&2.0320&0.0182&1.0022\\
				\hline
		\end{tabular}}
	\end{center}
\end{table}
\begin{table}[htbp]
	\vspace{-1.0em}
	\begin{center}
		\caption{Order of convergence of time discretization of Test 1 }\label{table311}
		\resizebox{\textwidth}{12mm}{
		\begin{tabular}{ccccccccccc}
		\hline
		$\Delta t$ & $\left\|\mathbf{u}_{h}^{\Delta t}-\mathbf{u}_{h}^{\frac{1}{2} \Delta t}\right\|_{L^2(\Omega)}$ & $\rho_{\Delta t, \mathbf{u}_{h}}$ & $\left\|p_{h}^{\Delta t}-p_{h}^{ \frac{1}{2} \Delta t}\right\|_{L^2(\Omega)}$ & $\rho_{\Delta t, p_h}$ & $\left\|T_{h}^{\Delta t}-T_{h}^{\frac{1}{2} \Delta t}\right\|_{L^2(\Omega)}$ & $\rho_{\Delta t, T_h}$ \\
		\hline
		$\frac{1}{10}$ & 3.1223e-10 &        & 0.0055 &        & 0.0055 &  \\
		$\frac{1}{20}$ &1.5704e-10 & 1.9882 & 0.0027 & 2.0244 & 0.0027 & 2.0245\\
		$\frac{1}{40}$ & 7.8750e-11 & 1.9942 & 0.0013 & 2.0123 &0.0013 & 2.0123\\
		$\frac{1}{80}$ &3.9434e-11 & 1.9970 & 6.7231e-04 & 2.0062 & 6.7253e-04 & 2.0062\\
		\hline
\end{tabular}}
	\end{center}
\end{table}
	\begin{figure}[htbp]
		\subfigure[]{
			\centering
			\includegraphics[width=2.5in]{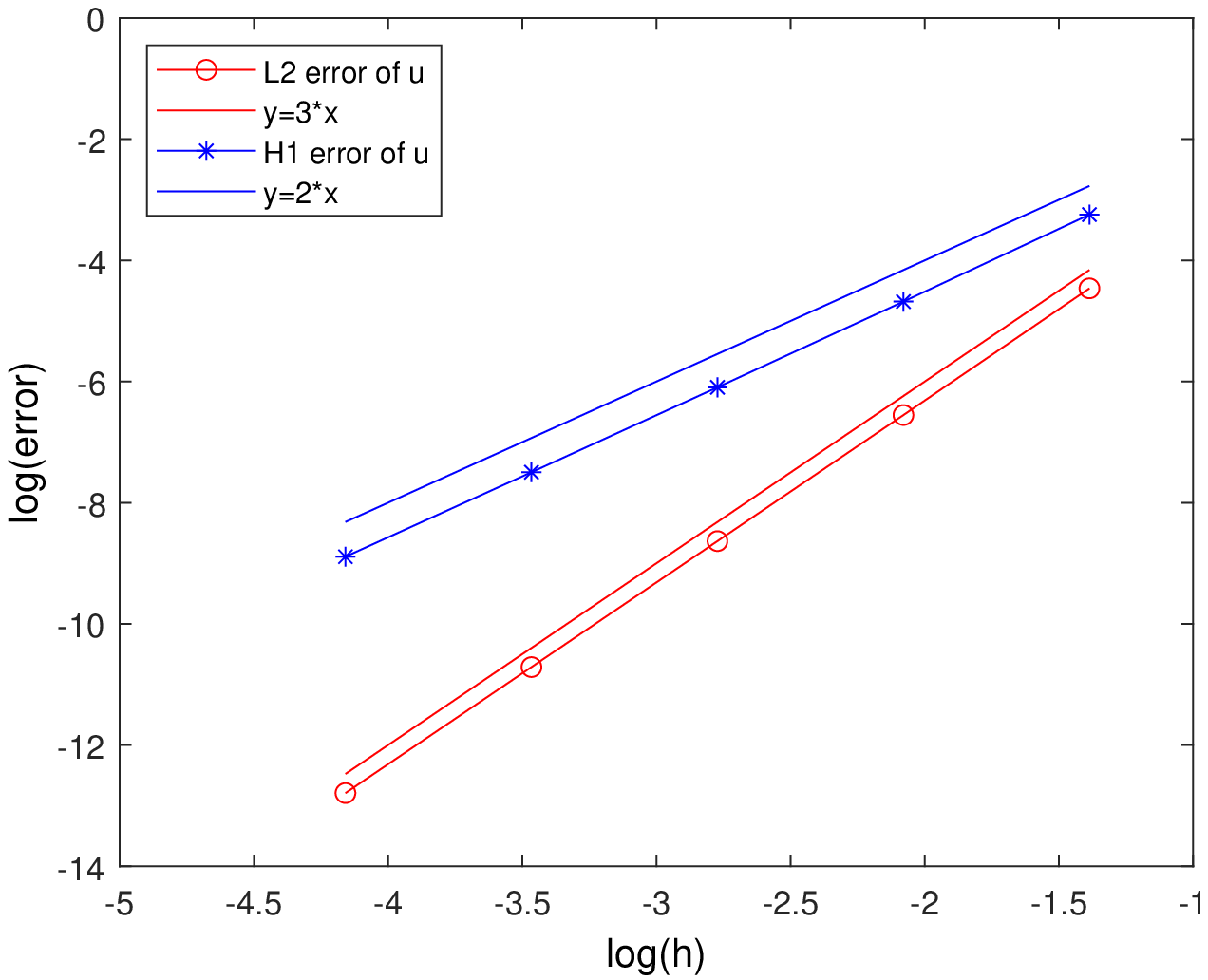}
			\label{fig1}
		}
		\subfigure[]{
			\centering
			\includegraphics[width=2.5in]{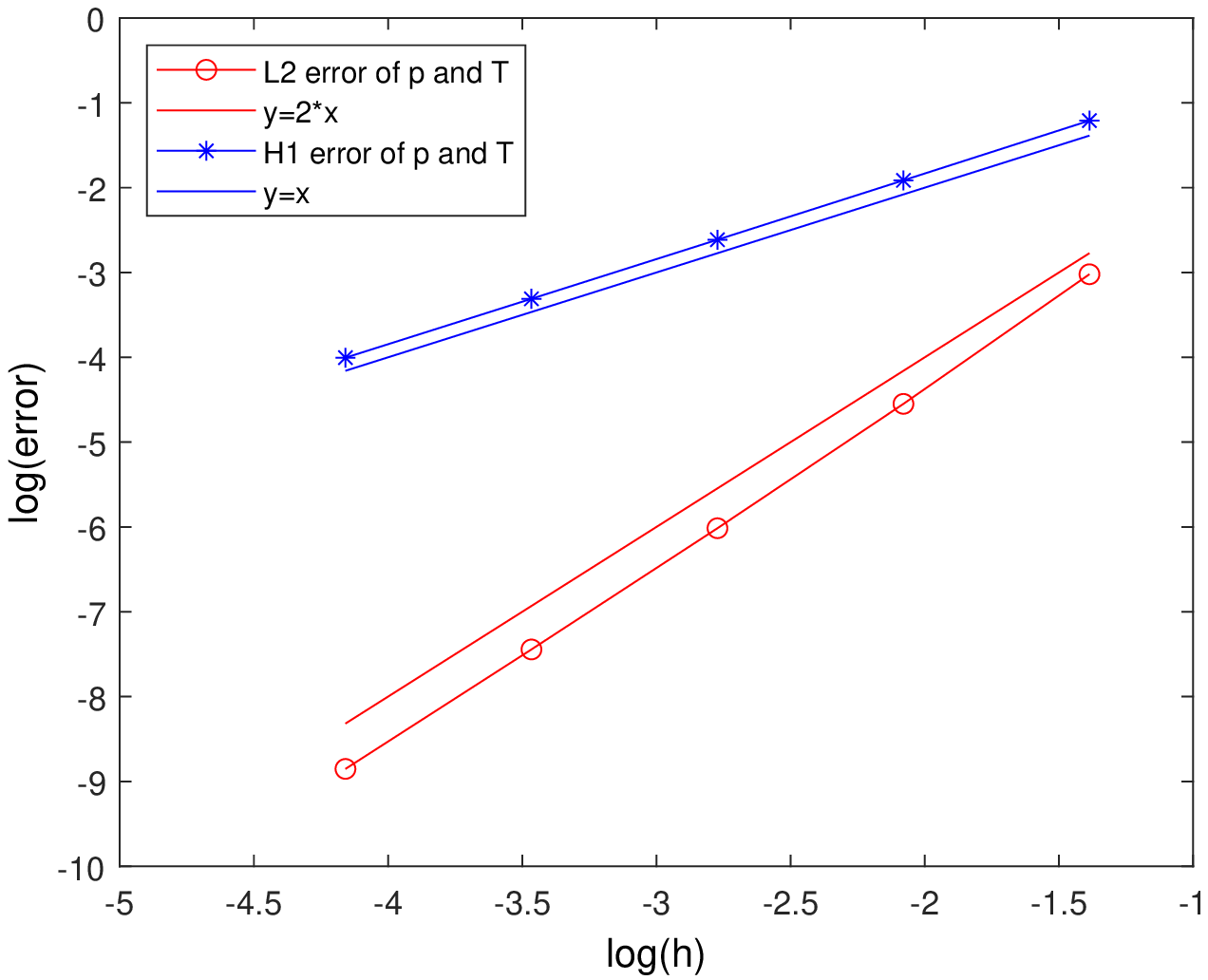}
			\label{fig2}
		}
		\caption{(a) spatial convergence order for $u^{n}_{h}$, (b) space convergence rate for $p^{n}_{h}, T^{n}_{h}$.}
	\end{figure}
	\begin{figure}[H]
		\subfigure[]{
			\centering
			\includegraphics[width=2.5in]{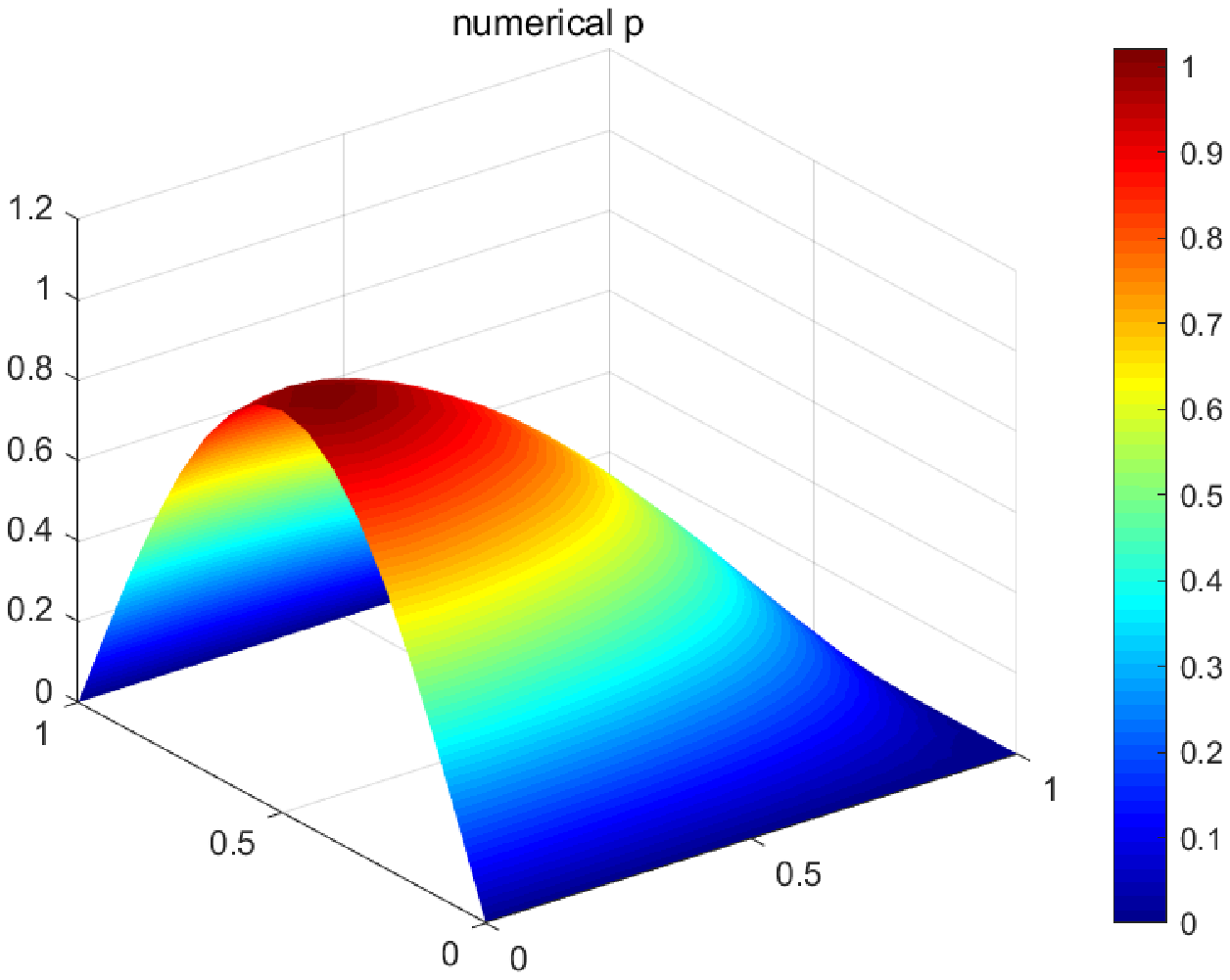}
			\label{fig51}
		}
		\subfigure[]{
			\centering
			\includegraphics[width=2.5in]{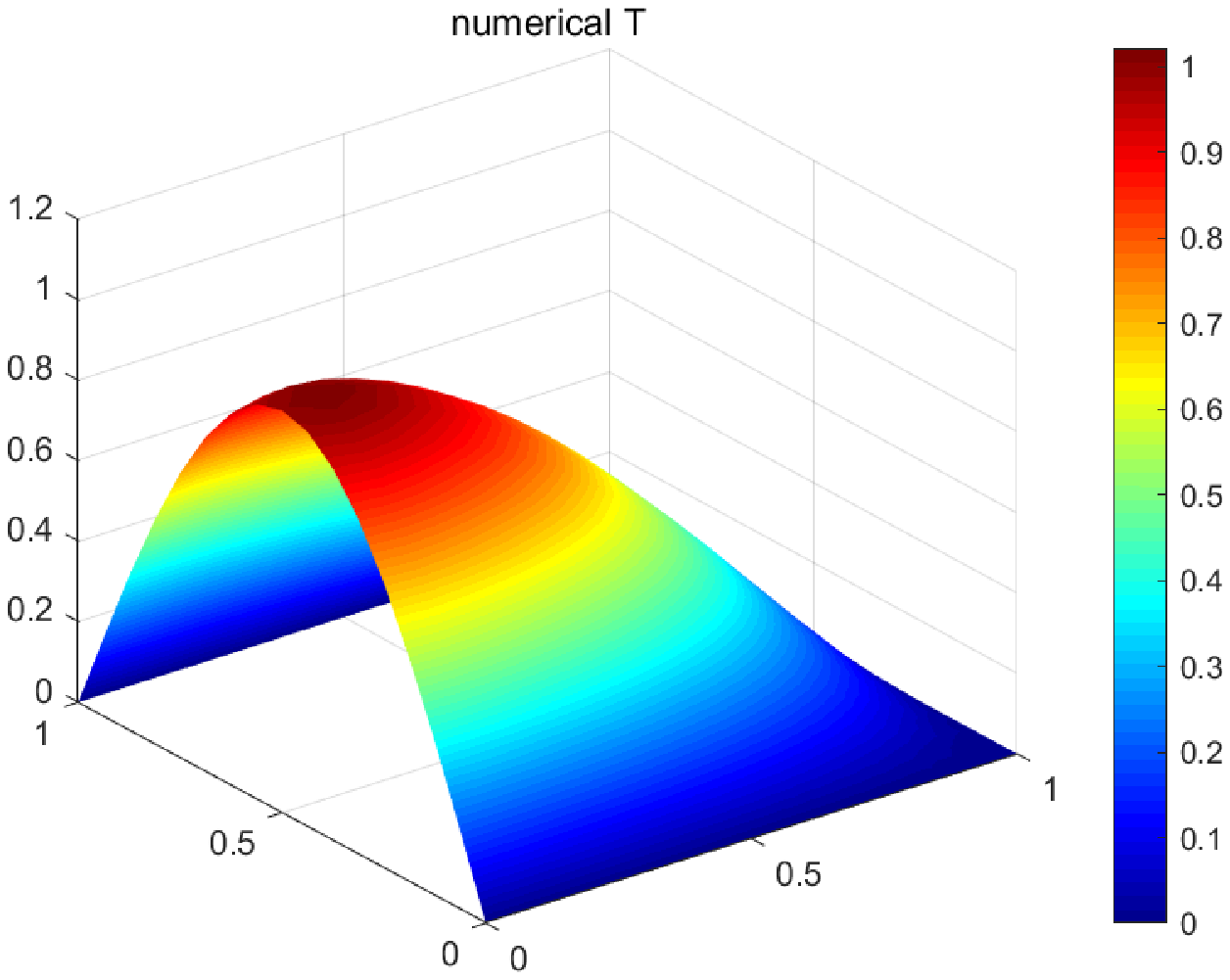}
			\label{fig61}
		}
		\caption{(a) surface plot of $p_h^n$ at the terminal time $\tau$, (b) surface plot of $T_h^n$ at the terminal time $\tau$.}
	\end{figure}
	
	\begin{figure}[H]
		\centering
		\subfigure[]{
			\centering
			\includegraphics[width=2.5in]{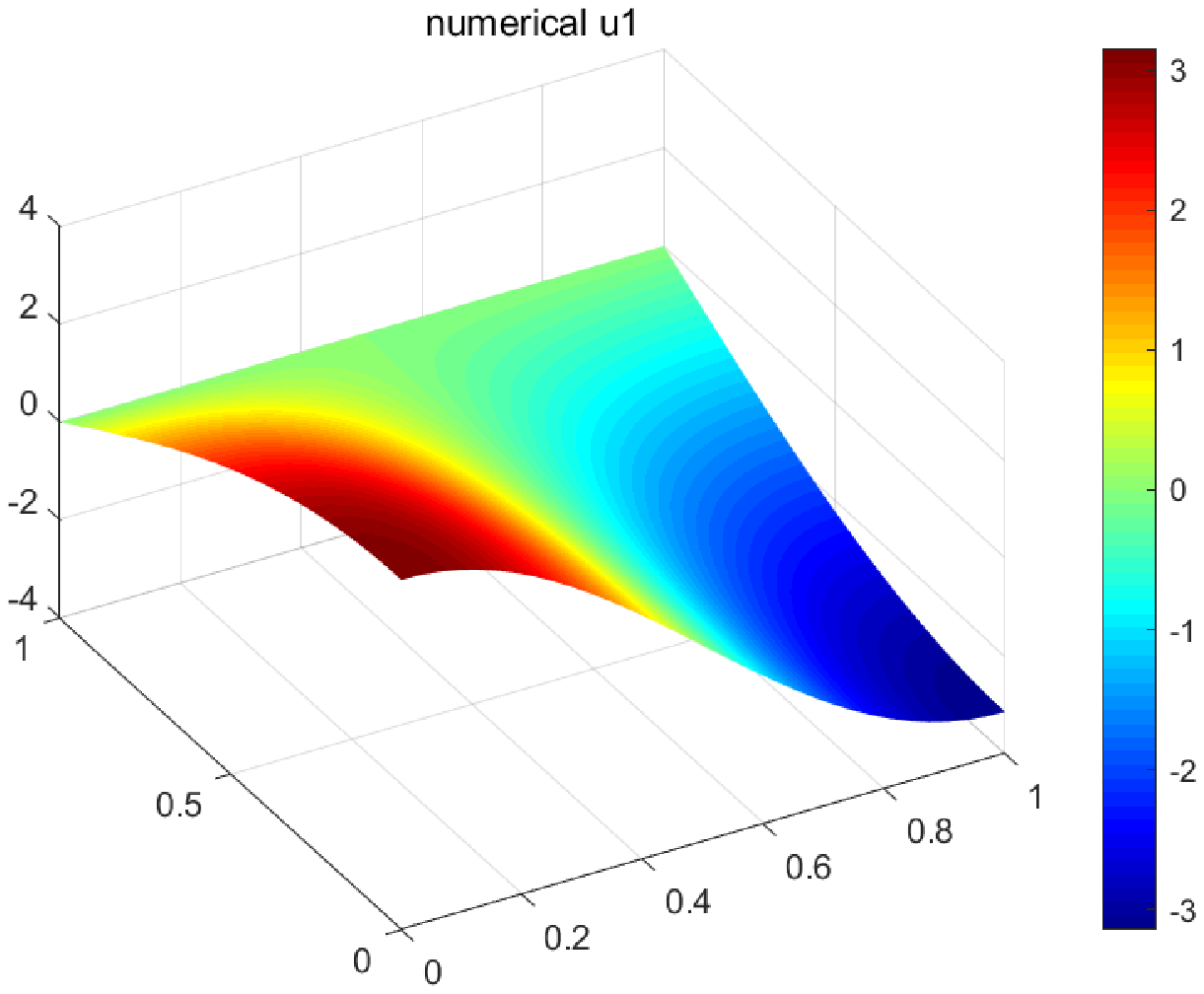}
			\label{fig81}}%
		\subfigure[]{
			\centering
			\includegraphics[width=2.5in]{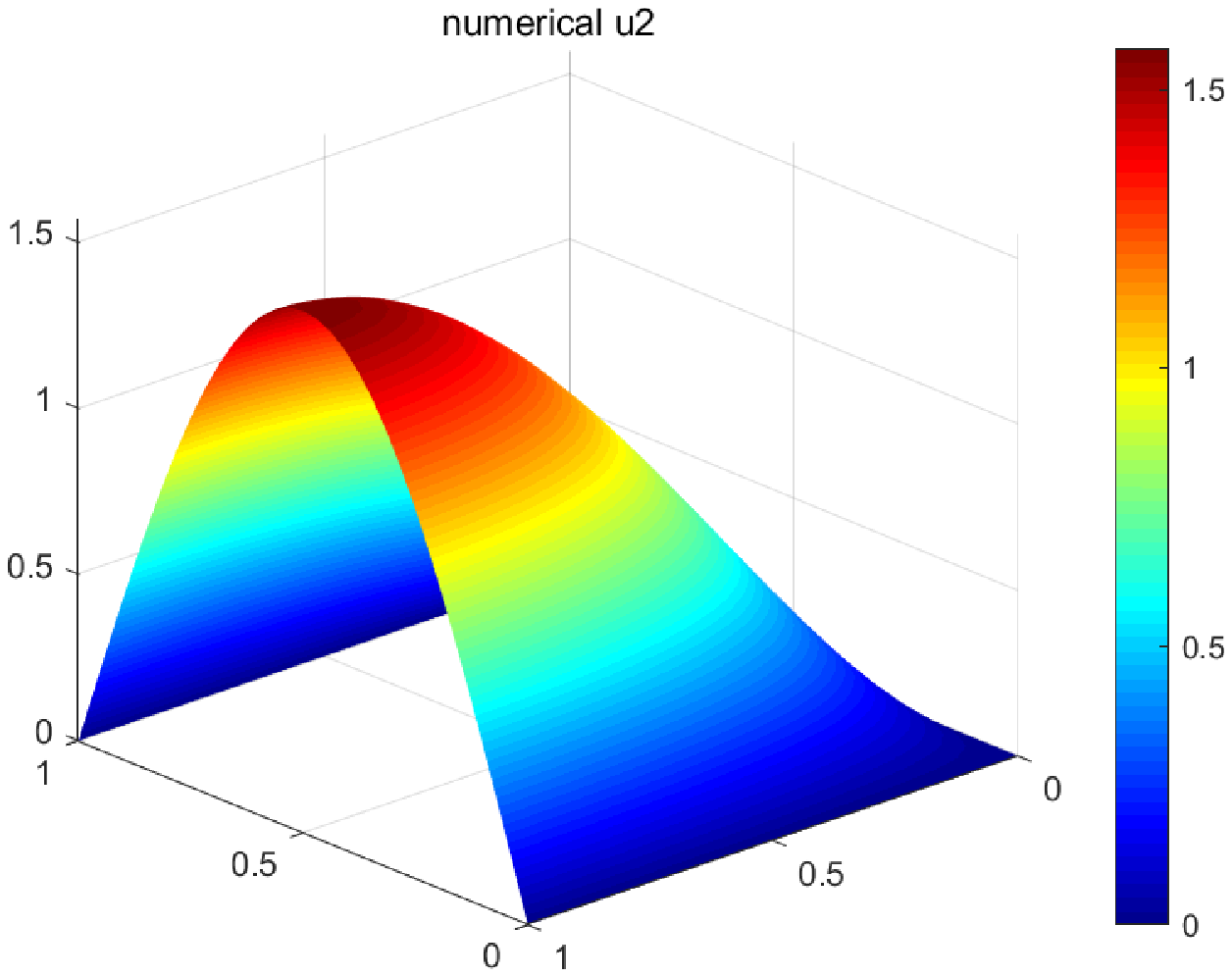}
			\label{fig91}}%
		\centering
		\caption{(a) surface plot of $u_{1h}^{n}$ at the terminal time $\tau$, (b) surface plot of $u_{2h}^{n}$ at the terminal time $\tau$.}
		\vspace{-2em}
	\end{figure}
\begin{figure}[H]
	\centering
	\includegraphics[height=5cm,width=7cm]{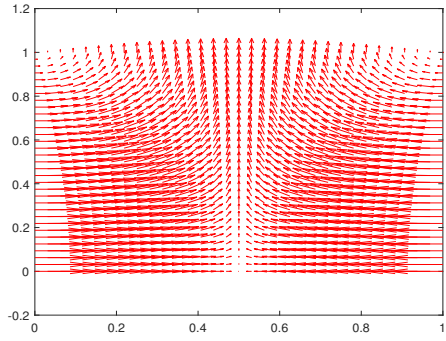}
	\caption{Arrow plot of the displacement $\textbf{u}_h^n$.}
	\label{fig71}
\end{figure}
	
	Table \ref{table2} displays the $L^{\infty}(0,T;L^{2}(\Omega))$-norm error and $L^{\infty}(0,T;H^{1}(\Omega))$-norm error of $u$, $p$, $T$ and the convergence order with respect to $h$ at terminal time $\tau$. Evidently, the spatial rates of convergence are consistent with Theorem \ref{th-3-6}.
	
	Figure \ref{fig1} and Figure \ref{fig2} describe the spatial convergence order of $u^{n}_{h}, p^{n}_{h}$, $T^{n}_{h}$. Figure \ref{fig51}, Figure \ref{fig61}, Figure \ref{fig81} and Figure \ref{fig91} show, respectively, the surface plot of $p^{n}_{h}$, $T^{n}_{h}$, $u^{n}  _{1h}$ and $u^{n}_{2h}$ at the terminal time $\tau$ and Figure \ref{fig71} shows arrow plot of  $\textbf{u}_h^n$. They coincide with the theoretical results.

	\textbf{Test 3.} 
	In this test, we consider Barry-Mercer's problem (cf. \cite{Wheeler2007}). The computational domain is $\Omega=[0,1]\times[0,1]$, \ $\Gamma_{1} = \{(1,y);~0\leq y\leq1\}$,\ $~\Gamma_{2}= \{(x,0);~0\leq x\leq1 \}$,\ $\Gamma_{3}= \{(0,y);~0\leq y\leq1 \}$,\ $~\Gamma_{4}= \{ (x,1);~0\leq x\leq1 \}$,and $\tau=1$. Barry-Mercer's problem has no source, that is, $\textbf{f}\equiv0$ and $g\equiv0$, we prescribe homogeneous boundary conditions and zero source term and initial condition for the heat problem and takes the following boundary and initial conditions
	\begin{eqnarray*}
		p=0  \quad &\mathrm{on}&\quad \Gamma_j\times (0,\tau),~j=1,~3,~4,\\
		p=p_2 \quad&\mathrm{on}& \quad\Gamma_j\times (0,\tau),~j=2,\\
		T=0 \quad &\mathrm{on}&\quad \Gamma_j\times (0,\tau),~j=1,~3,~4,\\
		T=T_2 \quad&\mathrm{on}& \quad\Gamma_j\times (0,\tau),~j=2,\\
		u_1=0\quad&\mathrm{on}& \quad\Gamma_j\times (0,\tau),~j=2,~4,\\
		u_2=0 \quad &\mathrm{on} & \quad\Gamma_j\times(0,\tau),~j=1,~3, \\
		\sigma(\pmb{\tau})\bm{n}-\alpha pI\bm{n}=\textbf{f}_1:=(0,\alpha p+\beta T)^{'} \quad &\mathrm{on}&\quad \partial\Omega_{\tau},\\
		\textbf{u}(x,0)=\bm{0},~ p(x,0)=0,~T(x,0)=0 \quad&\mathrm{in}&\quad \Omega,\nonumber
	\end{eqnarray*}
	where
	\begin{equation*}
		p_{2}\left(x, t\right)=\left\{\begin{array}{ll}
			\sin t &  { \rm if }\  x \in[0.2,0.8) \times(0, T), \\
			0 & \text { others. }
		\end{array}\right.
		T_{2}\left(x, t\right)=\left\{\begin{array}{ll}
			
			\sin t &  { \rm if }\  x \in[0.2,0.8) \times(0, T), \\
			0 & \text { others. }
		\end{array}\right.
	\end{equation*}
	\begin{table}[H]
		\centering
		\caption{ Physical parameters}\label{table9}
		\begin{tabular}{c l c }
			\hline
			Parameter    &\quad Description                              &  \quad Value    \\
			\hline
			$a_0$        &\quad Effective thermal capacity               &  \quad 0 or 1e-10\\
			$b_0$        &\quad Thermal dilation coefficient             &  \quad 0\\
			$c_0$        &\quad Constrained specific storage coefficient &  \quad 1e-10 or 0\\
			$\alpha$     &\quad Biot-Willis constant                     &  \quad 1 \\
			$\beta$      &\quad Thermal stress coefficient              &  \quad 1\\
			$\bm{K}$     &\quad Permeability tensor                      &  \quad 1e-7$I$\\
			$\bm{\Theta}$&\quad Effective thermal conductivity           &  \quad 1e-7$I$\\
			$E$          &\quad Young's modulus                          &  \quad  1.25e6\\
			$\nu$        &\quad Poisson ratio                            &  \quad  0.25\\
			\hline
		\end{tabular}
	\end{table}
	
	\begin{figure}[H]
		\subfigure[]{
			\centering
			\includegraphics[width=2.3in]{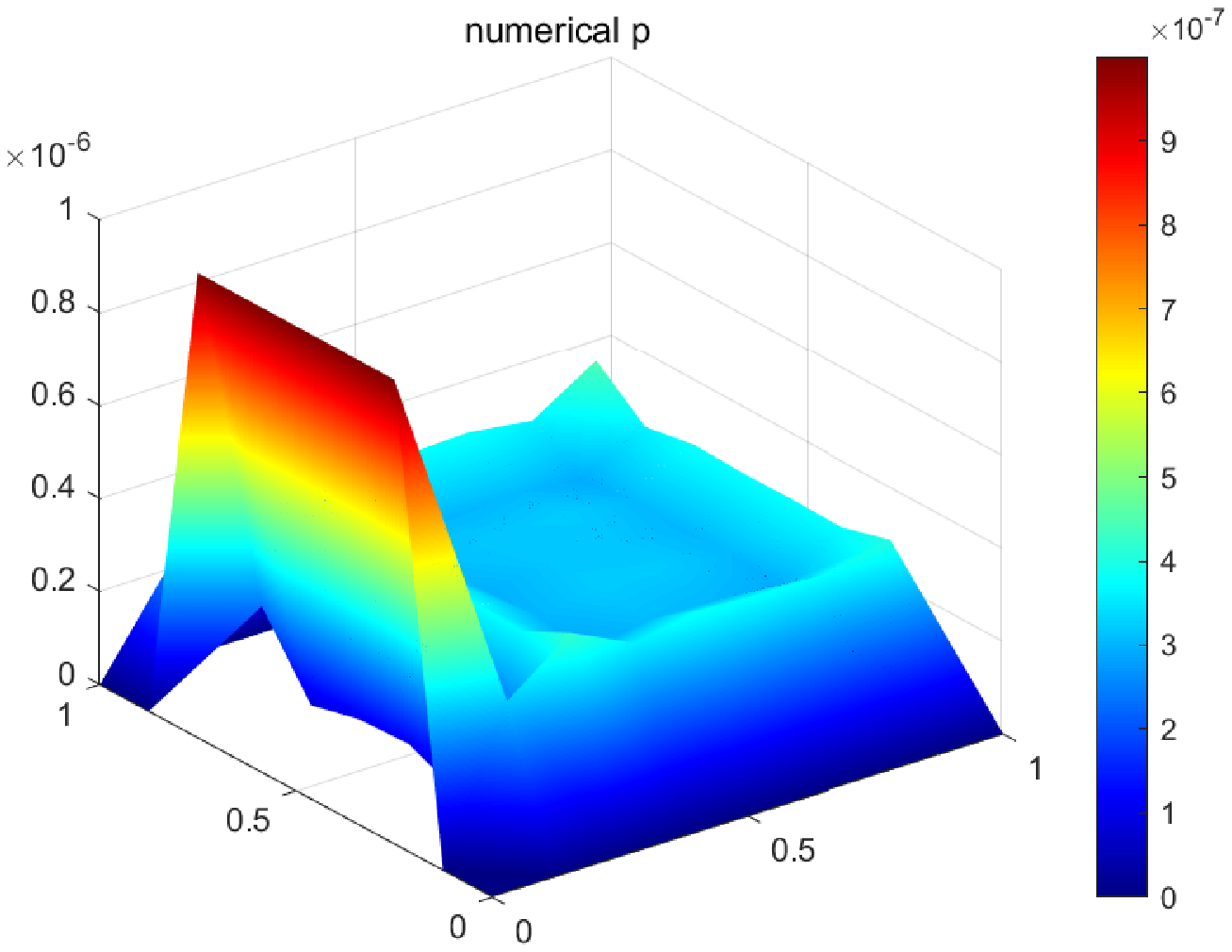}
			\label{fig11}
		}
		\subfigure[]{
			\centering
			\includegraphics[width=2.3in]{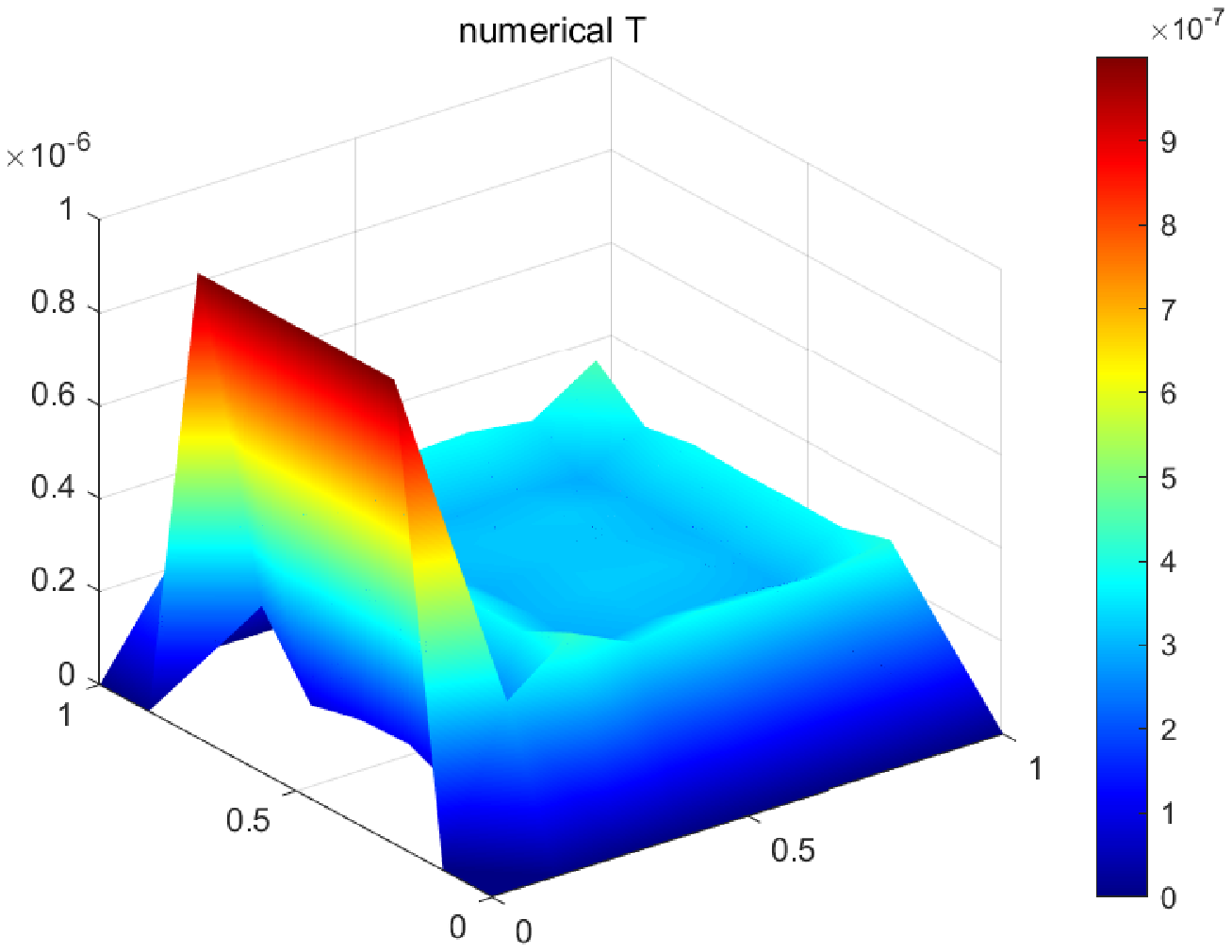}
			\label{fig12}}\caption {The numerical results by using $P_2-P_1-P_1$ element pair for the variables of  $\textbf{u}, p$ and $T$ of the problem \reff{eq-2-1}-\reff{eq-2-3}: (a) locking in pressure field, (b) locking in temperature field.}
	\end{figure}
	
	\begin{figure}[htbp]
		\subfigure[]{
			\centering
			\includegraphics[width=2.3in]{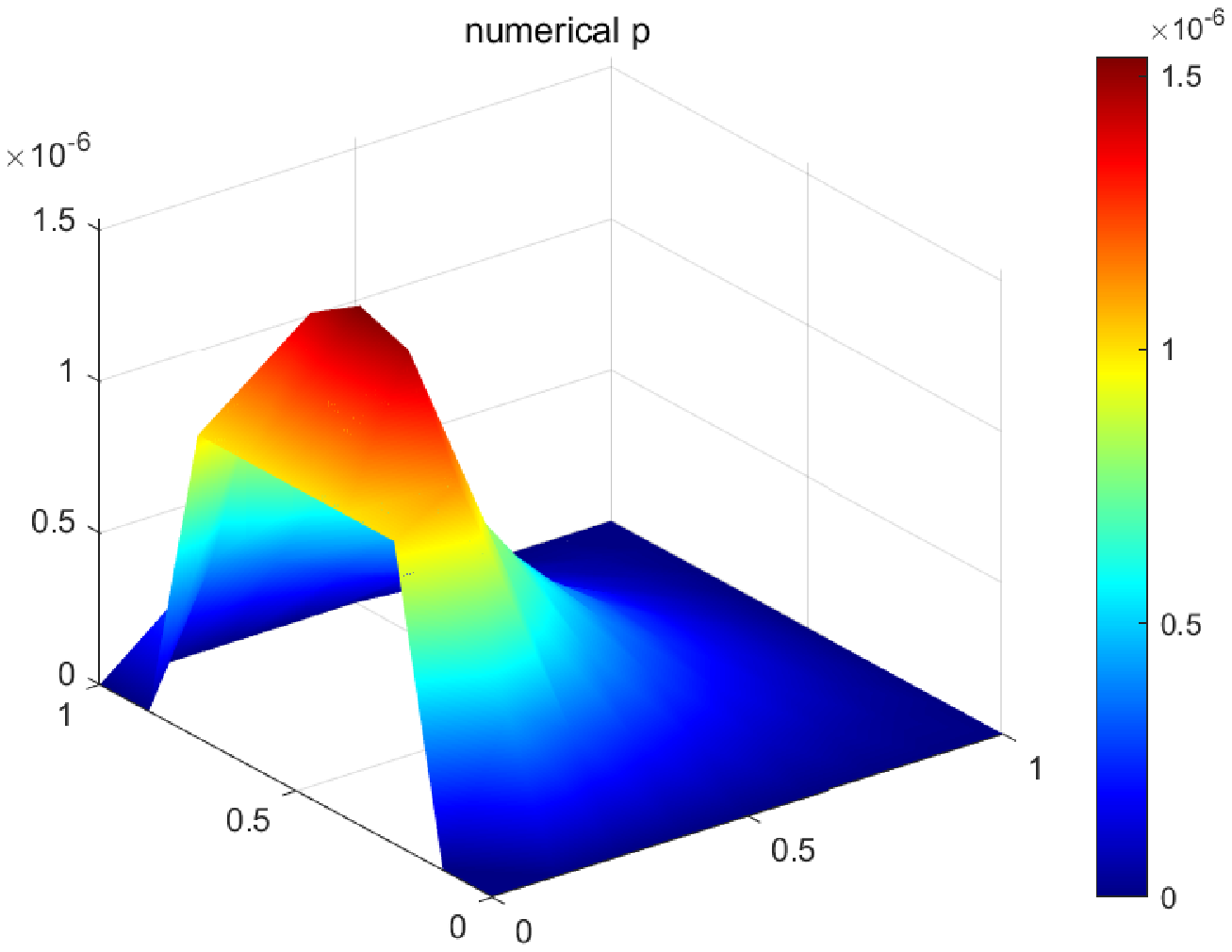}
			\label{fig19}
		}
		\subfigure[]{
			\centering
			\includegraphics[width=2.3in]{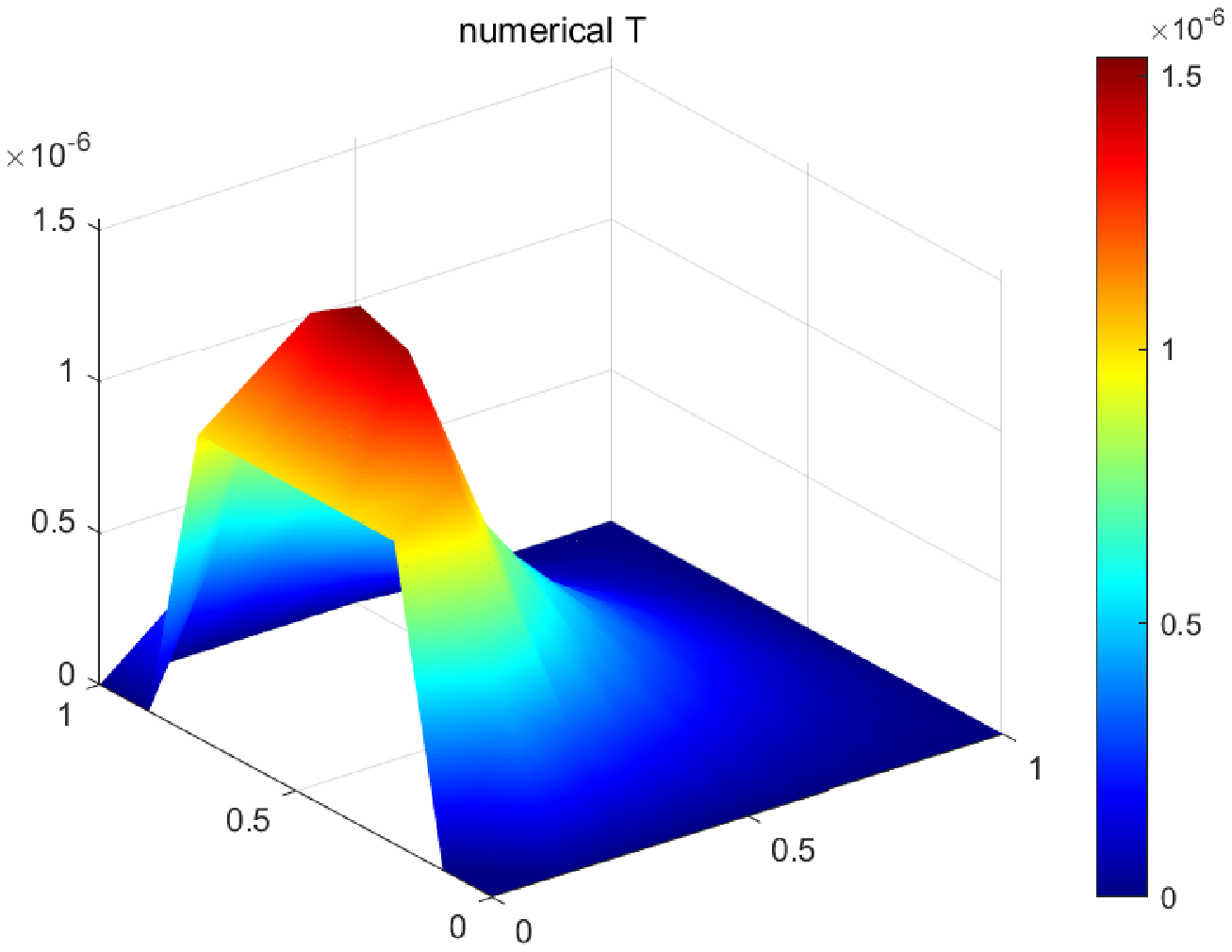}
			\label{fig20}}
		\caption{The numerical results by using $P_2-P_1-P_1-P_1$ element pair for the variables of  $\textbf{u}, \xi, \eta$ and $\gamma$ of the proposed MFEM: (a) no locking in the pressure field, (b) no locking in the temperature field.}
	\end{figure}
	\begin{figure}[H]
		\centering
		\subfigure[]{
			\centering
			\includegraphics[width=2.5in]{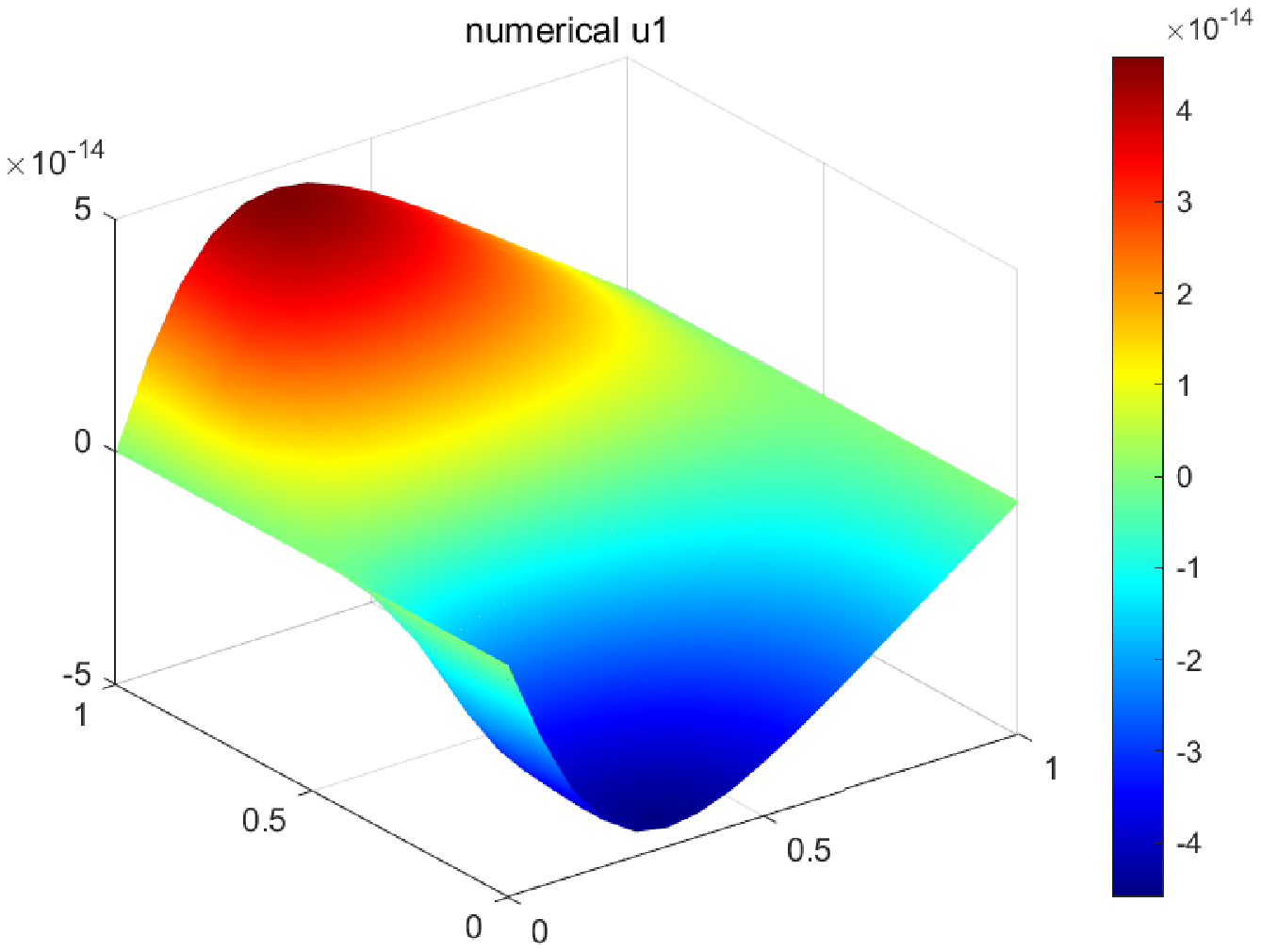}
			\label{fig117}}%
		\subfigure[]{
			\centering
			\includegraphics[width=2.5in]{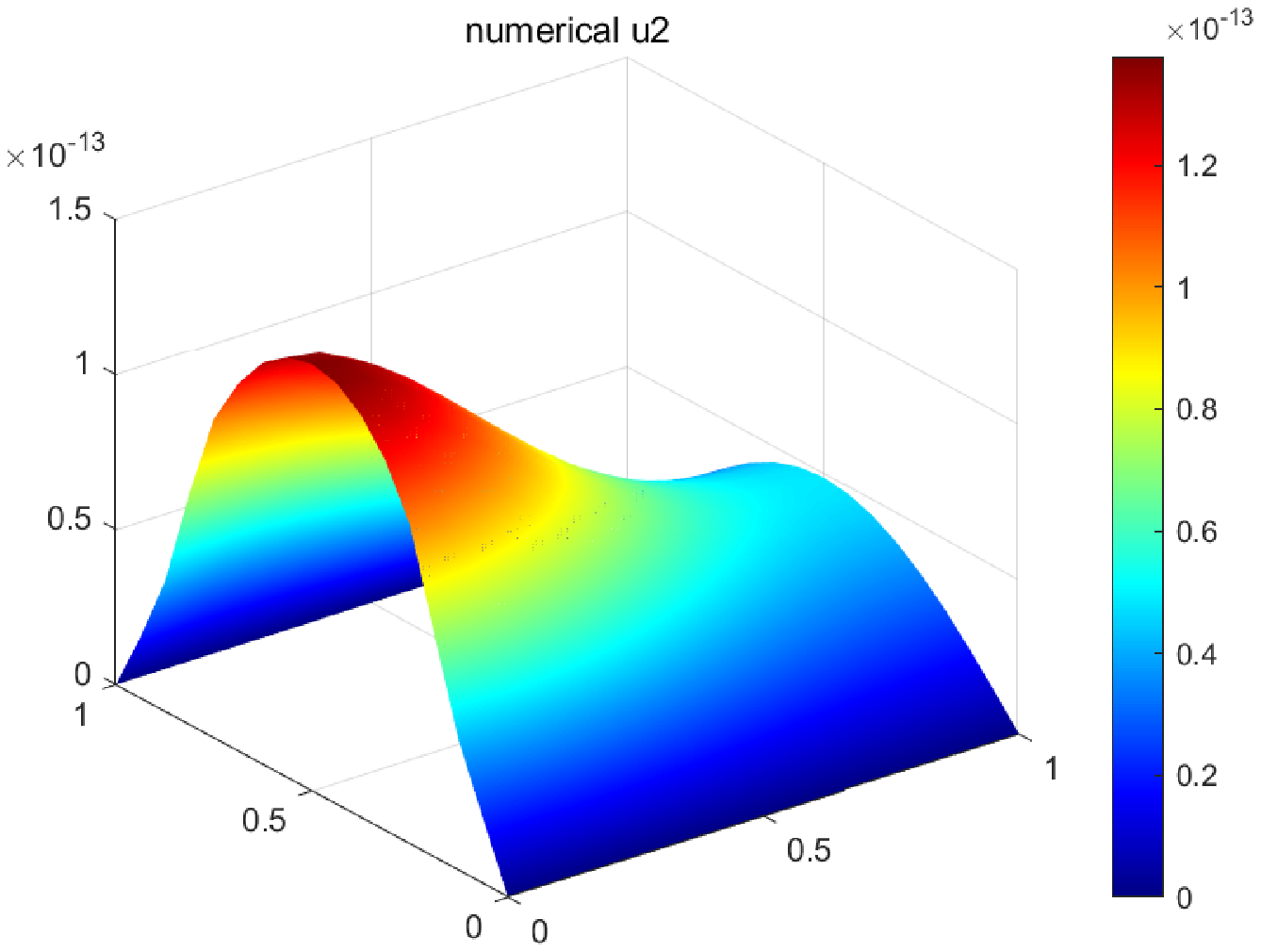}
			\label{fig118}}%
		\centering
		\caption{ (a) and (b) are the surface plot of  $u_{1h}^{n}$ and $u_{2h}^{n}$  at time $\tau$, respectively.}
	\end{figure}
	\begin{figure}[H]
		\centering
		\includegraphics[height=5cm,width=7cm]{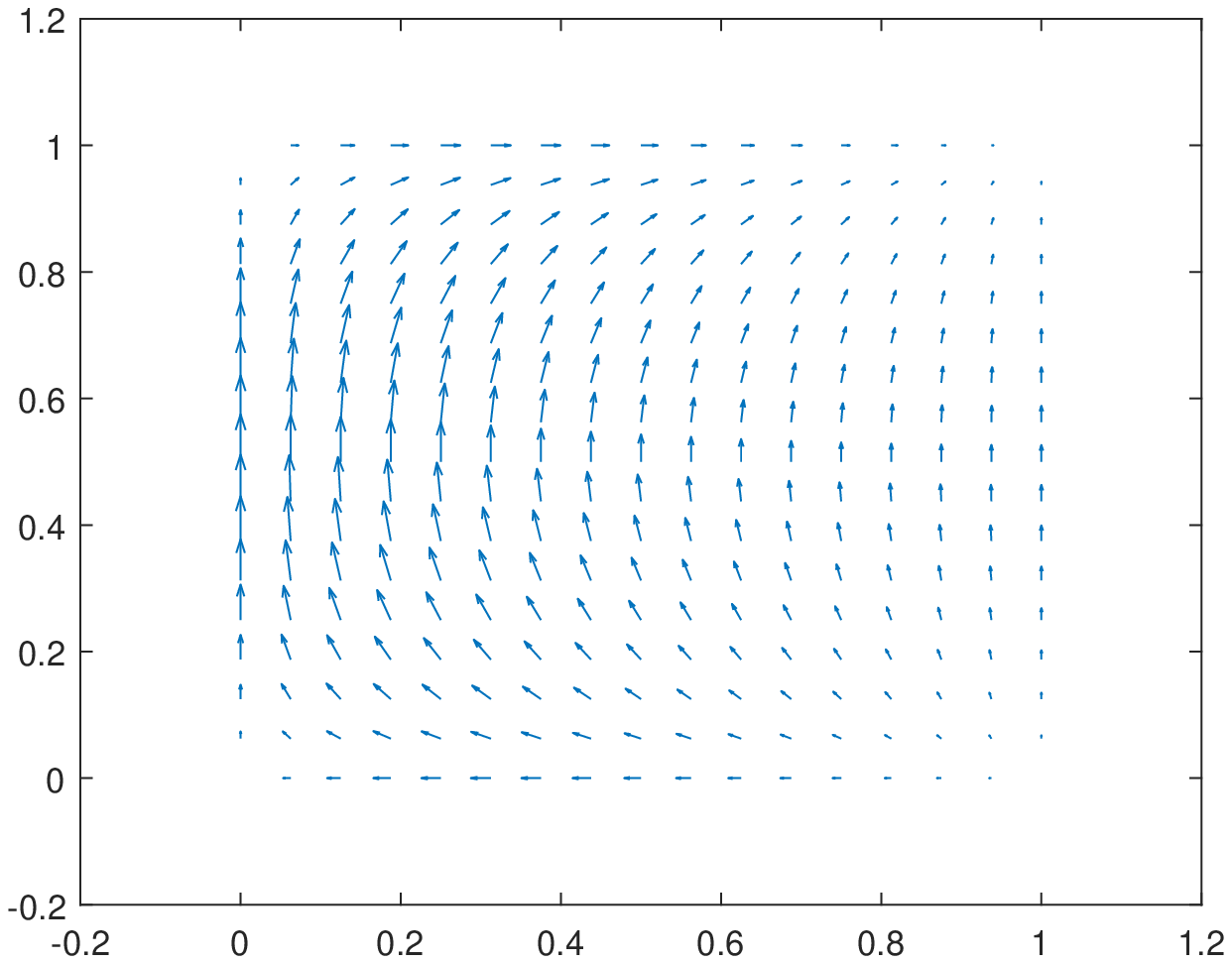}
		\caption{Arrow plot of the displacement $\textbf{u}_h^n$.}
		\label{fig119}
	\end{figure}
	Figure \ref{fig11} shows that pressure oscillations  occur by using $P_2-P_1-P_1$ element pair for the variables of  $\textbf{u}, p$ and $T$ of the problem \reff{eq-2-1}-\reff{eq-2-3} when $c_{0}=0, b_{0}=0$, and the permeability is very small for very short times, Figure \ref{fig12} shows that temperature oscillations occur by using $P_2-P_1-P_1$ element pair for the variables of  $\textbf{u}, p$ and $T$ of the problem \reff{eq-2-1}-\reff{eq-2-3} when $a_{0}=0, b_{0}=0$, and the thermal conductivity is very small for very short times. From Figure \ref{fig19} and Figure \ref{fig20}, we see that there is no locking phenomenon, which confirms that our approach and numerical methods have a built-in mechanism to prevent the "locking phenomenon". Figure \ref{fig117} and Figure \ref{fig118} display the surface plot of $u^{n}_{1h}$ and $u^{n}_{2h}$ at the terminal time $\tau$, Figure \ref{fig119} show the arrow plot of the computed displacement $\textbf{u}^{n}_{h}$.

	\textbf{Test 4.} 
	In order to check the robustness of the proposed schemes with respect to nonlinearity, we select the following parameters, in order to make this term dominate. 
	This test problem is same as \textbf{Test 1.}
		Table \ref{table201} displays the $L^{\infty}(0,T;L^{2}(\Omega))$-norm error and $L^{\infty}(0,T;H^{1}(\Omega))$-norm error of $u$, $p$, $T$ and the convergence order with respect to $h$ at terminal time $\tau=1$. Evidently, the spatial rates of convergence are consistent with Theorem \ref{th-3-6}.Figure \ref{fig1111}, Figure \ref{fig1112}, Figure \ref{fig1114} and Figure \ref{fig1117} show, respectively, the surface plot of $p^{n}_{h}$, $T^{n}_{h}$, $u^{n}  _{1h}$ and $u^{n}_{2h}$ at the terminal time $\tau$ and Figure \ref{fig1113} shows arrow plot of  $\textbf{u}_h^n$. They coincide with the theoretical results. We also compare the results when no stabilization is applied.
	 \begin{table}[H]
	 	\centering
	 	\caption{ Physical parameters}\label{table100}
	 	\begin{tabular}{c l c }
	 		\hline
	 		Parameter   &  \quad Description      &   \quad   Value    \\
	 		\hline
	 		$a_0$       &\quad Effective thermal capacity               &  \quad 2\\
	 		$b_0$       &\quad Thermal dilation coefficient              &  \quad 1 \\
	 		$c_0$       &\quad Constrained specific storage coefficient &  \quad 2 \\
	 		$\alpha$    &\quad Biot-Willis constant                     &  \quad  1 \\
	 		$\beta$     &\quad Thermal stress coefficient.               &  \quad  1\\
	 		$\bm{K}$         &\quad Permeability tensor                      &  \quad $2I$\\
	 		$\bm{\Theta}$     &\quad Effective thermal conductivity           &  \quad $1e-10I$\\
	 		$E$         &\quad Young's modulus                          &  \quad  1.25e5\\
	 		$\nu$       &\quad Poisson ratio                            &  \quad  0.25\\
	 		\hline
	 	\end{tabular}
	 \end{table}
	 \begin{table}[htbp]
	 	\vspace{-2.0em}
	 	\begin{center}
	 		\caption{Error and convergence rates of $u_h^n$, $p_h^n$, $T_h^n$($L$-type iterative schemes)}\label{table201}
	 		\resizebox{\textwidth}{12mm}{
	 			\begin{tabular}{ccccccccccccc}
	 				\hline
	 				$h$  & $\frac{\|e_u\|_{L^2(\Omega)}}{\|u\|_{L^2(\Omega)}}$  &  CR  &  $\frac{\|e_u\|_{H^1(\Omega)}}{\|u\|_{H^1(\Omega)}}$  &  CR & $\frac{\|e_p\|_{L^2(\Omega)}}{\|p\|_{L^2(\Omega)}}$ & CR  &  $\frac{\|e_p\|_{H^1(\Omega)}}{\|p\|_{H^1(\Omega)}}$  &  CR  &  $\frac{\|e_T\|_{L^2(\Omega)}}{\|T\|_{L^2(\Omega)}}$  &  CR  & $\frac{\|e_T\|_{H^1(\Omega)}}{\|T\|_{H^1(\Omega)}}$ & CR \\ 
	 				\hline
	 				$1/4$   &0.0079&      &0.0541   &      &0.1617    &      &0.4081&      &0.1026&      &0.4328&    \\
	 				$1/8$   &9.3514e-04&  3.0829    &0.0139    &  1.9590    &0.0420   & 1.9436     &0.2036&  1.0031   &0.0250& 2.0363     &0.2117& 1.0312  \\
	 				$1/16$  &1.1144e-04&3.0689 &0.0035    &1.9821&0.0106    &1.9839&0.1020&0.9976&0.0063&2.0001&0.1049&1.0133\\
	 				$1/32$  &1.3594e-05&3.0353&8.8573e-04&1.9910&0.0027&1.9959&0.0510&0.9991&0.0016&1.9890&0.0524&1.0019\\
	 				\hline
	 		\end{tabular}}
	 	\end{center}
	 \end{table}
 \begin{figure}[H]
 	\subfigure[]{
 		\centering
 		\includegraphics[width=2.5in]{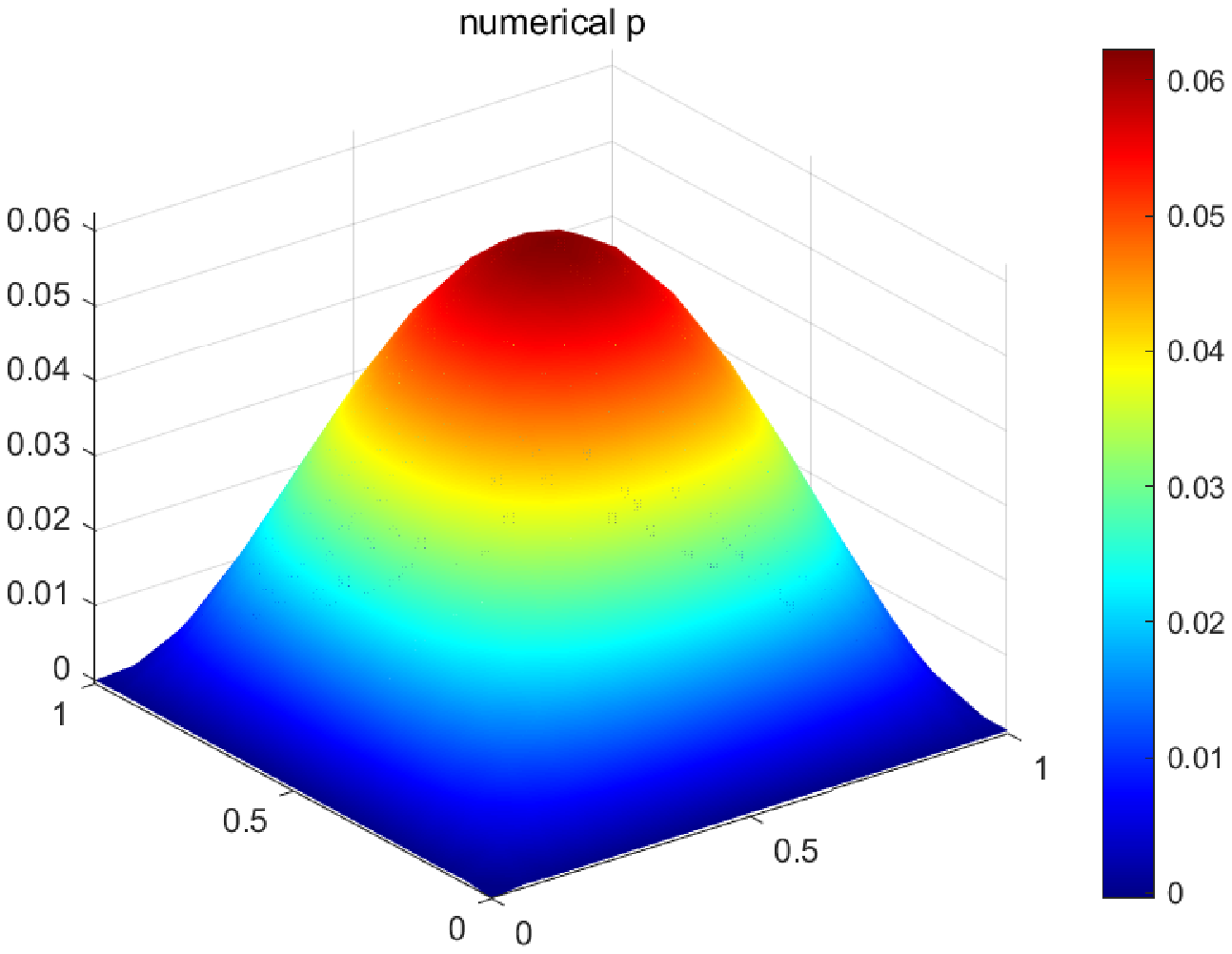}
 		\label{fig1111}
 	}
 	\subfigure[]{
 		\centering
 		\includegraphics[width=2.5in]{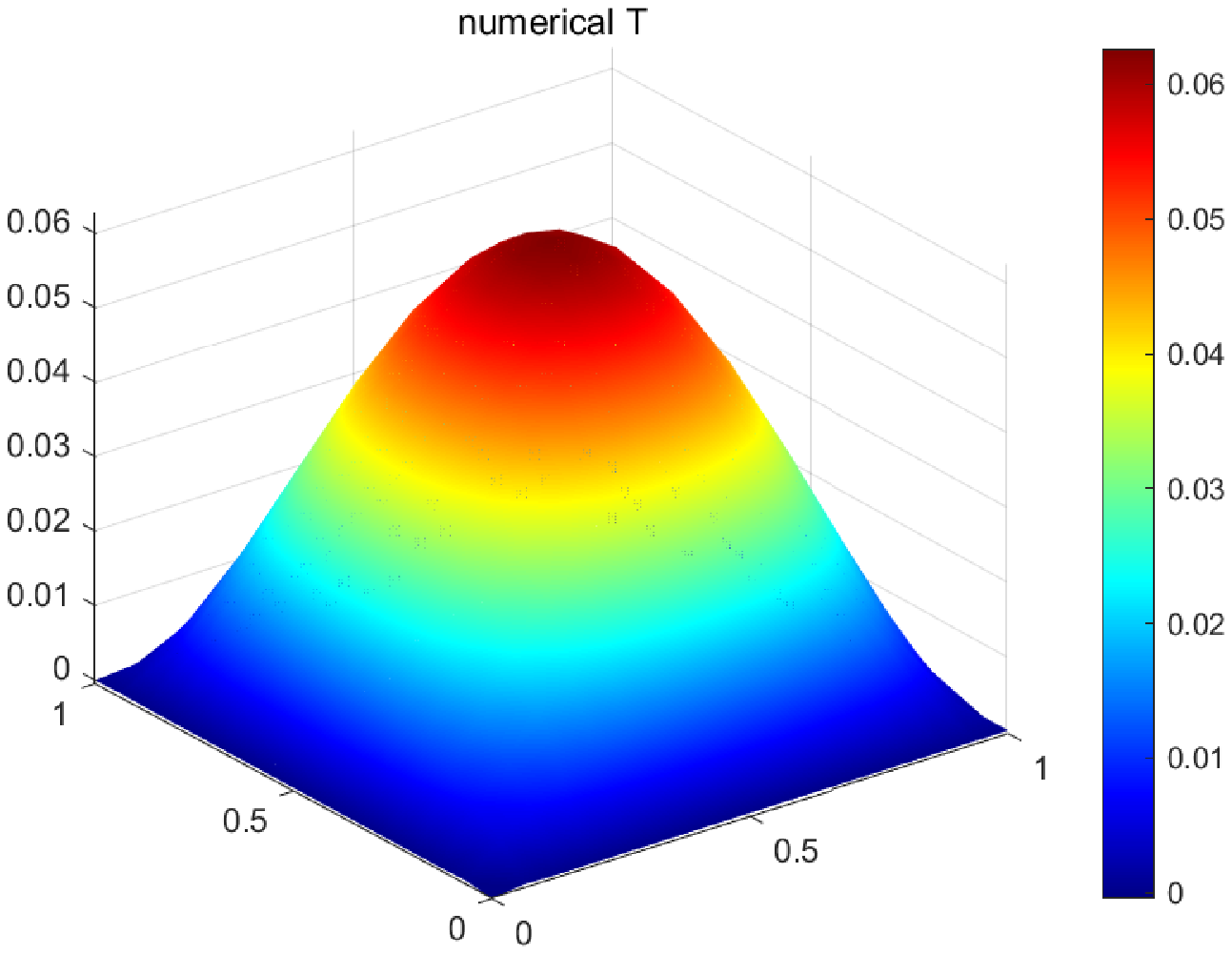}
 		\label{fig1112}
 	}
 	\caption{(a) and (b) are surface plot of the pressure $p_h^n$ and temperature $T_h^n$ at the terminal time $\tau$ respectively.}
 \end{figure}
 
 \begin{figure}[H]
 	\centering
 	\includegraphics[height=5cm,width=7cm]{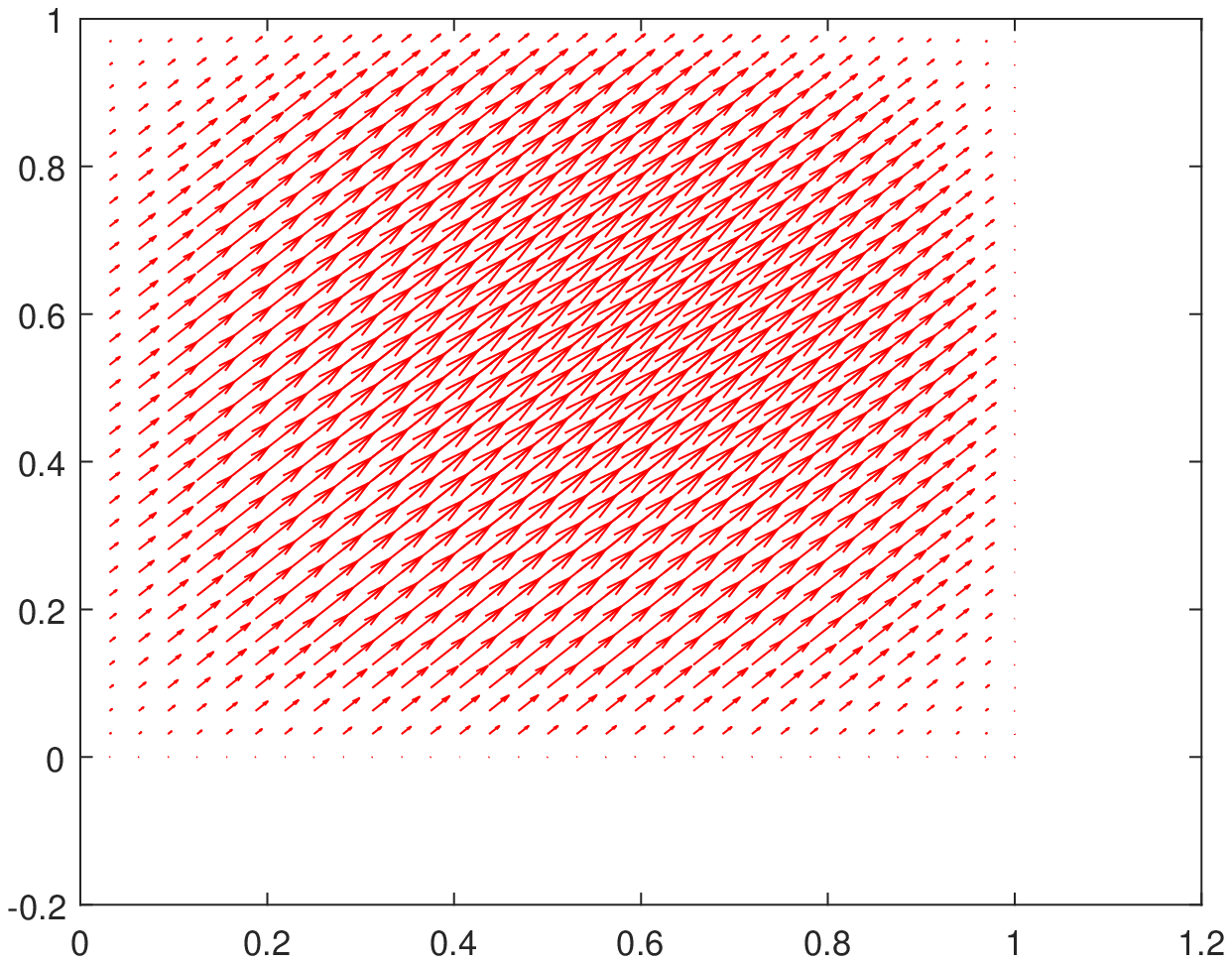}
 	\caption{Arrow plot of the computed displacement $\textbf{u}_h^n$.}
 	\label{fig1113}
 \end{figure}
 \vspace{-0.8cm}
 \begin{figure}[H]
 	\centering
 	\subfigure[]{
 		\centering
 		\includegraphics[width=2.5in]{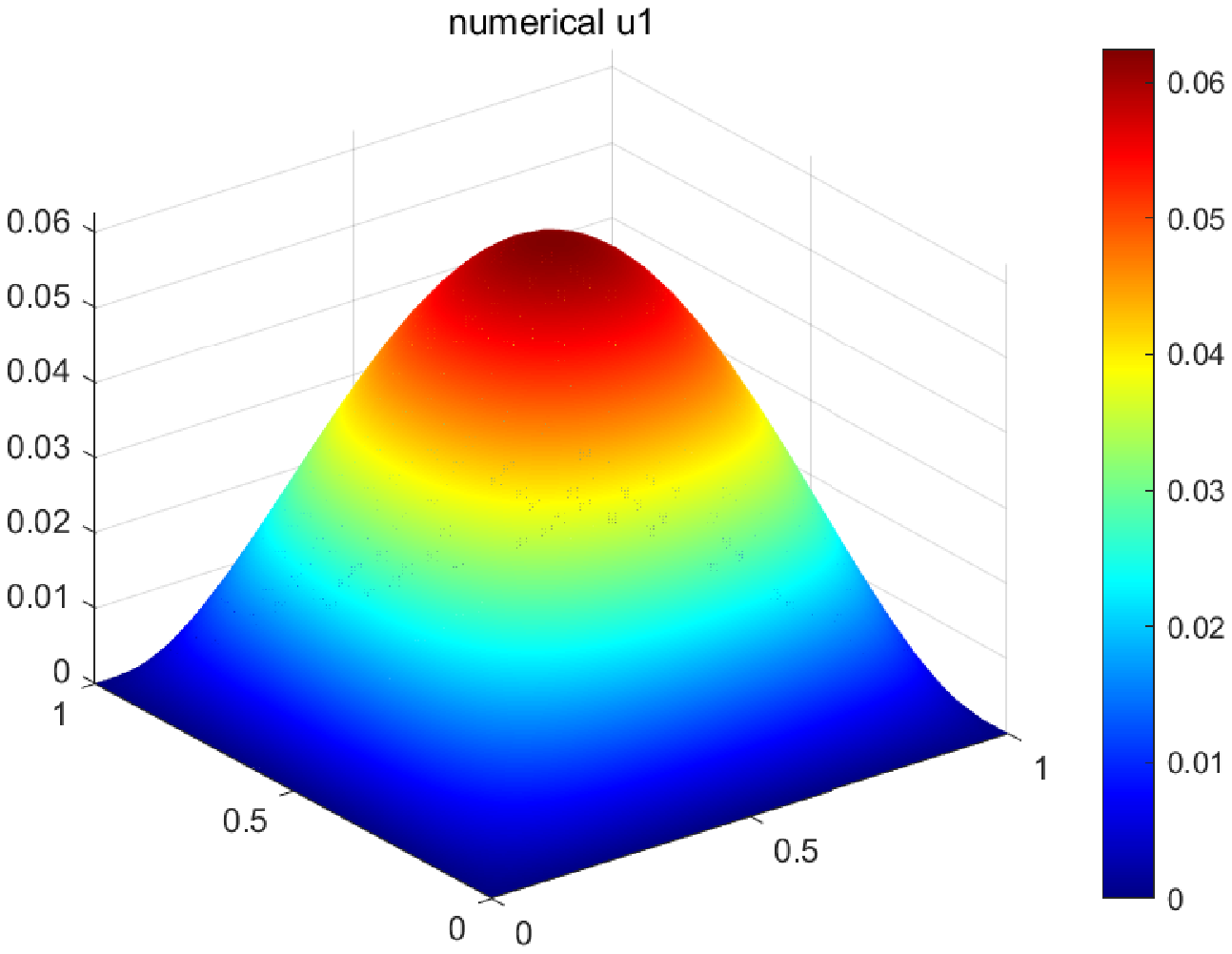}
 		\label{fig1114}}%
 	\subfigure[]{
 		\centering
 		\includegraphics[width=2.5in]{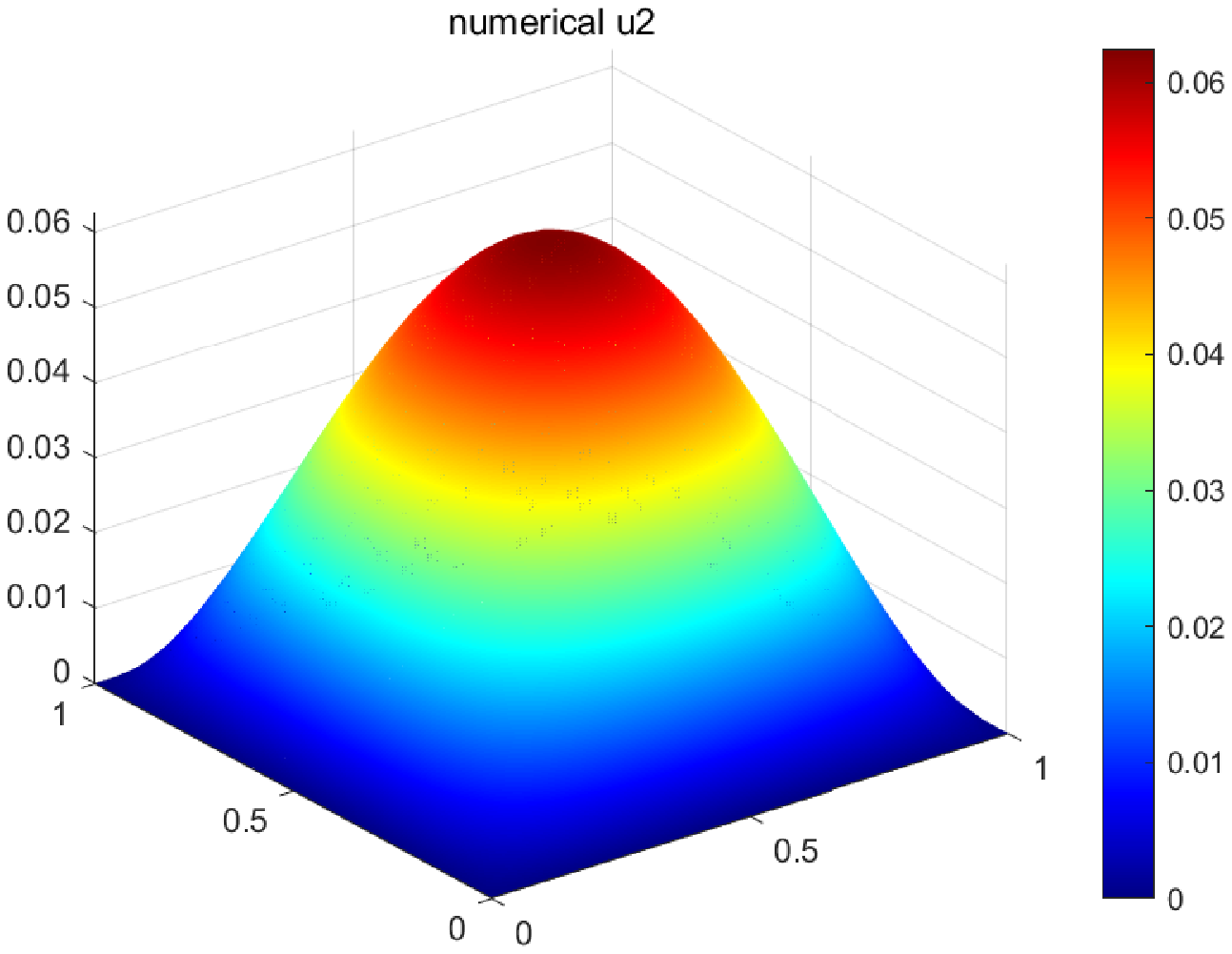}
 		\label{fig1117}}%
 	\centering
 	\caption{ (a) and (b) are Surface plot of $u_{1h}^{n}$ and $u_{2h}^{n}$ at the terminal time $\tau$ respectively.}
 \end{figure}

	  \begin{table}[htbp]
	  	\vspace{-2.0em}
	  	\begin{center}
	  		\caption{Error and convergence rates of $u_h^n$, $p_h^n$, $T_h^n$( There is no stabilizing term)}\label{table203}
	  		\resizebox{\textwidth}{12mm}{
	  			\begin{tabular}{ccccccccccccc}
	  				\hline
	  				$h$  & $\frac{\|e_u\|_{L^2(\Omega)}}{\|u\|_{L^2(\Omega)}}$  &  CR  &  $\frac{\|e_u\|_{H^1(\Omega)}}{\|u\|_{H^1(\Omega)}}$  &  CR & $\frac{\|e_p\|_{L^2(\Omega)}}{\|p\|_{L^2(\Omega)}}$ & CR  &  $\frac{\|e_p\|_{H^1(\Omega)}}{\|p\|_{H^1(\Omega)}}$  &  CR  &  $\frac{\|e_T\|_{L^2(\Omega)}}{\|T\|_{L^2(\Omega)}}$  &  CR  & $\frac{\|e_T\|_{H^1(\Omega)}}{\|T\|_{H^1(\Omega)}}$ & CR \\ 
	  				\hline
	  				$1/4$   &0.0079&      &0.0541   &      &0.2248    &      &0.4157&      &-&      &-&    \\
	  				$1/8$   &9.3504e-04&  3.0831    &0.0139    &  1.9590    &0.1117   & 1.0094     &0.2181&  0.9305   &-& -    &-& -\\
	  				$1/16$  &1.1146e-04&3.0685 &0.0035    &1.9821&0.0850    &0.3934&0.1293&0.7543&-&-&-&-\\
	  				$1/32$  &1.4358e-05&2.9566&8.8574e-04&1.9910&0.0784&0.1168&0.0945&0.4518&-&-&-&-\\
	  				\hline
	  		\end{tabular}}
	  	\end{center}
	  \end{table}
  \begin{figure}[H]
  	\subfigure[]{
  		\centering
  		\includegraphics[width=2.5in]{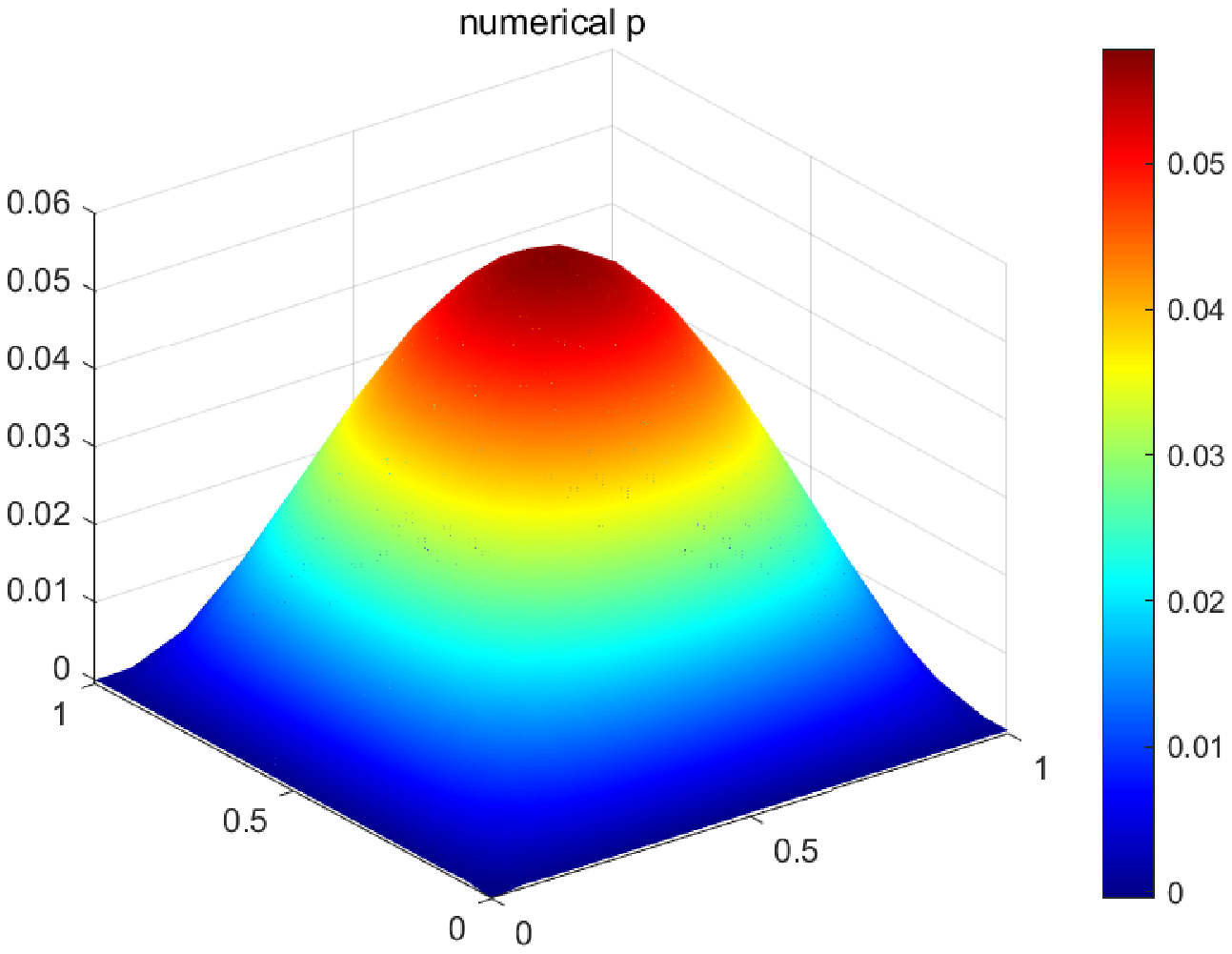}
  		\label{fig11118}
  	}
  	\subfigure[]{
  		\centering
  		\includegraphics[width=2.5in]{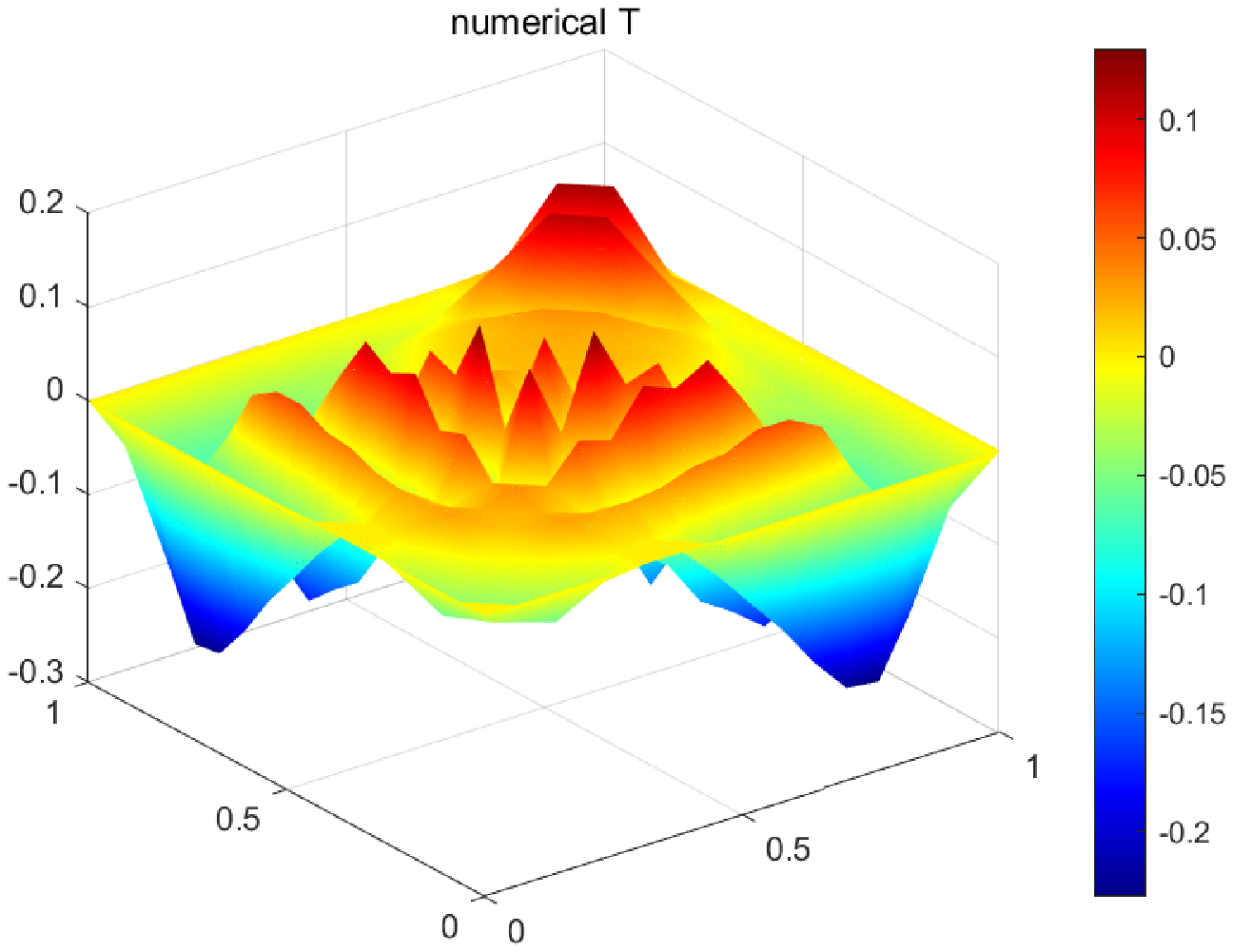}
  		\label{fig1119}
  	}
  	\caption{(a) and (b) are surface plot of the pressure $p_h^n$ and temperature $T_h^n$ at the terminal time $\tau$ respectively( There is no stabilizing term)}.
  \end{figure}
  
  \begin{figure}[H]
  	\centering
  	\includegraphics[height=5cm,width=7cm]{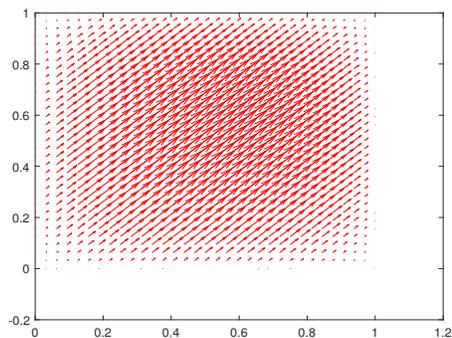}
  	\caption{Arrow plot of the computed displacement $\textbf{u}_h^n$( There is no stabilizing term)}.
  	\label{fig11110}
  \end{figure}
  \vspace{-0.8cm}
  \begin{figure}[H]
  	\centering
  	\subfigure[]{
  		\centering
  		\includegraphics[width=2.5in]{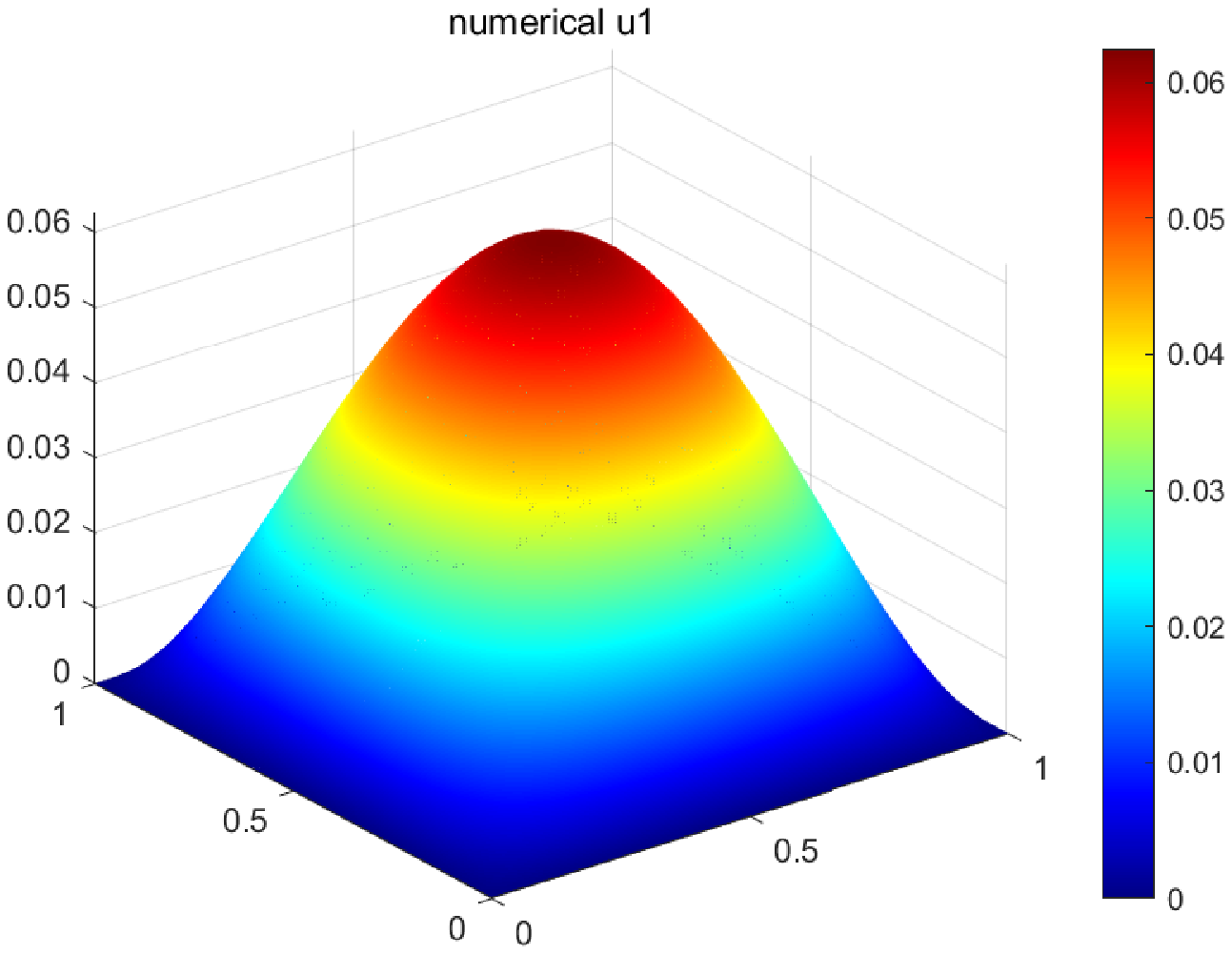}
  		\label{fig11111}}%
  	\subfigure[]{
  		\centering
  		\includegraphics[width=2.5in]{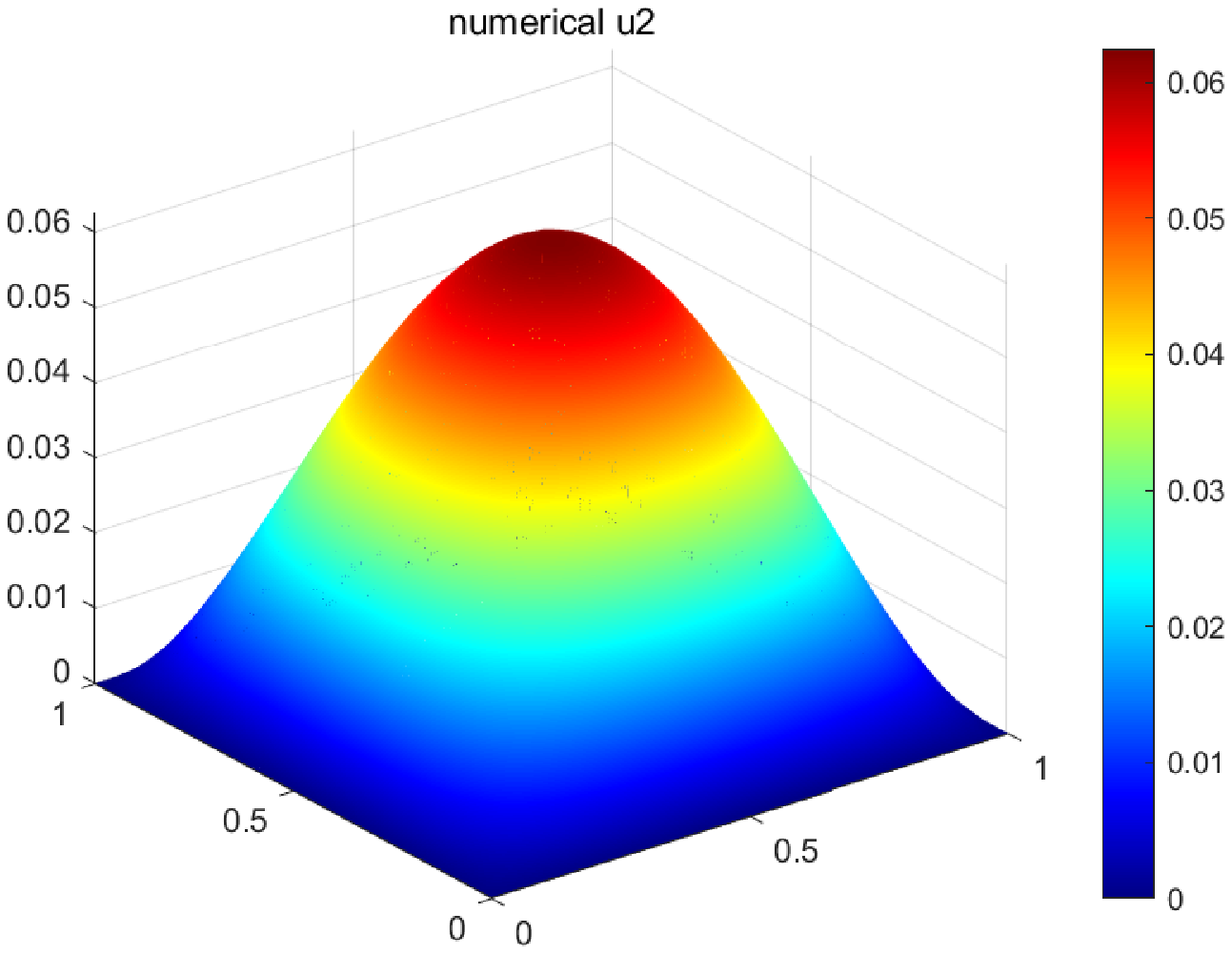}
  		\label{fig11112}}%
  	\centering
  	\caption{ (a) and (b) are Surface plot of $u_{1h}^{n}$ and $u_{2h}^{n}$ at the terminal time $\tau$ respectively( There is no stabilizing term)}.
  	
  \end{figure}
Table \ref{table203} displays the $L^{\infty}(0,T;L^{2}(\Omega))$-norm error and $L^{\infty}(0,T;H^{1}(\Omega))$-norm error of $u$, $p$, $T$ and the convergence order with respect to $h$ at terminal time $\tau=1$. From the above table, it can be seen that $T$ does not converge and $p$ has no optimal convergence order. Figure \ref{fig11118}, Figure \ref{fig1119}, Figure \ref{fig11111} and Figure \ref{fig11112} show, respectively, the surface plot of $p^{n}_{h}$, $T^{n}_{h}$, $u^{n}  _{1h}$ and $u^{n}_{2h}$ at the terminal time $\tau$ and Figure \ref{fig11110} shows arrow plot of  $\textbf{u}_h^n$. 
\section{Conclusion}
In this paper, we study a quasi-static nonlinear thermo-poroelasticity model with a nonlinear convective transport term in the energy equation . This makes analysis challenging. The main contributions are as follows: in order to clearly reveal the multi-physical process and overcome the "locking phenomenon" in the calculation, we introduce three new variables to reformulate the original model; in order to obtain the well-posedness of the nonlinear model, the Newton's solution procedure is introduced based on linearizing the heat flux term, which is shown to be well-defined, and which converges to the weak solution of the nonlinear problem in adequate norms; we propose a multiphysics finite element method with Newton's iterative algorithm, which is equivalent to a stabilized method, can effectively overcome the numerical oscillation caused by the nonlinear thermal convection term, the error estimate of the method is given. At the same time, it is proved that the method has the optimal convergence order. Finally, some numerical examples are given to verify the theoretical results. No "locking phenomenon" occurs in our numerical method, and the numerical oscillation caused by nonlinear convection term is overcome. This proves that both our method and the numerical method have built-in mechanisms to prevent "locking phenomenon".


\end{document}